\documentclass[11pt]{article}
\usepackage[a4paper, total={17cm,25cm}]{geometry}

\usepackage{amsmath,amssymb,amsfonts}
\usepackage{color, graphics}
\usepackage{tikz,tkz-euclide}
\usetikzlibrary{arrows,arrows.meta,patterns,calc,math,bending,cd}
\usepackage{float}

\usepackage{pst-all}
\usepackage{pst-solides3d}

\usepackage[new]{old-arrows}

\usepackage{mathrsfs}

\usepackage[alphabetic,y2k,initials]{amsrefs}

\usepackage[normalem]{ulem}
\usepackage{hyperref}

\def\mod {\mathrm{mod}}

\def\const{\mathrm{const}}

\def\arctan{\mathrm{arctan}}

\def\Id{\mathrm{Id}}

\numberwithin{equation}{section}

\newtheorem{theorem}{Theorem}[section]
\newtheorem{proposition}[theorem]{Proposition}
\newtheorem{corollary}[theorem]{Corollary}
\newtheorem{lemma}[theorem]{Lemma}
\newtheorem{remark}[theorem]{Remark}
\newtheorem{example}[theorem]{Example}
\newtheorem{definition}[theorem]{Definition}

\newenvironment{proof}[1][Proof]{\noindent\textit{#1.} }{\hfill$\Box$
	
\medskip}

 \title{Finite Groups of Random Walks in the Quarter Plane and Periodic $4$-bar Links}

\usepackage{authblk}
\author[1,3]{Vladimir Dragovi\'c}
\author[2,3]{Milena Radnovi\'c}
\affil[1]{\textsc{The University of Texas at Dallas, Department of Mathematical Sciences}}
\affil[2]{\textsc{The University of Sydney, School of Mathematics and Statistics}}
\affil[3]{\textsc{Mathematical Institute SANU, Belgrade}}
\affil[ ]{\texttt{vladimir.dragovic@utdallas.edu, milena.radnovic@sydney.edu.au}}

\date{}

\begin{document}

\maketitle

\begin{abstract}
We solve two long standing open problems, one from probability theory formulated by Malyshev in 1970 and another one from a crossroad of geometry and dynamics, going back to Darboux in 1879.
The Malyshev problem is of finding effective, explicit necessary and sufficient conditions in the closed form to characterize  all random walks in the quarter plane with a finite group of the random walk of order $2n$, for all $n\ge 2$, in the generic case where the underlining biquadratic is an elliptic curve.
Until now, the results were known only for $n=2, 3, 4$ and were obtained using ad-hoc methods developed separately for each of the three cases.
We provide a method that solves the problem for all $n$ and in a unified way.
In this paper, explicit examples of random walks with the groups of orders higher than $10$ are presented for the first time, including orders $12$, $14$, and $16$ and the same method applies to any higher order as well.
We also consider situations with singular biquadratics in a systematic manner.
Further, we establish a new two-way relationship between {\it diagonal} random walks in the quarter plane and $4$-bar links.
We describe all $n$-periodic Darboux transformations for $4$-bar link problems for all $n\ge 2$, thus completely solving the Darboux problem, in which after $n$ iterations, a polygonal configuration maps to a congruent one of the same orientation, that he solved for $n=2$, and which was very  recently extended to $n=3$.
We also study  \emph{$k$-semi-periodicity} as a natural type of periodicity of the Darboux transformations, where  after $k$ iterations of the Darboux transformation, a polygonal configuration maps to a congruent one, but of opposite orientation. By introducing a new object, \emph{the secondary $(2,2)$ correspondence} and the related secondary cubic of the centrally-symmetric biquadratics, we provide necessary and sufficient conditions for $k$-semi-periodicity for $4$-bar links for all $k\ge 2$ in an explicit closed form, while the case $k=2$ was solved very recently.

\vskip 0.5cm

\end{abstract}

\tableofcontents

\section{Introduction}

We solve the long standing open problem formulated by Malyshev in \cite{Mal} in 1970, of finding effective, explicit necessary and sufficient conditions in the closed form to characterize  all random walks in the quarter plane with a finite group of the random walk of order $2n$, for all $n\ge 2$, in the generic case where the underlining biquadratic, defining the kernel of the random walk, is an elliptic curve.

Each biquadratic curve has two natural involutions $h$ and $v$, see Figure \ref{fig:switches}.
In \cite{Mal}, Malyshev defined  the group of random walk as the group generated by $h$ and $v$ for the kernels of random walks. We provide all necessary formal definitions in Section \ref{sec:grw}.
In particular, see  \eqref{eq:H}.

\begin{figure}[h]
	\centering
\begin{tikzpicture}[scale=0.8]

\coordinate(xy)at(-1.99, -0.92565);
\coordinate(x'y)at(3.3488, -0.92565);
\coordinate(xy')at(-1.99,  1.66686);

\draw[fill=black,black](xy) circle (0.07) node[below left]{$(x,y)$};

\draw[fill=black,black](xy') circle (0.07) node[above left]{$(x,y')$};

\draw[fill=black,black](x'y) circle (0.07) node[below right]{$(x',y)$};

\draw[thick,Stealth-Stealth](-1.91, -0.92565)--(3.2688, -0.92565)node[midway,above]{$h$};

\draw[thick,Stealth-Stealth](-1.99, -0.84565)--(-1.99,  1.58686)node[midway,right]{$v$};

\node at (1.7,1.7){$\mathcal{C}$};

\draw [very thick] plot [smooth cycle] coordinates {(-3.59,0.179477)(-3.49,0.0230174)(-3.39,-0.0646813)(-3.29,-0.136015)(-3.19,-0.199611)(-3.09,-0.258959)(-2.99,-0.31593)(-2.89,-0.371707)(-2.79,-0.427131)(-2.69,-0.482863)(-2.59,-0.539461)(-2.49,-0.597435)(-2.39,-0.657276)(-2.29,-0.719479)(-2.19,-0.78456)(-2.09,-0.853078)(-1.99,-0.92565)(-1.89,-1.00297)(-1.79,-1.08581)(-1.69,-1.1751)(-1.59,-1.27186)(-1.49,-1.37733)(-1.39,-1.49294)(-1.29,-1.62035)(-1.19,-1.76149)(-1.09,-1.91861)(-0.99,-2.09422)(-0.89,-2.29103)(-0.79,-2.51182)(-0.69,-2.75896)(-0.59,-3.03375)(-0.49,-3.33514)(-0.39,-3.65801)(-0.29,-3.99103)(-0.19,-4.31514)(-0.09,-4.60427)(0.01,-4.82978)(0.11,-4.96816)(0.21,-5.00852)(0.31,-4.95527)(0.41,-4.82486)(0.51,-4.63923)(0.61,-4.41991)(0.71,-4.18458)(0.81,-3.94608)(0.91,-3.71284)(1.01,-3.48985)(1.11,-3.27972)(1.21,-3.08345)(1.31,-2.90105)(1.41,-2.73198)(1.51,-2.57541)(1.61,-2.43033)(1.71,-2.29575)(1.81,-2.17066)(1.91,-2.05414)(2.01,-1.94531)(2.11,-1.84339)(2.21,-1.74765)(2.31,-1.65743)(2.41,-1.57215)(2.51,-1.49126)(2.61,-1.41427)(2.71,-1.34073)(2.81,-1.27021)(2.91,-1.20231)(3.01,-1.13665)(3.11,-1.07284)(3.21,-1.01051)(3.31,-0.949235)(3.41,-0.888565)(3.51,-0.827961)(3.61,-0.766744)(3.71,-0.703977)(3.81,-0.638231)(3.91,-0.56698)(4.01,-0.484609)(4.11,-0.36977)
(4.14,-0.194305)(4.04,-0.0572221)(3.94,0.0179179)(3.84,0.0780297)(3.74,0.130756)(3.64,0.179159)(3.54,0.224848)(3.44,0.268809)(3.34,0.311714)(3.24,0.35406)(3.14,0.396238)(3.04,0.438577)(2.94,0.481366)(2.84,0.524874)(2.74,0.569358)(2.64,0.615073)(2.54,0.662282)(2.44,0.711261)(2.34,0.762303)(2.24,0.815728)(2.14,0.87189)(2.04,0.931182)(1.94,0.994046)(1.84,1.06098)(1.74,1.13256)(1.64,1.20943)(1.54,1.29233)(1.44,1.38213)(1.34,1.47979)(1.24,1.58642)(1.14,1.70326)(1.04,1.83168)(0.94,1.9731)(0.84,2.12895)(0.74,2.30042)(0.64,2.48819)(0.54,2.69182)(0.44,2.90902)(0.34,3.13459)(0.24,3.35947)(0.14,3.57037)(0.04,3.75091)(-0.06,3.88462)(-0.16,3.9593)(-0.26,3.97044)(-0.36,3.92203)(-0.46,3.82441)(-0.56,3.69068)(-0.66,3.53373)(-0.76,3.36436)(-0.86,3.1907)(-0.96,3.01841)(-1.06,2.85111)(-1.16,2.69089)(-1.26,2.53882)(-1.36,2.39529)(-1.46,2.2602)(-1.56,2.13322)(-1.66,2.01387)(-1.76,1.90158)(-1.86,1.79575)(-1.96,1.69579)(-2.06,1.60113)(-2.16,1.5112)(-2.26,1.42548)(-2.36,1.34346)(-2.46,1.26464)(-2.56,1.18854)(-2.66,1.11466)(-2.76,1.04251)(-2.86,0.971509)(-2.96,0.901025)(-3.06,0.830254)(-3.16,0.758114)(-3.26,0.682971)(-3.36,0.601966)(-3.46,0.508774)(-3.56,0.379442)
	};
\end{tikzpicture}
\caption{Two involutions on a biquadratic curve $\mathcal{C}$: the horizontal switch $h$ mapping the points $(x,y)$ and $(x',y)$ to each other and the vertical switch $v$ mapping $(x,y)$ and $(x,y')$ to each other.}
\label{fig:switches}
\end{figure}
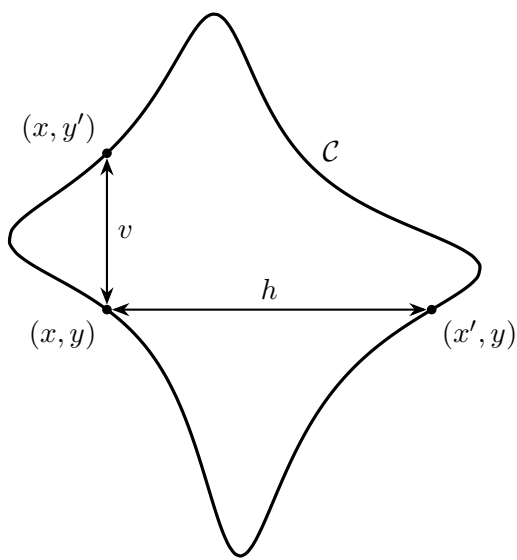

In \cite{Mal}, Malyshev qualified his problem of describing all random walks with a finite group, as ``sufficiently difficult". He found some particular cases of random walks with groups of order  four and six. He proved that among all random walks, those with a finite group of random walk form a set of Bair first category (i.e.~a meager set).

Malyshev strongly motivated his problem by showing that if the group of random walk is finite, then there is an algebraic construction of invariant measures.

The significance of the Malyshev problem of finiteness of groups of random walks also comes from queuing theory and analytic combinatorics.
In queuing theory, that kind of ideas and techniques were applied to the study of generation functions of the equilibrium probabilities in the theory of double queues, see \cite{FH} and Example \ref{ex:que}.
Similarly, the finiteness of the corresponding group of the walk is a sufficient condition for the generating functions of enumerating lattice walks to be algebraic in the theory of walks with small steps in the quarter plane, see e.g. \cites{BKR, BMM, FayRas1, KuRas, Ras} and Section \ref{sec:comb}.

Until now, the results that describe all random walks with groups of a given order $2n$ were known for $n=2, 3, 4$.
They were obtained using ad-hoc methods developed separately for each of the three cases and using specific properties of transition probabilities. The results for $n=2$ and $n=3$ were presented in 1999 in \cite{Randomwalks0}.
The case $n=4$ was solved in 2015 in \cite{FayIas}, where the general problem of an arbitrary order $2n$ of the group of random walk was described as ``deep".  The solution for $n=4$ from \cite{FayIas}, appeared also in the new, 2017 edition \cite{RandomWalks} of \cite{Randomwalks0}.

In this paper, explicit examples of random walks with the groups of orders higher than $10$ are presented for the first time. Here, new examples are given for random walks with the groups of orders $10$, $12$, $14$, and $16$, and we emphasize that the same method allows us  to construct random walks with the group of any higher order as well.

We also consider situations with singular biquadratics  in a systematic manner, see Section \ref{sec:singular}.

Further, we establish a new two-way relationship between {\it diagonal} random walks in the quarter plane (see Definition \ref{def:diagonal}) and $4$-bar links, see Section \ref{sec:fbrw}.
This connection between $4$-bar links with probability is novel with respect to the existing connections of $n$-bar links with probability theory, such as \cite{Far}*{Section 1.11}, \cite{Far2}, \cite{FarKap}.
We describe all $n$-periodic Darboux transformations for $4$-bar links for all $n\ge 2$, thus solving the long standing open problem on the stick of geometry and dynamics, that goes back to Darboux in 1879, \cite{Dar}, where he formulated the problem  as finding conditions that after $n$ iterations, a polygonal configuration maps to a congruent one of the same orientation.
Darboux solved this problem for $n=2$ in \cite{Dar}.

We also study \emph{$k$ semi-periodicity}, as a natural feature of the Darboux transformations, where  after $k$ iterations of the Darboux transformation, a polygonal configuration maps to a congruent one, but of opposite orientation. By introducing a new object, \emph{the secondary $(2,2)$ correspondence} and the related \emph{secondary cubic} of the centrally-symmetric biquadratics, we provide necessary and sufficient conditions for $k$-semi-periodicity for $4$-bar links for all $k\ge 2$ in an explicit closed form.

That secondary cubic of ours turns out to be isomorphic to the cubic curve constructed by Izmestiev in \cite{Izm}, see Remark \ref{rem:invariant}.
Using it in the smooth case, explicit conditions for periodicity of the Darboux transformation, \emph{irrespectively of orientation}, were derived in \cite{Izm}*{Theorem 4}.
Results obtained in \cite{Izm} include explicit conditions for $n$-periodicity in the orientation preserving sense of Darboux that we use here, for $n=2$ (already known to Darboux) and $n=3$, while conditions for $n=4$ and $n=6$ were also discussed.
See our Remarks  \ref{rem:invariant} and \ref{rem:4a}.

The theory of $(2,2)$ correspondences, especially of, generally speaking, non-symmetric ones,  is in the heart of this paper.
A $(2,2)$ correspondence defines a biquadratic curve in $\mathbb P^1\times \mathbb P^1$, which is isomorphic to a cubic curve in $\mathbb P^2$, which, in the smooth case, carries a natural group structure.
Though both a smooth biquadratic and the corresponding cubic are \emph{transcendentally} isomorphic to the same elliptic curve, the isomorphism between the two is apparently \emph{polynomial} in terms of the coefficients of the biquadratic.
This forms a unified framework for our solutions of the Malyshev and the Darboux problems,  although these two problems are seemingly quite contrasted to each other, as belonging to very distant fields of mathematics.
This theory has a long and illustrious history, starting with Euler in 1766, \cite{Euler1766}, see also \cites{Cayley1871, Fr}, as well as more modern accounts, like \cites{Clem, Sam, QRT1988, Tsu, DuistermaatBOOK}.
Symmetric $(2,2)$ correspondences played an important role in modern theory of integrable systems (see e.g.~\cites{Bax71b, Bax72a, Bax72b, Bax82, Krich, Dra1992, Dra1993, V1992, Dra2014}) as well as in the study of the Poncelet theorem, see \cites{GrifHar1978, FlattoBOOK, DR2011knjiga, DR2025rcd}, and references therein.
The necessary background material is presented in Sections \ref{sec:grw} and \ref{sec:P1P1P2} to make the paper reasonably self-contained and assessable to readers from various fields.

\subsection{The group of random walk in the quarter plane and the Malyshev problem}\label{sec:grw}

Following \cite{RandomWalks}, we consider \emph{maximally space homogeneous random walks} as a class of discrete time homogeneous  Markov chains, with the state space being the quarter plane
$\mathbb Z^2_+=\{(i, j)\mid i, j\in \mathbb N_0\}$.
In the interior of $\mathbb Z^2_+$, the jumps are of the size
one. The generators of the process in this region are
$\{p_{ij}\mid -1\le i,j\le 1\}$, where $p_{ij}$ is the transition probability for the jump from $(r, s)$ to $(r+i, s+j)$, for $rs>0$. Thus
$$
p_{ij}\ge0,\quad \sum_{i,j=-1}^1p_{ij}=1.
$$
The situation is different for the coordinate axes and the origin, where there are no bounds on the upward jumps and the downward jumps for both axes are bonded by one, see Figure \ref{fig:rw}.
\begin{figure}[h]
	\centering
		\begin{tikzpicture}[scale=1]
	\draw[thick,-Stealth] (-1,0) -- (8.5,0);
	\draw[thick,-Stealth] (0,-1) --(0,5.5);
	
	\foreach \x in {1,...,8}
	\draw[gray!50,dashed](\x,0)--(\x,5.5);
	
	\foreach \y in {1,...,5}
	\draw[gray!50,dashed](0,\y)--(8.5,\y);

	\foreach \x in {0,...,8}
	\foreach \y in {0,...,5}
	\draw[fill=gray!50,gray!50] (\x,\y) circle [radius=0.07];
	
	\draw[very thick,-Stealth](7,4)--(8,4)node[right]{{\small$p_{10}$}};
	\draw[very thick,-Stealth](7,4)--(6,4)node[left]{{\small$p_{-1,0}$}};
	\draw[very thick,-Stealth](7,4)--(7,3)node[below]{{\small$p_{0,-1}$}};
	\draw[very thick,-Stealth](7,4)--(7,5)node[above]{{\small$p_{01}$}};
	\draw[very thick,-Stealth](7,4)--(6,5)node[above]{{\small$p_{-1,1}$}};
	\draw[very thick,-Stealth](7,4)--(6,3)node[below]{{\small$p_{-1,-1}$}};
	\draw[very thick,-Stealth](7,4)--(8,5)node[above]{{\small$p_{11}$}};
	\draw[very thick,-Stealth](7,4)--(8,3)node[below]{{\small$p_{1,-1}$}};
	
	\draw[white,fill=white,opacity=0.95] (7,4) circle [radius=0.3];
	\node at(7,4){{\small$p_{00}$}};
	
	\draw[very thick,-Stealth](4,0)--(3,0);
	\draw[very thick,-Stealth](4,0)--(3,1);
	\draw[very thick,-Stealth](4,0)--(3,2);
	\draw[very thick,-Stealth](4,0)--(3,4);
	\draw[very thick,-Stealth](4,0)--(3,3);
	\draw[very thick,-Stealth](4,0)--(4,4);
	\draw[very thick,-Stealth](4,0)--(7,2);
	\draw[very thick,-Stealth](4,0)--(8,0);
	\draw[fill=black] (4,0) circle [radius=0.07];

	\draw[very thick,-Stealth](0,3)--(0,2);
	\draw[very thick,-Stealth](0,3)--(1,2);
	\draw[very thick,-Stealth](0,3)--(2,2);
	\draw[very thick,-Stealth](0,3)--(3,2);
	\draw[very thick,-Stealth](0,3)--(4,5);
	\draw[very thick,-Stealth](0,3)--(0,5);
	\draw[fill=black] (0,3) circle [radius=0.07];

	\draw[very thick,-Stealth,](0,0)--(2,2);
	\draw[very thick,-Stealth,](0,0)--(3,1);
	\draw[very thick,-Stealth,](0,0)--(3,0);
	\draw[very thick,-Stealth,](0,0)--(0,1);
		\draw[fill=black] (0,0) circle [radius=0.07];
	
\end{tikzpicture}
\caption{Maximally space homogeneous random walks: from a point in the interior of the quarter plane, jumps are possible only to the neighbouring points; from the points in the interior of the coordinate axes, the jumps downward the axes are bounded by one.}\label{fig:rw}
\end{figure}

We define the matrix  of a random walk $P$ by
\begin{equation}\label{eq:P}
	P=\begin{pmatrix}
		p_{11} & p_{10} &  p_{1,-1}\\
		 p_{01} &  p_{00}-1 &  p_{0,-1}\\
		 p_{-1,1} &  p_{-1,0} &  p_{-1,-1}
	\end{pmatrix}.
\end{equation}
Similarly, the transition matrix $\hat P$ is a square stochastic matrix that describes the probabilities of moving from one state to another, where each entry $\hat P_{\hat \ell, \hat k}$ represents a probability of moving from state  $\hat \ell$ to state $\hat k$. Here, $\hat \ell$ and $\hat k$ are points of the integer lattice in the quarter plane.
Then, an invariant measure $\pi$ is a left eigenvector of $\hat P$ with the eigenvalue equal to $1$: $\sum_{\hat \ell \in \mathbb Z_+^2} \pi_{\hat \ell}\hat P_{\hat \ell\hat k}=\pi_{\hat k}$,  see e.g. \cite{RandomWalks}. If an invariant measure is a probability measure, them it is called {\it a stationary distribution}.

The fundamental functional relation for the invariant measure $\pi(x,y)$ takes the following form \cite{RandomWalks}:
\begin{equation}\label{eq:fundamental}
-Q(x,y)\pi(x,y)=q(x,y)\pi(x)+\tilde q(x,y)\tilde\pi(y)+\pi_{00}q_0(x, y),
\end{equation}
where
\begin{align*}
&\pi(x, y)=\sum_{i,j=1}^\infty\pi_{ij}x^{i-1}y^{j-1},
\quad
\pi(x)=\sum_{i=1}^\infty\pi_{i0}x^{i-1},
\quad
\tilde \pi(y)=\sum_{j=1}^\infty\pi_{0j}y^{j-1},
\\
&Q(x,y)=xy\left(\sum_{i,j=-1}^1 p_{ij}x^iy^j-1\right),
\quad
p_{ij}\ge0,\quad \sum_{i,j=-1}^1p_{ij}=1,
\\
&q(x,y)=x\left(\sum_{i\ge-1,j\ge0} p'_{ij}x^iy^j-1\right),
\quad
\tilde q(x,y)=y\left(\sum_{i\ge0,j\ge-1} p''_{ij}x^iy^j-1\right),
\\
&q_0(x, y)=\sum_{i\ge0,j\ge0} p^0_{ij}x^iy^j-1.
\end{align*}
Here, $p^0_{ij}$, $p_{ij}'$, and $p_{ij}''$ are the transition probabilities from $(r,s)$ to $(r+i,s+j)$, where, respectively, $(r,s)=(0,0)$ is the origin, $(r,s)=(r,0)$, $r>0$ is on the $x$-axis, and $(r,s)=(0,s)$, $s>0$ is on the $y$-axis.

The instance of \eqref{eq:fundamental} with $Q(x, y)=0$ is of a special interest. This leads to the consideration of a biquadratic curve $\mathcal C$, given by its affine equation in $\mathbb C^2$:
\begin{equation}\label{eq:biqrandom}
\mathcal C\ :\ Q_P(x,y)=xy\left(\sum_{i,j=-1}^1 p_{ij}x^iy^j-1\right)=0,
\quad\text{with}\quad
p_{ij}\ge0,
\quad
\sum_{i,j=-1}^1p_{ij}=1.
\end{equation}

As it will be discussed in Section \ref{sec:QRT}, there are the vertical and horizontal switches $v$ and $h$ defined on the curve $\mathcal C$.
\emph{The group $\mathcal H$ of random walk in the quarter plane} is defined as the group of automorphisms of the curve $\mathcal{C}$ generated by those two switches:
\begin{equation}\label{eq:H}
\mathcal H:=\langle h, v\rangle.
\end{equation}
Since $h$ and $v$ are involutions, the group $\mathcal H$ has a normal cyclic subgroup of index $2$, denoted
$\mathcal H_0:=\langle\delta\rangle$, where $\delta= v\circ h$ is a map on $\mathcal C$, that is going to be an instance of a QRT map from  \eqref{eq:QRT}.
Thus, the group of random walk $\mathcal H$ is of a finite order if and only if the map $\delta$ is of a finite order $n$.
In such a case, the order of $\mathcal H$ equals $2n$.

In the case when $(x, y)$ belongs to the curve $\mathcal C$, i.e.~for $Q(x, y)=0$, the fundamental equation \eqref{eq:fundamental} simplifies to
\begin{equation}\label{eq:fundsimp}
q(x,y)\pi(x)+\tilde q(x,y)\tilde\pi(y)+\pi_{00}q_0(x, y)=0.
\end{equation}
Thus the finiteness of the group of random walk $\mathcal H$ implies that the simplified fundamental equation \eqref{eq:fundsimp} allows to be solved in an elegant, algebraic procedure. The main aim of this paper is to describe random walks with the groups of a given finite order $2n$, for all $n\in \mathbb N$. We will accomplish this in the case of smooth biquadratic curves $Q(x, y)=0$ in Theorem \ref{th:cayley} and for nonsmooth biquadratics in Section \ref{sec:singular}. The proofs of these results rely on a preparatory material that we present in Section \ref{sec:P1P1P2}.

\section{Biquadratic curves in $\mathbb{P}^1\times\mathbb{P}^1$ and cubics in $\mathbb{P}^2$}\label{sec:P1P1P2}

The role of this section is to present the essential properties of the biquadratic curves, which underlie the problems considered in this paper, and relate them to cubics in the projective plane.

While the curves that appear in the analysis of the main questions addressed by this paper are over the field of complex numbers, we note that the majority of results and considerations in this section hold for biquadratics and cubics over any field.

\subsection{Biquadratic polynomials}

This section outlines the most important algebraic properties of biquadratic polynomials in two variables and quartic polynomials in one variable.

\begin{definition}
A \emph{biquadratic polynomial} is a polynomial in two variables, denoted here $x$ and $y$, of degree two in each of these variables:
\begin{equation}\label{eq:biquad}
Q(x,y)
=
a_{22}x^2y^2+a_{21}x^2y+a_{20}x^2+a_{12}xy^2+a_{11}xy+a_{10}x+a_{02}y^2+a_{01}y+a_{00}.
\end{equation}
Here, the coefficients $a_{ij}$ are fixed complex numbers.
We say that $Q(x,y)$ is \emph{symmetric} if $Q(x,y)=Q(y,x)$.
\end{definition}

Any biquadratic polynomial $Q$ can also be seen
as a quadratic polynomial $y$ with the coefficients being polynomial in $x$, or as as a quadratic polynomial $x$ with the coefficients being polynomial in $y$.
Using the notation from \cite{RandomWalks}, we write:
\begin{equation}\label{eq:biquad1}
Q(x,y)=a(x)y^2+b(x)y+c(x)=\tilde{a}(y)x^2+\tilde{b}(y)x+\tilde{c}(y),
\end{equation}
with
\begin{gather*}
a(x)=a_{22}x^2+a_{12}x+a_{02},
\quad
b(x)=a_{21}x^2+a_{11}x+a_{01},
\quad
c(x)=a_{20}x^2+a_{10}x+a_{00};
\\	
\tilde{a}(y)=a_{22}y^2+a_{21}y+a_{20},
\quad
\tilde{b}(y)=a_{12}y^2+a_{11}y+a_{10},
\quad
\tilde{c}(y)=a_{02}y^2+a_{01}y+a_{00}.	
\end{gather*}

Denote by $\mathcal D_{Q_x}(y)$ the discriminant of $Q$, understood as a quadratic polynomial in $x$ and by $\mathcal D_{Q_y}(x)$ the discriminants of $Q$, understood as a  quadratic polynomial in $y$:
\begin{equation}\label{eq:discxy}
\mathcal D_{Q_x}(y)=\tilde{b}(y)^2-4\tilde{a}(y)\tilde{c}(y),
\quad
\mathcal D_{Q_y}(x)=b(x)^2-4a(x)c(x).
\end{equation}
In general, the discriminants $\mathcal D_{Q_x}(y)$ and $\mathcal D_{Q_y}(x)$ are polynomials of degree four:
\begin{align*}
\mathcal D_{Q_x}(y)=&\ y^4 \left(a_{12}^2-4 a_{02} a_{22}\right)
+y^3 (2 a_{11} a_{12}-4 a_{01} a_{22}-4 a_{02} a_{21})
\\&
+y^2 \left(a_{11}^2-4 a_{00} a_{22}-4 a_{01} a_{21}-4 a_{02} a_{20}+2 a_{10} a_{12}\right)
+y (2 a_{10} a_{11}-4 a_{00} a_{21}-4 a_{01} a_{20})
\\&
-4 a_{00} a_{20}+a_{10}^2,
\\
\mathcal D_{Q_y}(x)=&\ x^4 \left(a_{21}^2-4 a_{20} a_{22}\right)
+x^3 (2 a_{11} a_{21}-4 a_{10} a_{22}-4 a_{20} a_{12})
\\&
+x^2 \left(a_{11}^2-4 a_{00} a_{22}-4 a_{10} a_{12}-4 a_{20} a_{02}+2 a_{01} a_{21}\right)
+x (2 a_{01} a_{11}-4 a_{00} a_{12}-4 a_{10} a_{02})
\\&
-4 a_{00} a_{02}+a_{01}^2,
\end{align*}
and they are of a smaller degree if
$a_{12}^2-4 a_{02} a_{22}=0$ or $a_{21}^2-4 a_{20} a_{22}=0$, respectively.

Note that in this paper, we will consider a general situation, when a biquadratic polynomial is not necessarily symmetric.
In that case, the discriminant polynomials $\mathcal D_{Q_x}(y)$ and $\mathcal D_{Q_y}(x)$, in general, have distinct coefficients corresponding to the same degree of a variable.
Thus, it is of interest to examine the projective invariants of those two polynomials.

\begin{definition}
\emph{The Eisenstein invariants} of the quartic polynomial:
\begin{equation}\label{eq:quartic}
P(x)= a_4x^4+4a_3x^3+6a_2x^2+4a_1x+a_0,
\end{equation}
are
$$
D=a_0a_4+3a_2^2-4a_1a_3
\quad \text{and}\quad
E=a_0a_3^2+a_1^2a_4-a_0a_2a_4-2a_1a_2a_3+a_2^3.
$$
\end{definition}

Recall that \emph{the discrimanant} of a polynomial $P(x)$ is the resultant of $P(x)$ and $P'(x)$.
The discriminant is zero if and only if the polynomial $P$ has a double root. If the coefficients of a quartic polynomial $P$ are all real and the discriminant is negative, then it has two distinct real roots and a pair of complex-conjugated roots, while if the discriminant is positive, then either all four roots are real or the polynomial has two pairs of complex-conjugated roots.

The projective invariants satisfy the following.
\begin{proposition}
Suppose that $P(x)$ is a quartic polynomial given by \eqref{eq:quartic} with the Eisenstein invariants $D$, $E$.
Then:
\begin{itemize}
	\item the discriminant of $P(x)$ is $256(D^3-27E^2)$;
	\item the Eisenstein invariants of the polynomials $P(x+\alpha)$ and $x^4P(1/x)$ are equal to $D$, $E$;
	\item the Eisenstein invariants of the polynomial $P(\beta x)$ are $\beta^4 D$ and $\beta^6 E$.
\end{itemize}
\end{proposition}
\begin{proof}
By straightforward calculation.
\end{proof}

\begin{theorem}[Frobenius, \cite{Fr}]
Let $Q(x,y)$ be a biquadratic polynomial given by \eqref{eq:biquad1}.
If $\mathcal D_{Q_x}(y)$, $\mathcal D_{Q_y}(x)$ are its discriminant polynomials \eqref{eq:discxy}, and $D_y, E_y$ and $D_x, E_x$ their Eisenstein invariants, then:
$$D_y=D_x, \quad E_y=E_x.$$
\end{theorem}

\begin{corollary}\label{cor:discriminants}
Let $Q(x, y)$ be a biquadratic polynomial \eqref{eq:biquad1} and $\mathcal D_{Q_x}(y)$ and $\mathcal D_{Q_y}(x)$ its discriminant polynomials \eqref{eq:discxy}.
Then the discriminants of $\mathcal D_{Q_x}(y)$ and $\mathcal D_{Q_y}(x)$ are equal.
Moreover, we have:
\begin{align*}
	D_x=D_y=\, &
	\frac{1}{12} \left(a_{11}^2-4 a_{00} a_{22}-4 a_{01} a_{21}-4 a_{02} a_{20}+2 a_{10} a_{12}\right)^2
	\\&
	-( a_{10} a_{11}-2 a_{00} a_{21}-2 a_{01} a_{20})
	( a_{11} a_{12}-2 a_{01} a_{22}-2 a_{02} a_{21})
	\\&
	+\left(a_{10}^2-4 a_{00} a_{20}\right) \left(a_{12}^2-4 a_{02} a_{22}\right),\\
	E_x=E_y=\, &
	-\frac{1}{6} \left(a_{10}^2-4 a_{00} a_{20}\right)
	\left(a_{12}^2-4 a_{02} a_{22}\right)
	\left(a_{11}^2-4 a_{00} a_{22}-4 a_{01} a_{21}-4 a_{02} a_{20}+2 a_{10} a_{12}\right)
	\\&
	+\frac{1}{4} \left(a_{10}^2-4 a_{00} a_{20}\right)
	(a_{11} a_{12}-2 a_{01} a_{22}-2 a_{02} a_{21})^2
	\\&
	+\frac{1}{216} \left(a_{11}^2-4 a_{00} a_{22}-4 a_{01} a_{21}-4 a_{02} a_{20}+2 a_{10} a_{12}\right)^3
	\\&
	-\frac{1}{12} (a_{10} a_{11}-2 a_{00} a_{21}-2 a_{01} a_{20})
	(a_{11} a_{12}-2 a_{01} a_{22}-2 a_{02} a_{21})
	\times
	\\&\ \ \ \times
	\left(a_{11}^2-4 a_{00} a_{22}-4 a_{01} a_{21}-4 a_{02} a_{20}+2 a_{10} a_{12}\right)
	\\&
	+\frac{1}{4} \left(a_{12}^2-4 a_{02} a_{22}\right)
	(a_{10} a_{11}-2a_{00} a_{21}-2 a_{01} a_{20})^2.
\end{align*}
\end{corollary}

\emph{The homogenization} of the biquadratic polynomial $Q(x,y)$ given by \eqref{eq:biquad} is:
\begin{equation}\label{eq:Qhom}
	^h Q(x_0, x_1, y_0, y_1)=\sum_{i,j=0}^2a_{ij}x_0^ix_1^{2-i}y_0^jy_1^{2-j},
\end{equation}
while the homogenization of the quartic polynomial $P(x)$ given by \eqref{eq:quartic} is:
\begin{equation}\label{eq:quartic-hom}
	^h P(x_0,x_1)= a_4x_0^4+4a_3x_0^3x_1+6a_2x_0^2x_1^2+4a_1x_0x_1^3+a_0x_1^4,
\end{equation}
The invariants $D$ and $E$ and the discriminant of $^hP$ are defined as  the corresponding quantities for $P$.

\subsection{The surface $\mathbb{P}^1\times\mathbb{P}^1$ and the projective plane $\mathbb P^2$}

Before considering the curves given as the zero locus of biquadratic polynomials \eqref{eq:biquad} and \eqref{eq:Qhom}, we recall the geometric settings of the ambient spaces where those curves are defined, i.e.~the affine plane $\mathbb{C}^2$ and the surface $\mathbb{P}^1\times\mathbb{P}^1$.
Later in this section we will also review the projective plane  over complex numbers $\mathbb P^2$.
We will show that it is birationally equivalent to $\mathbb{P}^1\times\mathbb{P}^1$. We will use that for transforming biquadratic curves in $\mathbb{P}^1\times\mathbb{P}^1$ into cubic ones in $\mathbb{P}^2$.

Suppose that coordinates in $\mathbb P^1\times \mathbb P^1$ are $([x_0:x_1],[y_0:y_1])$.
The coordinate lines in that coordinate system belong to two classes that are given by
$x_0:x_1=\const$ and $y_0:y_1=\const$, respectively.
Note that two coordinate lines in the same class are disjoint, thus each coordinate line has self-intersection number equal to zero \cites{GrifHarPRINC,HartshorneAG,DuistermaatBOOK}.
The surface $\mathbb P^1\times \mathbb P^1$ is covered by four affine charts: each of those charts is obtained by one choice of
$i,j\in\{0,1\}$ and conditions $x_i\ne 0$ and $y_j\ne 0$.
The four charts and their local coordinate systems are presented schematically in the left-hand side of Figure \ref{fig:P1xP1}.

Now, consider the projective plane $\mathbb P^2$ and its coordinates $[X:Y:Z]$.
That plane is covered by three affine charts, each corresponding to one of the conditions $X\neq0$, $Y\neq0$, $Z\neq0$, see Figure \ref{fig:P1xP1}.
Since two distinct lines in the projective plane have a unique intersection point, the self-intersection number of any line is equal to $1$.

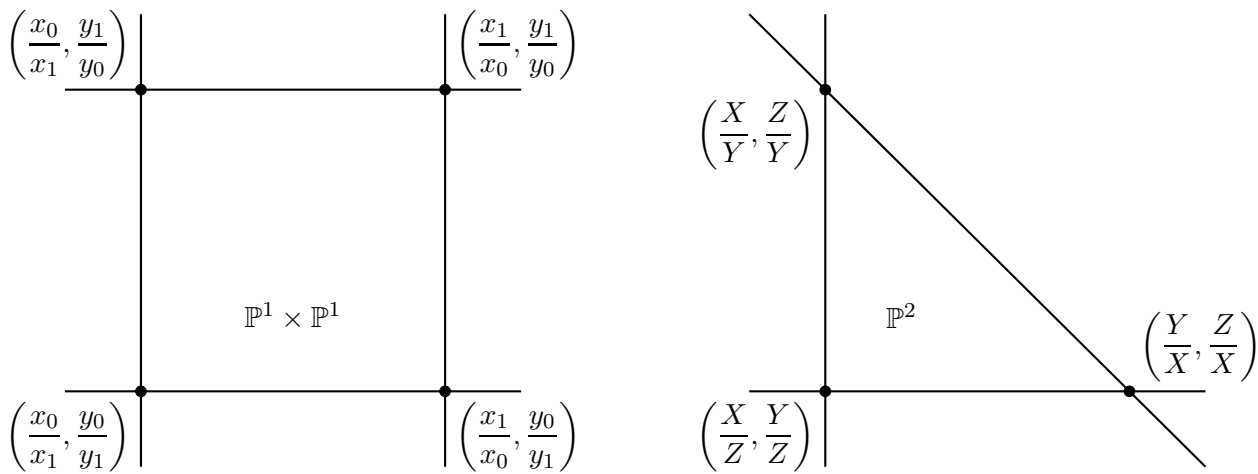
\begin{figure}[h]
	\begin{center}
		\begin{tikzpicture}[scale=0.8]
			\draw[thick] (-3,-2) -- (3,-2);
			\draw[thick] (-3,2) -- (3,2);
			
			\draw[thick] (-2,-3) --(-2,3);
			\draw[thick] (2,-3) --(2,3);
			
			\draw[fill=black] (2,2) circle [radius=0.07] node [above right]{$\left(\dfrac{x_1}{x_0},\dfrac{y_1}{y_0}\right)$};
			\draw[fill=black] (-2,-2) circle [radius=0.07] node [below left]{$\left(\dfrac{x_0}{x_1},\dfrac{y_0}{y_1}\right)$};		
			\draw[fill=black] (-2,2) circle [radius=0.07] node [above left]{$\left(\dfrac{x_0}{x_1},\dfrac{y_1}{y_0}\right)$};		
			\draw[fill=black] (2,-2) circle [radius=0.07] node [below right]{$\left(\dfrac{x_1}{x_0},\dfrac{y_0}{y_1}\right)$};	
\node at (0,-1){$\mathbb{P}^1\times\mathbb{P}^1$};	

\begin{scope}[shift={(9,0)}]
\draw[thick] (-3,-2) -- (3,-2);
\draw[thick] (-3,3) -- (3,-3);
\draw[thick] (-2,-3) --(-2,3);

\draw[fill=black] (-2,-2) circle [radius=0.07] node [below left]{$\left(\dfrac{X}{Z},\dfrac{Y}{Z}\right)$};		

\draw[fill=black] (-2,2) circle [radius=0.07] node [below left]{$\left(\dfrac{X}{Y},\dfrac{Z}{Y}\right)$};		

\draw[fill=black] (2,-2) circle [radius=0.07] node [above right]{$\left(\dfrac{Y}{X},\dfrac{Z}{X}\right)$};		

\node at (-1,-1){$\mathbb{P}^2$};	
\end{scope}
		\end{tikzpicture}
	\end{center}\caption{On the left: The surface $\mathbb{P}^1\times\mathbb{P}^1$, covered by four affine charts. All coordinate lines have self-intersection number equal to $0$.
	On the right:
	The projective plane $\mathbb P^2$, with local coordinate systems in three affine charts.  All lines in $\mathbb P^2$ have self-intersection number equal to $1$.
	 }\label{fig:P1xP1}
\end{figure}

The affine plane $\mathbb{C}^2$ with the coordinate system $(x,y)$ can be naturally embedded into $\mathbb{P}^1\times\mathbb{P}^1$ by the identification to any of the four affine charts: for example, take $(x,y)\to([x:1],[y:1])$.
In that way, the plane $\mathbb{C}^2$ is \emph{compactified} by adding two lines ``at infinity".
On the other hand, the affine plane can also be embedded into the projective plane: for example by the mapping $(x,y)\to[X:Y:1]$.
In that way, it is compactified by adding one line at infinity.
Those two embeddings together with the coordinate transformations are shown in Figure \ref{fig:embed}.

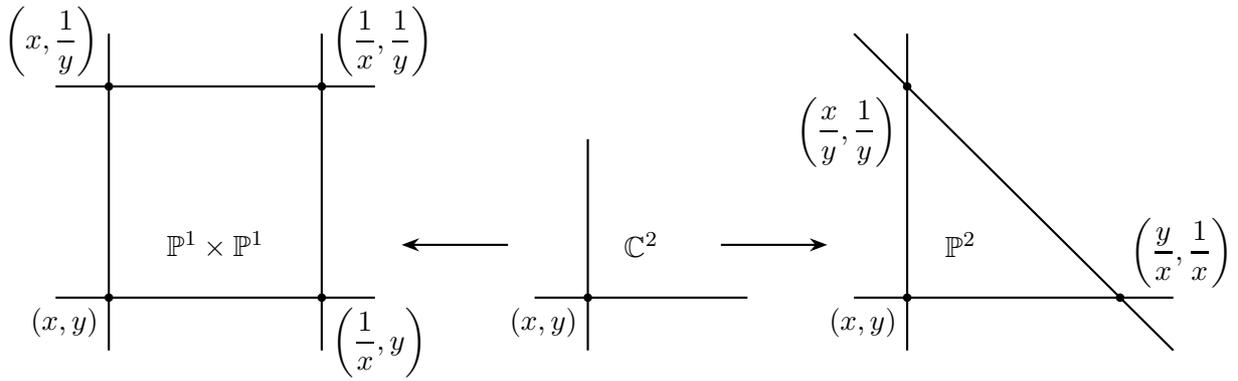
\begin{figure}[h]
	\begin{center}
		\begin{tikzpicture}[scale=0.5]
			\draw[thick] (-3,-2) -- (3,-2);
\draw[thick] (-3,2) -- (3,2);

\draw[thick] (-2,-3) --(-2,3);
\draw[thick] (2,-3) --(2,3);

\draw[fill=black] (2,2) circle [radius=0.07] node [above right]{$\left(\dfrac{1}{x},\dfrac{1}{y}\right)$};
\draw[fill=black] (-2,-2) circle [radius=0.07] node [below left]{$(x,y)$};		
\draw[fill=black] (-2,2) circle [radius=0.07] node [above left]{$\left(x,\dfrac{1}{y}\right)$};		
\draw[fill=black] (2,-2) circle [radius=0.07] node [below right]{$\left(\dfrac{1}{x},y\right)$};	
\node at (0,-1){$\mathbb{P}^1\times\mathbb{P}^1$};	

\begin{scope}[shift={(15,0)}]
			\draw[thick] (-3,-2) -- (3,-2);
			\draw[thick] (-3,3) -- (3,-3);
			
			\draw[thick] (-2,-3) --(-2,3);

			\draw[fill=black] (-2,-2) circle [radius=0.07] node [below left]{$(x,y)$};		
			
			\draw[fill=black] (-2,2) circle [radius=0.07] node [below left]{$\left(\dfrac{x}{y},\dfrac1y\right)$};		
			
			\draw[fill=black] (2,-2) circle [radius=0.07] node [above right]{$\left(\dfrac{y}{x},\dfrac1x\right)$};		
			
			\node at(-1,-1) {$\mathbb{P}^2$};
\end{scope}

\begin{scope}[shift={(9,0)}]
	\draw[thick] (-3,-2) -- (1,-2);
	\draw[thick] (-2,-3) -- (-2,1);
	
	\draw[fill=black] (-2,-2) circle [radius=0.07] node [below left]{$(x,y)$};		
	\node at(-1,-1) {$\mathbb{C}^2$};
\end{scope}			

\draw [thick, -Stealth] (9.5, -1) -- (11.5, -1);

\draw [thick, Stealth-] (3.5, -1) -- (5.5, -1);

		\end{tikzpicture}
	\end{center}\caption{The embeddings of the affine plane $\mathbb{C}^2$ into the surface $\mathbb{P}^1\times\mathbb{P}^1$ and the projective plane $\mathbb{P}^2$.}\label{fig:embed}
\end{figure}

Now, we will construct a surface that covers both  $\mathbb{P}^1\times\mathbb{P}^1$  and  $\mathbb{P}^2$.
For that, we will use the blow-up, which is one of the central constructions in algebraic geometry, see e.g.~\cites{HartshorneAG,GrifHarPRINC,DuistermaatBOOK, Clem}.

\begin{definition}\label{def:blow-up}
	\emph{The blow-up} of the plane $\mathbb{C}^2$ at point $(0,0)$ is the closed subset $\mathcal{X}$ of $\mathbb{C}^2\times\mathbb{CP}^1$ defined by the equation $u_1t_2=u_2t_1$, where $(u_1,u_2)\in\mathbb{C}^2$ and $[t_1:t_2]\in\mathbb{CP}^1$, see Figure \ref{fig:blow-up}.
	There is a natural morphism $\varphi: \mathcal{X}\to\mathbb{C}^2$, which is the restriction of the projection from $\mathbb{C}^2\times\mathbb{CP}^1$ to the first factor.
	The inverse image of the origin, $\varphi^{-1}(0,0)$ is the projective line $\mathcal{E}=\{(0,0)\}\times\mathbb{CP}^1$, called \emph{the exceptional line}.
	The morphism $\varphi$ is also called \emph{the blow-down along $\mathcal{E}$.}
\end{definition}

\begin{figure}[h]
	\centering
	\includegraphics[width=10.038cm, height=5.781cm]{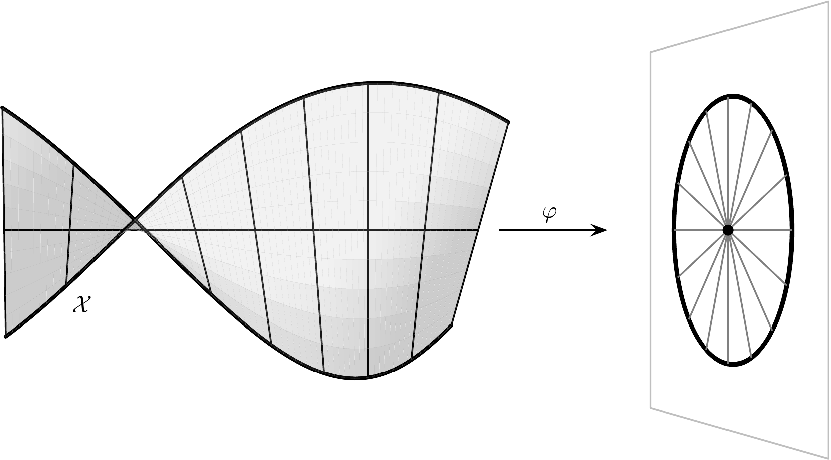}
	\caption{The blow-up of the plane at a point.}\label{fig:blow-up}
\end{figure}

\begin{remark}
	Notice that the points of the exceptional line $\varphi^{-1}(0,0)$ are in bijective correspondence with the lines containing $(0,0)$.
	On the other hand,
	$\varphi$ is an isomorphism between $\mathcal{X}\setminus\varphi^{-1}(0,0)$ and $\mathbb{C}^2\setminus\{(0,0)\}$.
	More generally, any complex two-dimensional surface can be blown up at a point \cites{HartshorneAG,GrifHarPRINC,DuistermaatBOOK}.
	In a local chart around that point, the construction will look the same as described for the case of the plane.
\end{remark}

\begin{remark}
If a curve $\mathcal{K}$ in $\mathbb{C}^2$ contains the origin, then its blow-up preimage in $\mathcal{X}$ is the union of the exceptional line with another curve, which is the closure of $\varphi^{-1}(\mathcal{K}\setminus\{(0,0)\})$.
Thus $\overline{\varphi^{-1}(\mathcal{K}\setminus\{(0,0)\})}$ is called
\emph{the proper preimage of $\mathcal{K}$}.
Notice that, since the blow-up separates curves intersecting at the origin, the self-intersection number of the proper preimage of $\mathcal{K}$ is less by $1$ than the self-intersection number of $\mathcal{K}$, if the curve contains the origin.
If the curve does not contain the origin, the self-intersection number remains the same for its preimage.
The self-intersection number of the exceptional line equals $-1$.
See \cites{HartshorneAG,GrifHarPRINC} for details.
\end{remark}

We construct a surface $\mathcal{S}$, that is obtained from $\mathbb{P}^2$ by blow-ups at two points and the same surface is also obtained from $\mathbb{P}^1\times\mathbb{P}^1$ by one blow-up.

Without losing generality, we choose the two points in $\mathbb{P}^2$ to be the intersection points of the line at infinity with the coordinate axes, i.e.~the points with coordinates $[X:Y:Z]=[1:0:0]$ and $[X:Y:Z]=[0:1:0]$; while in $\mathbb{P}^1\times\mathbb{P}^1$ we choose the point $([x_1:x_0],[y_1:y_0])=([1:0],[1:0])$.
Thus, using the coordinates represented in Figure \ref{fig:embed}, the surface $\mathcal{S}$, covered by five affine charts, is shown in Figure \ref{fig:S}.
\begin{figure}[h]
	\begin{center}
	\begin{tikzpicture}%[scale=0.8]
		\draw[thick] (-3,-2) -- (3,-2);
		
		\draw[thick] (-2,-3) --(-2,3);
		
		\draw[thick,dashed] (1.5,2) -- (-3,2);
		
		\draw[thick,dashed] (2,-3) --(2,1.5);
		
		\draw[thick,dashed](.5,2.5)--(2.5,.5);

		\draw[fill=black] (-2,-2) circle [radius=0.07] node [below left]{$(x,y)$};		
		\draw[fill=black] (-2,2) circle [radius=0.07] node [above left]{$\left(x,\dfrac1y\right)$};		
		\draw[fill=black] (2,-2) circle [radius=0.07] node [below right]{$\left(\dfrac1x,y\right)$};	
		
		\draw[fill=black] (1,2) circle [radius=0.07] node [above right]{$\left(\dfrac1{x},\dfrac{x}y\right)$};	
		
		\draw[fill=black] (2,1) circle [radius=0.07] node [below left]{$\left(\dfrac{y}{x},\dfrac1y\right)$};	
		
		\node at(-1,-1) {$\mathcal{S}$};
	\end{tikzpicture}
\end{center}
\caption{The surface $\mathcal{S}$. The dashed lines are exceptional, i.e.~their self-intersection number is $-1$.}\label{fig:S}
\end{figure}
The surface $\mathcal{S}$ is projected to $\mathbb{P}^1\times\mathbb{P}^1$ by blowing down one of the exceptional lines and to $\mathbb{P}^2$ by blowing down two of them, see Figure \ref{fig:S-proj}.
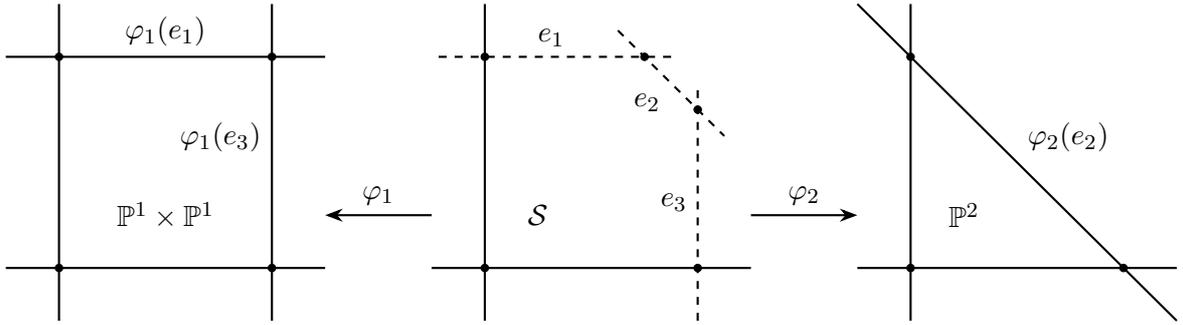
\begin{figure}[h]
	\begin{center}
		\begin{tikzpicture}[scale=0.7]
			\draw[thick] (-3,-2) -- (3,-2);
			
			\draw[thick] (-2,-3) --(-2,3);
			
			\draw[thick,dashed] (1.5,2) -- (-3,2) node[midway,above]{$e_1$};
			
			\draw[thick,dashed] (2,-3) --(2,1.5) node[midway,left]{$e_3$};
			
			\draw[thick,dashed](.5,2.5)--(2.5,.5) node[midway,below left]{$e_2$};
			
			\draw[fill=black] (-2,-2) circle [radius=0.07];		
			\draw[fill=black] (-2,2) circle [radius=0.07];		
			\draw[fill=black] (2,-2) circle [radius=0.07];	
			\draw[fill=black] (1,2) circle [radius=0.07];	
			\draw[fill=black] (2,1) circle [radius=0.07];	
			
			\node at(-1,-1) {$\mathcal{S}$};

\draw [thick, -Stealth] (3, -1) -- (5, -1)node[midway,above]{$\varphi_2$};
\draw [thick, -Stealth] (-3, -1) -- (-5, -1) node[midway,above]{$\varphi_1$};

\begin{scope}[shift={(-8,0)}]			
			\draw[thick] (-3,-2) -- (3,-2);
\draw[thick] (-3,2) -- (3,2) node[midway,above]{$\varphi_1(e_1)$};

\draw[thick] (-2,-3) --(-2,3);
\draw[thick] (2,-3) --(2,3)node[midway,above left]{$\varphi_1(e_3)$};

\draw[fill=black] (2,2) circle [radius=0.07];
\draw[fill=black] (-2,-2) circle [radius=0.07];		
\draw[fill=black] (-2,2) circle [radius=0.07];		
\draw[fill=black] (2,-2) circle [radius=0.07];	
\node at (0,-1){$\mathbb{P}^1\times\mathbb{P}^1$};	
\end{scope}

\begin{scope}[shift={(8,0)}]
	\draw[thick] (-3,-2) -- (3,-2);
	\draw[thick] (-3,3) -- (3,-3)node[midway,above right]{$\varphi_2(e_2)$};
	
	\draw[thick] (-2,-3) --(-2,3);

	\draw[fill=black] (-2,-2) circle [radius=0.07];		
	
	\draw[fill=black] (-2,2) circle [radius=0.07];		
	
	\draw[fill=black] (2,-2) circle [radius=0.07];		
	
	\node at(-1,-1) {$\mathbb{P}^2$};
\end{scope}

		\end{tikzpicture}
	\end{center}
	\caption{The projections of $\mathcal{S}$ to $\mathbb{P}^1\times\mathbb{P}^1$ and $\mathbb{P}^2$: map $\varphi_1$ is the blowdown along the exceptional line $e_2$, while $\varphi_2$ is the blowdown along $e_1$ and $e_3$.}\label{fig:S-proj}
\end{figure}

\subsection{Biquadratic curves in $\mathbb C^2$ and in $\mathbb{P}^1\times\mathbb{P}^1$}

\emph{A biquadratic curve} $\mathcal C_A$ in $\mathbb C^2$ is defined by the equation $Q(x,y)=0$, where $Q(x,y)$ is a biquadratic polynomial \eqref{eq:biquad}.
The compactification of that curve in $\mathbb{P}^1\times\mathbb{P}^1$ is the curve $\mathcal{C}$ given by the equation $^hQ(x_0,x_1,y_0,y_1)=0$, where $^hQ$ is the homogenization of $Q$ and given by the equation \eqref{eq:Qhom}.
Curve $\mathcal{C}_A$ is also called \emph{the affine part} of $\mathcal{C}$.

 Following \cite{DuistermaatBOOK}, we get:

\begin{theorem}\label{th:smooth}
Given a biquadratic polynomial $Q(x, y)$ \eqref{eq:biquad1} and its discriminant polynomials $\mathcal D_{Q_x}(y)$ and $\mathcal D_{Q_y}(x)$  \eqref{eq:discxy}, denote their fundamental projective invariants by $D_y=D_x$ and $E_y=E_x$.  The curve $\mathcal C$ in  $\mathbb P^1\times \mathbb P^1$, whose affine part is given as the zero set \eqref{eq:biquad}. Then the curve $\mathcal C$  is smooth if and only if the discriminant of the polynomials $\mathcal D_{Q_x}(y)$ and $\mathcal D_{Q_y}(x)$, which is equal to $256(D_x^3-27E_x^2)=256(D_y^3-27E_y^2)$, is non-zero. In this case, the curve $\mathcal C$ is elliptic and its $J$ invariant is equal to
$$
J= \frac{D_x^3}{D_x^3-27E_x^2}=\frac{D_y^3}{D_y^3-27E_y^2}.
$$
\end{theorem}

\begin{definition}
\emph{The projective invariants} $D_{\mathcal C}$ and $E_{\mathcal C}$ and the discriminant $F_{\mathcal C}$ of the biquadratic curve $\mathcal C$ in  $\mathbb P^1\times \mathbb P^1$, whose affine part is given as the zero set \eqref{eq:biquad}, are:
\begin{equation}\label{eq:GxGyExEy}
D_{\mathcal C}:=D_x=D_y, \quad E_{\mathcal C}:=E_x=E_y, \quad F_{\mathcal C}:=256(D_x^3-27E_x^2)=256(D_y^3-27E_y^2).
\end{equation}
\end{definition}

In \cite{Cayley1871}, Cayley proved that using a linear transformation in one variable in $Q$ given by \eqref{eq:biquad1}, one can get that the discriminant polynomials $\mathcal D_{Q_x}(y)$ and $\mathcal D_{Q_y}(x)$ have equal corresponding coefficients.
Thus, he proved the following:

\begin{proposition}[Cayley, \cite{Cayley1871}]\label{prop:sym}
Let  $\mathcal C$ be a smooth curve in  $\mathbb P^1\times \mathbb P^1$ defined by a  nonsymmetric biquadratic equation \eqref{eq:biquad}.
Then there exists a projective transformation $f$ in one variable such that $\hat Q(x, y)=Q(x,f(y))$ is symmetric.
\end{proposition}
\begin{proof}
Here we give an outline.
For a complete proof, see \cite{Sam}.
According to Theorem \ref{th:smooth}, each of the two the discriminant polynomials has four distinct roots, and the two cross-ratios of the roots of each of these quartic polynomials are equal.
Thus, there exists a M\"obius transformation $f$ that maps the roots of  $\mathcal D_{Q_x}(y)$ to the roots of $\mathcal D_{Q_y}(x)$.
The transformation $f$ symmetrizes the biquadratic polynomial $Q$.
\end{proof}

Frobenius gave in \cite{Fr} a complete list of those that can be symmetrized as well as the list of those that cannot, see also \cite{Sam}.
We will come back to that list in Section \ref{sec:singular}.

\subsection{ From a smooth biquadratic in $\mathbb{P}^1\times\mathbb{P}^1$ to a smooth cubic in $\mathbb{P}^2$}

Consider a smooth biquadratic $\mathcal C$ in $\mathbb{P}^1\times\mathbb{P}^1$.
According to Theorem \ref{th:smooth}, $\mathcal C$ is an elliptic curve, thus it is isomorphic to  $\mathbb C/\Lambda$, for a nondegenerate lattice $\Lambda\subset \mathbb C$.
If we denote by $g_2$ and $g_3$ the invariants of $\Lambda$, then
the curve $\mathcal{C}$ is isomorphic to a smooth cubic $\Gamma$ in $\mathbb{P}^2$, which can be represented in the Weierstrass form as:
$y^2= 4x^3-g_2x-g_3$,
see, for example, \cites{DuistermaatBOOK, Tsu, Clem, DR2011knjiga}.
The isomorphism between $\mathcal{C}$ and $\Gamma$ is represented in Figure \ref{fig:com-diagram}.
\begin{figure}[h]
	\centering
	\begin{tikzpicture}
\node at (0,0){$\mathbb{C}/\Lambda$};
\node at (-2.2,-1.7){$\mathbf{P}^1\times\mathbf{P}^1\supset	\mathcal{C}$};
\node at (1.8,-1.7){$\Gamma\subset\mathbf{P}^2$};

\draw[thick,-Stealth](-0.2,-0.3)--(-1.1,-1.5) node[midway,above left]{$\Phi$};

\draw[thick,-Stealth](0.2,-0.3)--(1.1,-1.5) node[midway,above right]{$p$};

\draw[thick,-Stealth](-1,-1.7)--(1,-1.7) node[midway,below]{$p\circ\Phi^{-1}$};
	\end{tikzpicture}
	\caption{Mapping $\Phi$ is an analytic diffeomorphism between  $\mathbb C/\Lambda$ and $\mathcal C$, while
$p(z) =[\wp(z):\wp'(z):1]$ is an analytic diffeomorphism from $\mathbb C/\Lambda$ to cubic $\Gamma$ in $\mathbb P^2$, where $\wp(z)$ is the corresponding Weierstrass function. The map $p\circ \Phi^{-1}$ is a complex analytic diffeomorphism from the smooth biquadratic $\mathcal C\subset \mathbb{P}^1\times\mathbb{P}^1$ to the smooth cubic $\Gamma \subset \mathbb P^2$.
	}\label{fig:com-diagram}
\end{figure}

Moreover, there is the following connection between the projective invariants of $\mathcal{C}$ and the coefficients of the cubic $\Gamma$:

\begin{proposition}[\cite{DuistermaatBOOK}]\label{prop:biquadcubic}
	Consider a smooth biquadratic $\mathcal C$ in $\mathbb{P}^1\times\mathbb{P}^1$.
	Then its projective invariants $D_{\mathcal{C}}$ and $E_{\mathcal{C}}$ satisfy the following:
	\begin{equation}\label{eq:EFg}
		g_2=D_{\mathcal C}, \quad g_3=-E_{\mathcal C}.
	\end{equation}
	In other words, $\mathcal C$ is isomorphic to the smooth cubic $\Gamma \subset \mathbb P^2$ given by the affine equation:
	\begin{equation}\label{eq:cubic}
		\Gamma: y^2= 4x^3-g_2x-g_3 = 4x^3- D_{\mathcal C}x+E_{\mathcal C}.
	\end{equation}
\end{proposition}

The isomorphism between $\mathcal{C}$ and $\Gamma$ can be geometrically realised as follows.
First, choose a coordinate system in $\mathbb{P}^1\times\mathbb{P}^1$ such that the point $\mathcal{O}=([1:0],[1:0])$ belongs to $\mathcal{C}$ and the tangent line at that point does not coincide with any of the two coordinate lines through it.
Notice that in that coordinate system, the curve $\mathcal{C}$ will be given by the polynomial $^hQ$ \eqref{eq:Qhom} satisfying $a_{22}=0$ and $a_{21}a_{12}\neq0$.

Second, we apply the blow-up $\varphi_1$, as shown in Figure \ref{fig:S-proj}.
There, the point $\mathcal{O}$ is blown-up to the exceptional line $e_2$, while the proper preimages of the coordinate lines through $\mathcal{O}$ are the lines $e_1$ and $e_2$, which both have self-intersection number equal to $-1$.
Note that the proper preimage $\tilde{\mathcal{C}}$ of the curve $\mathcal{C}$ intersects each of the lines $e_1$, $e_2$, $e_3$ at a single point and that the intersection with $e_2$ does not lie on $e_1$ or $e_3$.

Third, we apply the blow down $\varphi_2$ of the lines $e_1$ and $e_3$.
The projection $\varphi_2(\tilde{\mathcal{C}})$ is a smooth curve.
Now, since $a_{22}=0$, notice that in the first affine chart of $\mathbb{P}^1\times\mathbb{P}^1$, the equation of the curve $\mathcal{C}$ will be cubic, thus $\varphi_2(\tilde{\mathcal{C}})$ is a cubic in $\mathbb{P}^2$.
Applying an appropriate change of coordinates of the projective plane, one can get the Weierstrass form \eqref{eq:cubic} \cite{GrifHarPRINC}.

Notice that this construction shows that the isomorphism between biquadratic $\mathcal{C}$ and cubic $\Gamma$ is a restriction of the birational equivalence between $\mathbb{P}^1\times\mathbb{P}^1$ and $\mathbb{P}^2$.

An explicit analytic form of that isomorphism is given in the following:

\begin{proposition}[\cite{DuistermaatBOOK}*{Lemma 2.4.13}]\label{prop:transformation}
Consider a smooth biquadratic $\mathcal C$ in $\mathbb{P}^1\times\mathbb{P}^1$ given by polynomial \eqref{eq:Qhom}, where the coordinates are chosen such that $a_{22}=0$.
Using the notation introduced in this subsection, denote by $\Lambda$ and $\wp$ the corresponding lattice and the Weierstrass function, and suppose that $\Phi$ is the isomorphism between $\mathbb{C}/\Lambda$ and $\mathcal{C}$ such that $\Phi(0+\Lambda)=\mathcal{O}$.
Define the rational functions of two variables $\mathcal P(x, y)$ and $\mathcal P'(x, y)$ as:
\begin{align*}
\mathcal P(x, y)&= - \frac{(a_{02}x+a_{12})(a_{20}y+a_{21})}{xy}+
\frac{a_{11}^2-4a_{10}a_{12}-4a_{01}a_{21}+8a_{02}a_{20} }{12}\\
\mathcal P'(x, y)&= \frac{\mathcal Q(x, y)}{x^2y^2},
\end{align*}
where
\begin{align*}
\mathcal Q(x, y)=&-a_{12}^2a_{21}x-3a_{02}a_{12}a_{21}x^2 -2a_{02}^2a_{21}x^3+a_{12}a_{21}^2y-(a_{02}a_{11}+a_{01}a_{12})a_{21}x^2y\\
&-2a_{01}a_{02}a_{21}x^3y+3a_{12}a_{20}a_{21}y^2+(a_{11}a_{20}+a_{10}a_{21})a_{12}xy^2+(a_{01}a_{12}a_{20}-a_{02}a_{10}a_{21})x^2y^2\\
&-2a_{00}a_{02}a_{21}x^3y^2+2a_{12}a_{20}^2y^3+2a_{10}a_{12}a_{20}xy^3+2a_{00}a_{12}a_{20}x^2y^3.
\end{align*}
Then:
$$
\wp(z)=\mathcal P(\Phi(z)),\quad \frac{d}{dz} \wp(z)=\mathcal P'(\Phi(z))
$$
and
\begin{equation}\label{eq:psi}
 ([x:1], [y:1])\mapsto [\mathcal P(x, y):\mathcal P'(x, y):1]\in \Gamma,
\end{equation}
is the formula of an isomorphism between the smooth biquadratic curve $\mathcal C$ in $\mathbf{P}^1\times\mathbf{P}^1$ and the smooth cubic $\Gamma \subset \mathbb P^2$ given by \eqref{eq:cubic}.
\end{proposition}

One of the beauties of the problem at hand is that both maps $p$ and $\Phi^{-1}$ are transcendental, and nevertheless their composition $p\circ\Phi^{-1}$, which has been explicitly written as \eqref{eq:psi}, is \emph{polynomial} in terms of the coefficients of the biquadratic polynomial that defines $\mathcal C$.

\subsection{A $(2,2)$ correspondence and QRT transformations}\label{sec:QRT}
A biquadratic curve $\mathcal C$, given in the plane $\mathbb{P}^1\times\mathbb{P}^1$ by the polynomial \eqref{eq:Qhom}, defines a \emph{$(2,2)$ correspondence}:
for a given point $[x_0:x_1]$ in the first copy of $\mathbb P^1$, there are, in general, two points $[y_0:y_1]$ in the second copy of
$\mathbb P^1$ such that $([x_0:x_1],[y_0:y_1])\in\mathcal{C}$, and vice versa:
for each $[y_0:y_1]$ in the second copy of $\mathbb{P}^1$ there are two points $[x_0:x_1]$ in the first copy such that $([x_0:x_1],[y_0:y_1])\in\mathcal{C}$.

The $(2,2)$ correspondence induces two natural maps on the curve $\mathcal C$ in $\mathbb{P}^1\times\mathbb{P}^1$, see Figure \ref{fig:switches}.
These two maps are generated by the following two maps of the affine part $\mathcal C_A$ of $\mathcal C$:
\begin{itemize}
	\item \emph{the horizontal switch}: $h:(x,y)\mapsto(x',y)$; and
	\item \emph{the vertical switch}: $v:(x,y)\mapsto(x,y')$,
\end{itemize}
where we assume that $x$ and $x'$ are the two solutions of the quadratic equation in $x$ with fixed $y$: $Q(x, y)=0$ and $Q(x', y)=0$, with $Q$ given by \eqref{eq:biquad}.
Similarly,  $y$ and $y'$ are the two solutions of the quadratic equation in $y$ with fixed $x$: $Q(x, y)=0$ and $Q(x, y')=0$.
By applying the Vieta formulas to \eqref{eq:biquad1}, explicit formulas can be written for both switches: $x'=-x - \tilde b(y)/\tilde a(y)$ and  $y'=-y -  b(x)/a(x)$, assuming that $\tilde a(y)\ne 0$ and $a(x)\ne 0$.

Both maps $h$ and $v$ are involutions on the curve $\mathcal C$, i.e.~their squares are the identity map, or in other words, each of them is bijective and equal to its inverse.

A point $x$ in $\mathbb P^1$ is \emph{critical for the projection parallel to the second axis} if the corresponding two points $y, y_1$ coincide, i.e. if $(x, y)$ is a fixed point of the vertical switch
	$$
	v(x, y)=(x,y).
	$$
	Similarly, a point $y$ in $\mathbb P^1$ is \emph{critical for the projection parallel to the first  axis} if the corresponding two points $x, x_1$ coincide, i.e. if $(x, y)$ is a fixed point of the horizontal switch
	$$
	h(x, y)=(x,y).
	$$
	The fixed points of horizontal and vertical switches  are exactly the zeros of $\mathcal D_{Q_x}(y)$ and $\mathcal D_{Q_y}(x)$, respectively, thus each $h$ and $v$ have four such points counting multiplicities.
	Denote by $d_1$ and $d_2$ the type of the critical divisor of the critical points at the first and the second coordinate, respectively.
	The types can be $(1,1,1,1)$, $(2, 1, 1)$, $(2, 2)$, $(3, 1)$, $(4)$, reflecting the structure of zeros, including infinity and counting multiplicity of the  polynomials
	$\mathcal D_{Q_x}(y)$ and $\mathcal D_{Q_y}(x)$.
	The type can be also undefined.
	In Section \ref{sec:singular} we will provide a full correspondence between the types and singular biquadratics.
	Here we just point out that the case of a smooth biquadratic $\mathcal C$ is characterized by
	$$
	d_1=d_2=(1, 1, 1, 1).
	$$
	This means that a biquadratic $\mathcal C$ is a smooth elliptic curve if and only if each of the vertical and the horizontal switches have four distinct fixed points, including points at infinity \cite{DuistermaatBOOK}.

The main object of our study is the composition of these two involutions:
\begin{equation}\label{eq:QRT}
\delta: \mathcal C\rightarrow\mathcal C: \delta= v\circ h.
\end{equation}
One should keep in mind that generally $v$ and $h$ do not commute with each other.

In the modern literature, this map $\delta$ is sometimes called the \emph{QRT transformation}, named after Quispel, Roberts, and Thompson, see \cite{DuistermaatBOOK}, where several examples of applications in particular to discrete integrable systems were provided. There is a very important relationship with the Poncelet  theorem from projective geometry and billiards within conics, see \cites{GrifHar1978,DR2011knjiga} and references therein, where the instances with  $\delta$ being of a finite order play the most significant role.

The main goal of this paper is to describe the biquadratic curves in $\mathbb{P}^1\times\mathbb{P}^1$ for which the order of the QRT map $\delta$ is finite.
While in the existing literature, including the Poncelet theorem and billiards within conics, the underlying $(2,2)$ correspondence, as well as the biquadratic polynomials $Q$ \eqref{eq:biquad1} were symmetric (see e.g.~\cite{DR2011knjiga}*{Section 4.12} or \cite{Sam}*{Theorem F}), here we cover nonsymmetric cases as well.
This is primarily motivated by the study of the finiteness of the groups of random walks in the quarter plane, as defined in Section \ref{sec:grw}, and which we are going to solve in the case of smooth biquadratics next.

\section{Smooth case: QRT transformations and groups of random walks of a finite order}\label{sec:smooth}

In this section, we consider a general case, that is when the biquadratic curve is a smooth elliptic curve, and we will derive the explicit analytic conditions for the finiteness of the order of the corresponding group of random walk.

The involutions given by the horizontal and vertical switches on such a biquadratic curve correspond to central symmetries on the Jacobian $\mathbb{C}/\Lambda$ of the curve.
Thus, their composition $h\circ v$ corresponds to a translation on the Jacobian.
We are interested to find explicit conditions for that translation to be of finite order.

The vector of the translation is explicitly obtained by the following:
\begin{proposition}[\cite{DuistermaatBOOK}]\label{prop:deltagamma}
Let $\mathcal C$ be a smooth biquadratic  in $\mathbb{P}^1\times\mathbb{P}^1$ defined by the polynomial \eqref{eq:Qhom}.
Then, its QRT transformation $\delta$ corresponds to the translation from the point at infinity $[0:1:0]$ to the point $[X_0:Y_0:1]$ in the smooth cubic $\Gamma$ in $\mathbb{P}^2$ given by  \eqref{eq:cubic}, where
\begin{equation}\label{eq:XY}
	\begin{aligned}
X_0
=&\frac{a_{11}^2-4a_{12}a_{10}-4a_{21}a_{01}+8a_{02}a_{20}+8a_{22}a_{00}}{12},\\
Y_0
=&\det(a_{ij}), \quad i,j=0, 1, 2.
\end{aligned}
\end{equation}
\end{proposition}

\begin{remark}
If $\mathcal C$ is singular, but its set of regular points $\mathcal C^{reg}$ non-empty, it is possible in Proposition \ref{prop:deltagamma} to substitute  $\mathcal C$ with a connected component of $\mathcal C^{reg}$ that is contained neither in horizontal nor vertical coordinate line and is invariant with respect to the QRT transformation.
The case of singular biquadratics is presented in detail in Section \ref{sec:singular}.
\end{remark}

Now we formulate the explicit conditions for periodicity of the QRT transformation.

\begin{theorem}\label{th:cayley}
Let $\mathcal C$ be a smooth biquadratic curve given by the polynomial \eqref{eq:biquad},  and $\Gamma$ its corresponding cubic \eqref{eq:cubic}.
Then the following statements hold:
\begin{itemize}
\item[(a)] The QRT transformation on  $\mathcal C$ is of order $n$ if and only if the translation on the cubic $\Gamma$  from the point at infinity $[0:1:0]$ to the point $[X_0:Y_0:1]$ given by \eqref{eq:XY} is of order $n$.
\item[(b)]
The translation on the cubic $\Gamma$ from the point at infinity $[0:0:1]$ to the point $[X_0:Y_0:1]$ is of order $n$ if and only if one of the following cases hold:
\begin{itemize}
	\item [(i)] $n=2$ and $\det(a_{ij})=0$;
	\item [(ii)] $n=2k+1$, $k\ge1$, and
	$$
	\det\begin{pmatrix}
	C_2 & C_3 & \dots & C_{k+1}\\
	C_3 & C_4 & \dots & C_{k+2}\\
	&&\dots& \\
	C_{k+1} & C_{k+2} &\dots & C_{2k}
	\end{pmatrix}
	=0;
	$$
	\item [(iii)] $n=2k$, $k\ge2$, and
	$$
	\det\begin{pmatrix}
		C_3 & C_4 & \dots & C_{k+1}\\
		C_4 & C_5 & \dots & C_{k+2}\\
		&&\dots& \\
		C_{k+1} & C_{k+2} &\dots & C_{2k-1}
	\end{pmatrix}
	=0.
	$$
    \end{itemize}
Here, the entries $C_k$ of the matrices are the coefficients in the following Taylor expansion about point $[X_0:Y_0:1]$ on the curve $\Gamma$:
\begin{equation}\label{eq:taylor}
\sqrt{4x^3- D_{\mathcal C}x+E_{\mathcal C}}
=
C_0+C_1(x-X_0)+C_2(x-X_0)^2+C_3(x-X_0)^3+\dots.
\end{equation}
\item[(c)] Moreover, the coefficients $C_0$, $C_1$, \dots are rational in $a_{ij}$.
\end{itemize}
\end{theorem}
\begin{proof}
Part (a) follows directly from Proposition \ref{prop:deltagamma}.

Denote $P_{\infty}=[0:1:0]$ and $P=[X_0:Y_0:1]$.
Then $2(P-P_{\infty})\sim0$ if and only if $P$ is a fixed point of the involution $(x,y)\mapsto(x,-y)$ of $\Gamma$, i.e.~if $Y_0=0$.
Thus, part (b)(i) follows from \eqref{eq:XY}.

The conditions for $n>2$ are derived using the method which is described in detail, for example, in \cite{GrifHar1978} and \cite{DR2011knjiga}, which concludes parts (ii) and (iii) of (b).

Finally, the coefficients $C_0$, $C_1$, \dots from \eqref{eq:taylor} are rational in $a_{ij}$, since:
$$
C_n
=
\frac{1}{n!}
\left(
\frac{d^n}{dx^n}\sqrt{4x^3- D_{\mathcal C}x+E_{\mathcal C}}
\right)_{x=X_0},
$$
and $\sqrt{4x^3- D_{\mathcal C}x+E_{\mathcal C}}=Y_0=\det(a_{ij})$.
This proves (c).
\end{proof}	

\begin{corollary}\label{cor:small}
Let $\mathcal C$ be a smooth biquadratic curve given by the polynomial  \eqref{eq:biquad}.
Then its group of random walk $\mathcal H$ is of order $m=2n$ if and only if:
\begin{itemize}
	\item [(i)] $\det(a_{ij})=0$ for $m=4$;
	\item [(ii)] $C_2=0$ for $m=6$;
\item [(iii)] $C_3=0$ and $\det(a_{ij})\ne0$ for $m=8$;
\item [(iv)] $C_3^2=C_2C_4$ for $m=10$;
\item [(v)] $C_4^2=C_3C_5$ and $C_2\ne 0$ for $m=12$.
\end{itemize}
Here, $C_2$, $C_3$, $C_4$, $C_5$ are coefficients in the Taylor expansion \eqref{eq:taylor}.
	Explicitly, they are calculated as follows:
\begin{align*}
C_2&=\frac{- D_{\mathcal{C}}^2-24 D_{\mathcal{C}} X_0^2+48 E_{\mathcal{C}} X_0+48 X_0^4}
{8 \left(4 X_0^3-D_{\mathcal{C}} X_0+E_{\mathcal{C}}\right)^{3/2}}
;\\
C_3&= \frac{-D_{\mathcal{C}}^3+20 D_{\mathcal{C}}^2 X_0^2-16 D_{\mathcal{C}}E_{\mathcal{C}} X_0+80 D_{\mathcal{C}} X_0^4+32E_{\mathcal{C}}^2-320E_{\mathcal{C}} X_0^3-64 X_0^6}
{16 \left(4 X_0^3-D_{\mathcal{C}} X_0+E_{\mathcal{C}}\right)^{5/2}};\\
C_4&= \frac1{128 \left(4 X_0^3-D_{\mathcal{C}} X_0+E_{\mathcal{C}}\right)^{7/2}}
\Big(
-5 D_{\mathcal{C}}^4+80 D_{\mathcal{C}}^3 X_0^2+32 D_{\mathcal{C}}^2 E_{\mathcal{C}} X_0-1120 D_{\mathcal{C}}^2 X_0^4+128 D_{\mathcal{C}} E_{\mathcal{C}}^2
\\&
\qquad\qquad\qquad\qquad
+1792 D_{\mathcal{C}} E_{\mathcal{C}} X_0^3-1792 D_{\mathcal{C}} X_0^6-3840 E_{\mathcal{C}}^2 X_0^2+10752 E_{\mathcal{C}} X_0^5+768 X_0^8
\Big)
;\\
C_5&= \frac1{256 \left(4 X_0^3-D_{\mathcal{C}} X_0+E_{\mathcal{C}}\right)^{9/2}}
\Big(-7 D_{\mathcal{C}}^5+132 D_{\mathcal{C}}^4 X_0^2+96 D_{\mathcal{C}}^3 X_0 \left(E_{\mathcal{C}}-9 X_0^3\right)
\\
&\qquad\qquad\qquad\qquad
+192 D_{\mathcal{C}}^2 \left(E_{\mathcal{C}}^2-10 E_{\mathcal{C}} X_0^3+70 X_0^6\right)-2304 D_{\mathcal{C}} X_0^2 \left(E_{\mathcal{C}}^2+14 E_{\mathcal{C}} X_0^3-5 X_0^6\right)
\\
&\qquad\qquad\qquad\qquad
-3072 X_0 \left(E_{\mathcal{C}}^3-24 E_{\mathcal{C}}^2 X_0^3+30 E_{\mathcal{C}} X_0^6+X_0^9\right)
\Big),
\end{align*}
where $X_0$ is given by \eqref{eq:XY}.
\end{corollary}

\section{Groups of random walks of small and not so small orders}\label{sec:small}

Groups of random walks of small orders were analysed in \cite{RandomWalks}, and analytic conditions were derived there for orders $4, 6$ and $8$.
We note that those conditions in form and the method of derivation differ from the general conditions we derived in Section \ref{sec:smooth}.
Thus, the aim of this section is to discuss the details, provide some examples, and show equivalence, while keeping in mind that in the derivations of \cite{RandomWalks}, the reality properties of transition probabilities were used, while our considerations are free from those restrictions.
We also present the new explicit characterization of random walks with the groups of order $10$ and provide new examples of random walks with the groups of orders $10$, $12$, $14$ and $16$.
No examples of random walks with the groups of such higher orders were known in the literature so far.
Our methodology can be easily used to generate random walks with the groups of any given order.

\subsection{Groups of order $4$}

Here, we want to provide a direct, independent proof  for a biquadratic curve \eqref{eq:biquad},
 that the horizontal and vertical switches generate a group of order four if and only if
$\det M_Q=0$, with
\begin{equation}\label{eq:MQ}
M_Q
=
\begin{pmatrix}
	a_{22} & a_{21} & a_{20}\\
	a_{12} & a_{11} & a_{10}\\
	a_{02} & a_{01} & a_{00}
\end{pmatrix}.
\end{equation}
This was proved for the groups of random walk in  \cite{RandomWalks}, see Example \ref{ex:rw4} below.

\begin{proposition}
The condition $\det M_Q=0$ is preserved by the Moebius transformations on the coordinates $x$, $y$.
\end{proposition}
\begin{proof}
Let $Q_1(x,y)=Q(x+\beta,y)$,
$Q_2(x,y)=Q(\alpha x,y)$, $Q_3(x,y)=x^2Q(1/x,y)$.
Then $\det(M_Q)=\det(M_{Q_1})$,
$\det(M_{Q_2})=\alpha^3\det(M_Q)$, and $\det(M_{Q_3})=-\det(M_Q)$.
\end{proof}

\begin{example}\label{ex:order4}
In the projective plane with homogeneous coordinates $[X:Y:Z]$, consider the following cubic curve, given by its affine equation in the chart $Z=1$:
$$
Y^2=X(\alpha-X)(\beta-X),
\quad\text{with}
\quad
\alpha\neq\beta,\ \alpha\beta\neq0.
$$
Denote by $P_{\infty}$ the point at the infinity of the cubic, and by $P_0$,  $P_{\alpha}$, $P_{\beta}$ the points with coordinates $(0,0)$, $(\alpha,0)$, $(\beta,0)$ respectively.
Let $\ell_{\infty}$, $\ell_0$ be the lines $Z=0$ and $X=0$.
	Note that $\ell_0$ is the tangent line to the cubic at $P_0$ and that it contains also point $P_{\infty}$, while $\ell_{\infty}$ is touching the curve at $P_{\infty}$, which is their triple intersection point.
	
For any $P$ on the curve, there is a natural involution $i_{P}$, which maps any point $Q$ of the curve to the third intersection point of the line $PQ$ with the curve.
	
	We note that $i_{P_{\infty}}$ fixes points $P_{\infty}$, $P_0$, $P_{\alpha}$, $P_{\beta}$, while $i_{P_0}$ maps those four points to $P_{0}$, $P_{\infty}$, $P_{\beta}$, $P_{\alpha}$ respectively.

It can be easily checked directly from definition that the composition $i_{P_0}\circ i_{P_{\infty}}$ is of order $2$, so we can conclude that a group of order $4$ is generated by those two involutions.
	
Another way to see that is to notice that the composition of those involutions is a translation $P_0-P_{\infty}$, which is of order $2$ since $2P_0\sim 2P_{\infty}$ on the Jacobian of the curve, see e.g. \cite{DR2011knjiga}.

Now consider the following transformation:
$[X:Y:Z]\mapsto[X_1:Y_1:Z_1]=[Z:Y:X]$.
That transformation maps $P_{\infty}$, which has coordinates $[0:1:0]$, to itself, $P_0$ to $[1:0:0]$, $P_{\alpha}$ to $[1/\alpha:0:1]$, $P_{\beta}$ to $[1/\beta:0:1]$.

In the affine chart $Z_1=1$, the equation of the curve is:
$$
X_1Y_1^2=(\alpha X_1-1)(\beta X_1-1),
$$
i.e.
$$
-\alpha\beta X_1^2+X_1Y_1^2+(\alpha+\beta)X_1-1=0.
$$
That affine chart can be embedded in $\mathbb{P}^1\times\mathbb{P}^1$, using the following transformation: $(X_1,Y_1)\mapsto([X_1:1],[Y_1:1])=([x:1],[y:1])$, so we get the equation:
$$
Q(x,y)=-\alpha\beta x^2+xy^2+(\alpha+\beta)x-1=0.
$$
Note that this represents blow-ups at $P_0$ and $P_{\infty}$ followed by a blow-down of the preimage of the line at the infinity.
Thus the reflections in $P_{\infty}$ and $P_0$ should then be lifted to the vertical and horizontal switches in the new coordinates, when the affine chart is completed to $\mathbf{P}^1\times\mathbf{P}^1$.

The corresponding matrix $M_Q$ \eqref{eq:MQ}  is:
$$
M_Q=\begin{pmatrix}
0 & 0 & -\alpha\beta\\
1 & 0 & \alpha+\beta\\
0 & 0 & -1
\end{pmatrix},
$$
which obviously has determinant equal to zero.

Let us analyse the fixed points of the vertical switch.
We have:
$$
Q(x,y)=a(x)y^2+b(x)y+c(x),
\quad\text{with}\quad
a(x)=x,\ b(x)=0,\ c(x)=-\alpha\beta x^2+(\alpha+\beta)x-1.
$$
The discriminant with respect to $y$ is:
$$
b^2(x)-4a(x)c(x)=4x(\alpha x-1)(\beta x-1).
$$
Thus the fixed points for the vertical switch are $(0,\infty)$, $(1/\alpha,0)$, $(1/\beta,0)$, $(\infty,\infty)$.

For the horizontal switch, we have:
$$
Q(x,y)=\tilde{a}(y)x^2+\tilde{b}(y)x+\tilde{c}(y),
\quad\text{with}\quad
\tilde{a}(y)=-\alpha\beta,\ \tilde{b}(y)=y^2+(\alpha+\beta),\ \tilde{c}(y)=-1.
$$
The discriminant with respect to $x$ is:
$$
\tilde{b}^2(x)-4\tilde{a}(x)\tilde{c}(x)=(y^2+\alpha+\beta)^2-4\alpha\beta,
$$
thus, the fixed points for the horizontal switch are:
$\left(-\frac{1}{\sqrt{\alpha\beta}},\pm i(\sqrt{\alpha}+\sqrt{\beta})\right)$ and
$\left(-\frac{1}{\sqrt{\alpha\beta}},\pm i(\sqrt{\alpha}-\sqrt{\beta})\right)$.
We can conclude that the biquadratic curve is smooth elliptic if and only if $\alpha\neq\beta$ and $\alpha\beta\neq0$, which agrees with the initial assumptions.
\end{example}

\begin{example}\label{ex:curve-trans}
In the projective plane, consider the cubic curve given by its affine equation:
	$$
	Y^2=(a-X)(b-X)(c-X),
	\quad\text{with}
	\quad
	a\neq b\neq c\neq a,\ abc\neq0.
	$$
Let $P_{\infty}$ be the point at the infinity, and $P_0$ the point with coordinates $(0,\sqrt{abc})$.
Here, unlike the previous example,  $P_0$  is not a branch point any more, thus $2P_0$ is not equivalent to $2P_{\infty}$.
	
	If $[X:Y:Z]$ are the projective coordinates of the plane, consider the following transformation:
	$[X:Y:Z]\mapsto[X_1:Y_1:Z_1]=[Z:Y-Z\sqrt{abc}:X]$.
	That transformation maps $P_{\infty}$, which has coordinates $[0:1:0]$, to itself  and $P_0$ to $[1:0:0]$.
	Thus, the reflections in $P_{\infty}$ and $P_0$ should represent the vertical and horizontal switches in the new coordinates, when the new affine chart is completed to $\mathbb{P}^1\times\mathbb{P}^1$.
	
The curve in the homogeneous coordinates is:
$$
Y^2Z=(aZ-X)(bZ-X)(cZ-X),
$$
and after the transformation:
$$
(Y_1+X_1\sqrt{abc})^2 X_1=(aX_1-Z_1)(bX_1-Z_1)(cX_1-Z_1).
$$
	
In the affine chart $Z_1=1$, the equation is:
	$$
	X_1Y_1^2+2X_1^2Y_1\sqrt{abc}+(ab+ac+bc)X_1^2-(a+b+c)X_1+1=0.
	$$
	The corresponding matrix is:
	$$
	M_Q=\begin{pmatrix}
		0 & 2\sqrt{abc} & ab+ac+bc\\
		1 & 0 & -(a+b+c)\\
		0 & 0 & 1
	\end{pmatrix}.
	$$
The determinant of this matrix equals to $-2\sqrt{abc}\neq0$, thus the QRT transformation is not of order $2$, which agrees with the initial observation that $2(P_0-P_{\infty})\not\sim0$.
\end{example}

From Examples \ref{ex:order4} and \ref{ex:curve-trans} we conclude:

\begin{proposition}\label{prop:order4} In the projective plane, consider the cubic curve:
	$$
	y^2=(a-x)(b-x)(c-x),
	\quad\text{with}
	\quad
	a\neq b\neq c\neq a.
	$$
Let  $P_0$ be its point with coordinates $(0,\sqrt{abc})$, $P_{\infty}$ the point at infinity, and
$Q$ a $(2,2)$ correspondence determined by central reflections in $P_0$ and $P_{\infty}$.
Then the determinant of $M_Q$ \eqref{eq:MQ} is zero if and only if $P_0$ is a branch point of the cubic curve.
%Thus, the group of random walks associated with the $(2,2)$ correspondence $Q$ corresponding to the cubic curve is of order four if and only if  $P_0$ is a branch point of the cubic curve.
\end{proposition}

\begin{example}[A two-coupled processor model, \cite{FayIas0}]\label{ex:que}
A classical example from queuing theory is about two parallel queues with infinite capacities, where arrivals are two independent Poisson processes with parameters $\lambda_1$ and $\lambda_2$. Service times are distributed exponentially with instantaneous service rates $S_1$ and $S_2$, so that when both queues are busy we have $S_i=\mu_i$ for $i=1,2$; when queue $2$ is empty, $S_1=\mu_1^*$;  when queue $1$ is empty, $S_2=\mu_2^*$. The service is first-in-first-out in both queues.

The evolution of the system is described by a two-dimensional continuous Markov process. Its probabilistic kernel is
$$
\frac{xyT(x, y)}{\lambda_1+\lambda_2+\mu_1+\mu_2},
$$
where
$$
T(x, y)=\lambda_1(1-x)+\lambda_2(1-y) + \mu_1\Big(1-\frac{1}{x}\Big)+\mu_2\Big(1-\frac{1}{y}\Big).
$$
The curve $xyT(x, y)=0$ is non-singular for $\mu_{1,2}\ne 0$, $\lambda_{1,2}\ne 0$ and $\lambda_1\ne \mu_1$ or $\lambda_2\ne \mu_2$.
It is known that the group of random walk of the curve $xyT(x, y)=0$ is of order four, see e.g.~\cite{RandomWalks}.
We can verify this using Proposition \ref{prop:order4} as well as Corollary \ref{cor:small} (i) and Theorem \ref{th:cayley} (b(i)), since $\det M_Q=0$ for:
$$
	M_Q=\begin{pmatrix}
		0 & -\mu_1 & 0\\
		-\mu_2 & \lambda_1+\lambda_2+\mu_1+\mu_2 & -\lambda_2\\
		0 & -\lambda_1 & 0
	\end{pmatrix}.
	$$

\end{example}

\begin{example} \label{ex:rw4}
It was shown in \cite{RandomWalks}*{Proposition 4.1.7 and equation (4.1.17)} that for the random walks, the group of random walk is of order four if and only if $\det P=0$ for  $P$ given by \eqref{eq:P}. For random walks, this coincides with conditions given in Corollary \ref{cor:small} (i) and Theorem \ref{th:cayley} (b(i)).
\end{example}

\subsection{Groups of order $6$}

Motivated by \cite{RandomWalks}, see Remark \ref{rem:rw6} below, we want to provide a direct proof of the following:

\begin{proposition}[\cite{RandomWalks}]\label{prop:6}
Let $\mathcal C$ be a smooth biquadratic curve given by its affine equation \eqref{eq:biquad}.
The group generated by its horizontal and vertical switches is of order $6$ if and only if
$\det \Delta_Q=0$, with
\begin{equation}\label{eq:DeltaQ}
	\Delta_Q
	=
\begin{pmatrix}
	\Delta_{11} & \Delta_{21} & \Delta_{12} & \Delta_{22} \\
	\Delta_{12} & \Delta_{22} & \Delta_{13} & \Delta_{23} \\
	\Delta_{21} & \Delta_{31} & \Delta_{22} & \Delta_{32} \\
	\Delta_{22} & \Delta_{32} & \Delta_{23} & \Delta_{33} \\
\end{pmatrix},
\end{equation}
where $\Delta_{ij}$ are the cofactors of the matrix $M_Q$, given by \eqref{eq:MQ}.
\end{proposition}
\begin{proof}
According to Corollary \ref{cor:small}, the group is of order $6$ if and only if $C_2=0$.
Note that the explicit expression for $C_2$ is also given in that corollary, and one can directly show that:
\begin{equation}\label{eq:C2}
C_2=\frac{2\det \Delta_Q}{\det(a_{ij})^3},
\end{equation}
which completes the proof.
\end{proof}

\begin{example}
Using notation from Example \ref{ex:curve-trans}, we get that $3P_0\sim3P_{\infty}$ is equivalent to $C_2=0$, where:
$$
\sqrt{(a-x)(b-x)(c-x)}=C_0+C_1x+C_2x^2+\dots
$$
is the Taylor series about $x=0$.
A straightforward calculation gives:
$$
C_2=\frac{4 a b c (a+b+c)-(a b+a c+b c)^2}{4(abc)^{3/2}}.
$$

Let $\Delta_{ij}$ be cofactors of the $3\times3$ matrix from Example \ref{ex:curve-trans}.
Then:
$$
\begin{vmatrix}
	\Delta_{11} & \Delta_{21} & \Delta_{12} & \Delta_{22} \\
	\Delta_{12} & \Delta_{22} & \Delta_{13} & \Delta_{23} \\
	\Delta_{21} & \Delta_{31} & \Delta_{22} & \Delta_{32} \\
	\Delta_{22} & \Delta_{32} & \Delta_{23} & \Delta_{33} \\
\end{vmatrix}
=
4 a b c (a+b+c)-(a b+a c+b c)^2.
$$
\end{example}

\begin{remark}\label{rem:rw6}
In \cite{RandomWalks}*{Proposition 4.1.8}, it was shown that the group of random walks is of order $6$ if and only if $\det \Delta_Q=0$, with $\Delta_Q$ given by \eqref{eq:DeltaQ}.
That proof uses additional assumptions on the entries of matrix $P$, from \eqref{eq:P}, that follow from their probabilistic nature.
In our proofs of Proposition \ref{prop:6}, we do not use those additional assumptions.
\end{remark}

\subsection{Groups of order $8$}

 We are now going to describe the biquadratic curves that have groups generated by horizontal and vertical switches of order $8$.

\begin{proposition} \label{prop:order8dir}
Let $\mathcal C$ be a biquadratic curve given by \eqref{eq:biquad}.
Then the the group generated by its horizontal and vertical switches is of order $8$ if and only if
\begin{align*}
4608\det(a_{ij})^4
=&
\frac{1}{12}
\Big(
\left(8 a_{00} a_{22}-4 a_{01} a_{21}+8 a_{02} a_{20}-4 a_{10} a_{12}+a_{11}^2
\right)^2
\\&\quad\quad
-
\left(
4 (a_{00} a_{22}+a_{01} a_{21}+a_{02} a_{20})-2 a_{10} a_{12}-a_{11}^2
\right)^2
\\&\quad\quad
+12 (a_{10} a_{11}-2 (a_{00} a_{21}+a_{01} a_{20}))
(a_{11} a_{12}-2 (a_{01} a_{22}+a_{02} a_{21}))
\\&\quad\quad
-12 \left(a_{10}^2-4 a_{00} a_{20}\right)
 \left(a_{12}^2-4 a_{02} a_{22}\right)
 \Big)
 \times
\\&\
\times
\Bigg(
576 (\det(a_{ij}))^2
\left({8 a_{00} a_{22}-4 a_{01} a_{21}+8 a_{02} a_{20}-4 a_{10} a_{12}+a_{11}^2}\right)
\\&\quad\quad
-\Big(
\left(
8 a_{00} a_{22}-4 a_{01} a_{21}+8 a_{02} a_{20}-4 a_{10} a_{12}+a_{11}^2
\right)^2
\\&\quad\quad\quad\quad
-\left(4 (a_{00} a_{22}+a_{01} a_{21}+a_{02} a_{20})-2 a_{10} a_{12}-a_{11}^2\right)^2
\\&\quad\quad\quad\quad
+12 (a_{10} a_{11}-2 (a_{00} a_{21}+a_{01} a_{20})) (a_{11} a_{12}-2 (a_{01} a_{22}+a_{02} a_{21}))
\\&\quad\quad\quad\quad
-12 \left(a_{10}^2-4 a_{00} a_{20}\right) \left(a_{12}^2-4 a_{02} a_{22}\right)
\Big)^2
\Bigg).
\end{align*}
\end{proposition}
\begin{proof}
According to Corollary \ref{cor:small}, the group is of order $8$ if and only if $C_3=0$.
The expression for $C_3$ is also given in Corollary \ref{cor:small}, so the statement is obtained by substituting $D_{\mathcal{C}}$ and $E_{\mathcal{C}}$, which are given in Theorem \ref{th:cayley}, and $X$, given in \eqref{eq:XY}.
\end{proof}

In \cite{RandomWalks}*{Proposition 4.1.11}, it was proved that the group of the random walk is of order $8$ if and only if $\det \Omega_Q=0$, where:
\begin{equation}\label{eq:omega}
\Omega_Q=
\left(
	\begin{array}{ccc}
		M_1 & M_2 & M_3 \\
		\Delta_{32}^2-\Delta_{31} \Delta_{33} & \Delta_{21} \Delta_{33}-2 \Delta_{22} \Delta_{32}+\Delta_{23} \Delta_{31} & \Delta_{22}^2-\Delta_{21} \Delta_{23} \\
		\Delta_{22}^2-\Delta_{21} \Delta_{23} & \Delta_{11} \Delta_{23}-2 \Delta_{12} \Delta_{22}+\Delta_{13} \Delta_{21} & \Delta_{12}^2-\Delta_{11} \Delta_{13} \\
	\end{array}
	\right),
\end{equation}
with
	\begin{align*}
		&M_1=-\Delta_{21} \Delta_{33}+2 \Delta_{22} \Delta_{32}-\Delta_{23} \Delta_{31},
		\\
		&M_2=\Delta_{11} \Delta_{33}-2 \left(\Delta_{12} \Delta_{32}-\Delta_{21} \Delta_{23}+\Delta_{22}^2\right)+\Delta_{13} \Delta_{31},
		\\
		&M_3=-\Delta_{11} \Delta_{23}+2 \Delta_{12} \Delta_{22}-\Delta_{13} \Delta_{21},
	\end{align*}
	and $\Delta_{ij}$ being the cofactors of the matrix $M_Q$, given by \eqref{eq:MQ}.

While the proof in \cite{RandomWalks} relies on the specific properties of the coefficients of the biquadratic, we provide here another proof, that holds for arbitrary non-singular biquadratic curve.

\begin{proposition}\label{prop:order8}
The group generated by the horizontal and vertical switches of the biquadratic curve $\mathcal{C}_A$ \eqref{eq:biquad} is of order $8$ if and only if the determinant of the matrix $\Omega_Q$ \eqref{eq:omega} vanishes.
\end{proposition}
\begin{proof}
According to Corollary \ref{cor:small}, the group is of order $8$ if and only if $C_3=0$.
Using the expression for $C_3$ from that Corollary \ref{cor:small} and substituting $D_{\mathcal{C}}$ and $E_{\mathcal{C}}$, which are given in Theorem \ref{th:cayley}, and $X_0$, given in \eqref{eq:XY}, into it, we calculate:
\begin{equation}\label{eq:C3}
C_3=-\frac{2\det\Omega_Q}{\det(a_{ij})^5},
\end{equation}
which immediately implies the statement.
\end{proof}

\begin{example}
Consider random walk with the following matrix $P$ \eqref{eq:P}:
$$
P=\left(
\begin{array}{ccc}
	\frac{1}{4}-\frac{1}{2}  \sqrt{\sqrt{5}-2} & 0 & \frac{1}{4} \\
	0 & -1 & 0 \\
	\frac{1}{2} \sqrt{\sqrt{5}-2}+\frac{1}{4} & 0 & \frac{1}{4} \\
\end{array}
\right).
$$
The corresponding cubic curve   $\Gamma$, given by \eqref{eq:cubic}, is:
$$
\Gamma: y^2=4x^3-\frac{1}{12} \left(7-3 \sqrt{5}\right)x+\frac{1}{432} \left(9 \sqrt{5}-20\right).
$$
Note that the curve $\Gamma$ is smooth, since the cubic polynomial in $x$ on the righthand side of its equation has three distinct zeroes:
$$
-\dfrac{1}{12},
\quad
\dfrac{7-3 \sqrt{5}}{24},
\quad
\dfrac{3 \sqrt{5}-5}{24}.
$$
Form \eqref{eq:XY}, we have:
$$
X_0=\frac16,
\quad
Y_0=\frac{\sqrt{\sqrt{5}-2}}{4}.
$$

One can calculate directly that $C_3=0$ using the formula from Proposition \ref{prop:order8dir}, thus the group is of order $8$.
On the other hand, the cofactors are:
$$
\left(
\begin{array}{ccc}
	-\frac{1}{4} & 0 & \frac12\sqrt{\sqrt{5}-2}+\frac{1}{4} \\
	0 & -\frac{1}{4} \sqrt{\sqrt{5}-2} & 0 \\
	\frac{1}{4} & 0 & \frac12\sqrt{\sqrt{5}-2}-\frac{1}{4} \\
\end{array}
\right),
$$
and the matrix $\Omega_Q$ \eqref{eq:omega} is:
$$
\Omega_Q=\frac1{16}\left(
\begin{array}{ccc}
	0 & 2\left(3-\sqrt{5}\right) & 0
	\\
	1-2 \sqrt{\sqrt{5}-2} & 0 & \sqrt{5}-2
	\\
	\sqrt{5}-2 & 0 & 2 \sqrt{\sqrt{5}-2}+1
\end{array}
\right).
$$
The determinant of that matrix equals $0$, in accordance with Proposition \ref{prop:order8}.
\end{example}

\subsection{Groups of order 10}

Here we look at random walks with groups of order $10$ more closely, derive a new characterization and provide new examples.

\begin{proposition} \label{prop:order10dir}
Let $\mathcal C$ be a biquadratic curve given by \eqref{eq:biquad}.
Then we have:
\begin{itemize}
\item[(i)] The group generated by the horizontal and vertical switches of the curve is of order $10$ if and only if
\begin{equation}\label{eq:cay10}
	\frac{2\det\Omega_Q^2}{\det\Delta_Q\det(a_{ij})^7}=C_4,
\end{equation}
where $\Omega_Q$ is given in \eqref{eq:omega}, $\Delta_Q$ is given in \eqref{eq:DeltaQ}, and
\begin{equation}\label{eq:C4}
	C_4=
	-\frac{5}{8\det(a_{ij})^7}
	\hat X^4 +\frac{3}{4\det(a_{ij})^5} \hat B_1\hat X^2
	+\frac1{\det(a_{ij})^3}\hat C_1.
\end{equation}
where
\begin{align*}
	\hat X=\;&a_{02} \left(2 a_{00} a_{21}^2-4 a_{00} a_{20} a_{22}-2 a_{01} a_{20} a_{21}+2 a_{10}^2 a_{22}-a_{10} a_{11} a_{21}-2 a_{10} a_{12} a_{20}+a_{11}^2 a_{20}\right)
	\\
	&-a_{01} (2 a_{00} a_{21} a_{22}+a_{10} a_{11} a_{22}-2 a_{10} a_{12} a_{21}+a_{11} a_{12} a_{20})+2 a_{01}^2 a_{20} a_{22}+2 a_{02}^2 a_{20}^2
	\\
	&+a_{00}
	\left(a_{22} \left(2 a_{00} a_{22}+a_{11}^2\right)-a_{12} (2 a_{10} a_{22}+a_{11} a_{21})+2 a_{12}^2 a_{20}
	\right);
	\\
	\hat B_1=\;&8 a_{00} a_{22}-4 a_{01} a_{21}+8 a_{02} a_{20}-4 a_{10} a_{12}+a_{11}^2;
	\\
	\hat C_1=\;&
	a_{11}^2 (a_{01} a_{21}-4 a_{00} a_{22}-4 a_{02} a_{20}+a_{10} a_{12})-\frac{a_{11}^4}{8}
	\\&
	+2 a_{11} (a_{00} a_{12} a_{21}+a_{01} a_{10} a_{22}+a_{01} a_{12} a_{20}
	+a_{02} a_{10} a_{21})
	\\
	&-2 \Big(
	6 a_{00}^2 a_{22}^2
	-2 a_{10} a_{12} (3 a_{00} a_{22}-2 a_{01} a_{21}+3 a_{02} a_{20})
	\\&\qquad
	+2 a_{00}
	\left(2 a_{02} a_{20} a_{22}-3 a_{01} a_{21} a_{22}
	+a_{02} a_{21}^2+a_{12}^2 a_{20}
	\right)
	\\
	&\qquad
	+2 a_{01}^2 a_{20} a_{22}+a_{01}^2 a_{21}^2
	-6 a_{01} a_{02} a_{20} a_{21}+6 a_{02}^2 a_{20}^2+a_{10}^2 \left(2 a_{02} a_{22}+a_{12}^2\right)
	\Big).
\end{align*}
\item[(ii)]
Equivalently, the group is of order $10$ if and only if
\begin{equation}\label{eq:hatX}
	5\hat X^2
	=
	3\det(a_{ij})\hat B_1
	\pm
	\sqrt{
		9\det(a_{ij})^2\hat B_1^2
		-
		40\left(\frac{2\det(\Omega_Q)^2}{\det(\Delta_Q)}-\det(a_{ij})^4\hat C_1\right)
	}.
\end{equation}
\item[(iii)]
If the coefficients $a_{ij}$ are real and the group is of order $10$, then the following inequalities are satisfied:
\begin{equation}\label{eq:dis}
	\det(a_{ij})^2\hat B_1^2
	\ge
	\frac{40}{9}
	\left(\frac{2\det(\Omega_Q)^2}{\det(\Delta_Q)}-\det(a_{ij})^4\hat C_1\right),
\end{equation}
\begin{equation*}
	3\det(a_{ij})\hat B_1\pm \sqrt{9\det(a_{ij})^2\hat B_1^2-40\Big(\frac{2\det(\Omega_Q)^2}{\det(\Delta_Q)}-\det(a_{ij})^4\hat C_1\Big)}\ge 0.
\end{equation*}
Moreover, the equation
$$
\frac{5}{8}
\hat Y^2 -\frac{3}{4}\det(a_{ij})^2 \hat B_1\hat Y
-\det(a_{ij})^4\hat C_1+\frac{2\det(\Omega_Q)^2}{\det(\Delta_Q)}=0,
$$
has at least one nonnegative solution $\hat Y$ if and only if the inequality \eqref{eq:dis} is satisfied and
\begin{equation}\label{eq:ineq23}
	\hat B_1\ge 0 \quad\text{or}\quad \hat C_1\ge\frac{2\det(\Omega_Q)^2}{\det(\Delta_Q)\det(a_{ij})^4}.
\end{equation}
\end{itemize}
\end{proposition}
\begin{proof}
According to Corollary \ref{cor:small}, the group is of order $10$ if and only if $C_3^2=C_2C_4$.
Substituting there the expressions \eqref{eq:C2} and \eqref{eq:C3} for $C_2$ and $C_3$, we get the condition \eqref{eq:cay10}.
Relation \eqref{eq:C4} is obtained from Corollary \ref{cor:small} and substituting $D_{\mathcal{C}}$ and $E_{\mathcal{C}}$, which are given in Theorem \ref{th:cayley}, and $X_0$, given in \eqref{eq:XY}, into it.
\end{proof}

\begin{example}\label{ex:order10}
Consider random walk with the following matrix P \eqref{eq:P}:
	$$
	P=\left(
	\begin{array}{ccc}
		\frac{1}{4}-\alpha & 0 & \frac{1}{4} \\
		0 & -1 & 0 \\
		\frac14+\alpha & 0 & \frac{1}{4} \\
	\end{array}
	\right),
	\quad\text{where}\quad |\alpha|\le\frac14.
	$$
The transition probabilities for such random walks are:
$$
p_{11}=\frac14-\alpha,
\quad
p_{10}=p_{01}=p_{0,-1}=p_{-1,0}=p_{00}=0,
\quad
p_{1,-1}=p_{-1,-1}=\frac14,
\quad
p_{-1,1}=\frac14+\alpha.
$$
The condition for $5$-periodicity of the corresponding QRT map, according to Corollary \ref{cor:small}, is:
$$
\left(64 \alpha^6-192 \alpha^5-112 \alpha^4-32 \alpha^3+28 \alpha^2+4 \alpha-1\right)
\left(64 \alpha^6+192 \alpha^5-112 \alpha^4+32 \alpha^3+28 \alpha^2-4 \alpha-1\right)
=
0.
$$	
Note that, for
\begin{equation}\label{eq:P5}
P_5(x)
=
64 x^6-192 x^5-112 x^4-32 x^3+ x^2+4 x-1
\end{equation}
 we have that
$P_5(-1/4)$, $P_5(1/4)$, $P_5(4)$ are positive, $P_5(0)$, $P_5(1)$ are negative, while the discriminant of the polynomial is also negative.
That means the polynomial has exactly four real roots, each in one of the intervals $(-1/4,0)$, $(0,1/4)$, $(1/4,1)$, $(1,4)$, and two complex conjugated roots.
We can conclude that we obtained two random walks with the groups of order $10$, corresponding to the two real values of $\alpha$ for which $P_5(\alpha)=0$ and $|\alpha|<1/4$.

Now, suppose that $\alpha$ is the root of $P_5$ in $(0,1/4)$.
The corresponding biquadratic curve is:
\begin{equation}\label{eq:biquad10}
Q(x,y)
=
\left(\frac{1}{4}-\alpha\right) x^2 y^2
+
\left(\alpha+\frac{1}{4}\right) y^2
+\frac{x^2}{4}
-x y
+\frac{1}{4}
=
0.
\end{equation}
A direct check shows that this curve is smooth and of genus $1$, and that its real part has two connected components, which are symmetric to each other with respect to the origin, as shown in the lefthand side of Figure \ref{fig:Q10}.
Notice that each of the two components is invariant for the QRT transformation.
In the righthand side of Figure \ref{fig:Q10}, the lower left component of the curve is zoomed in and the orbit of one point for the group of the random walk is represented.
The case when $\alpha$ is the root of $P_5$ which lies in $(-1/4,0)$ can be discussed similarly.
\end{example}
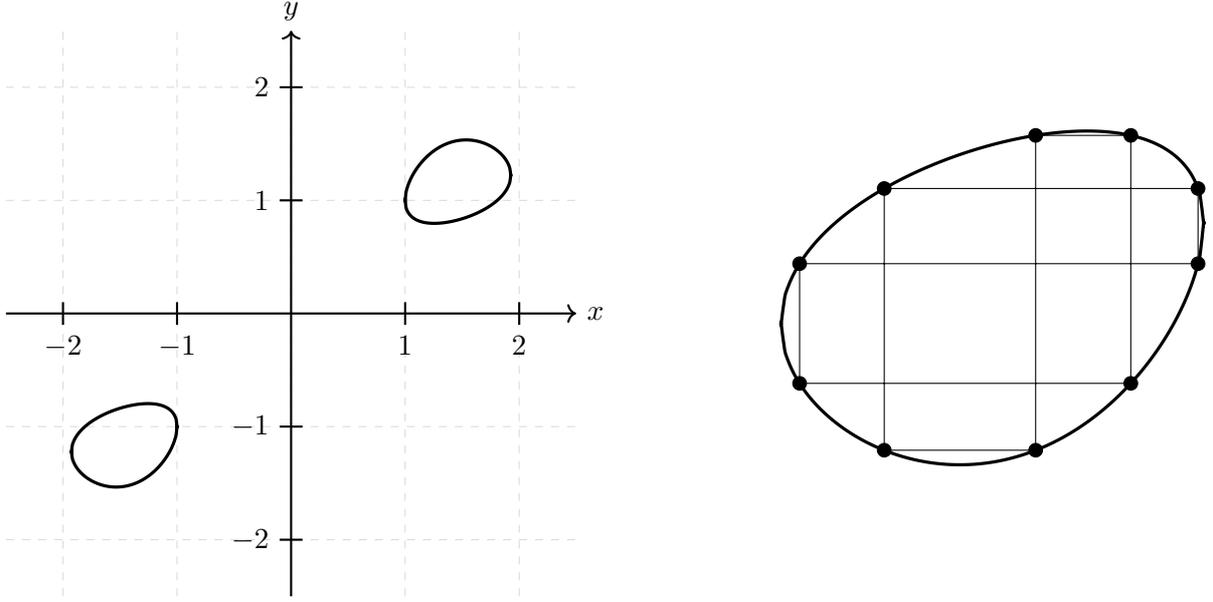
\begin{figure}[h]
	\centering
	\begin{tikzpicture}[scale=1.5]
		
\draw[very thin,color=gray!30,dashed] (-2.5,-2.5) grid (2.5,2.5);

\draw[thick,->] (-2.5,0) -- (2.5,0) node[right] {$x$};
\draw[thick,->] (0,-2.5) -- (0,2.5) node[above] {$y$};

\draw[thick](0.1,1)--(-0.1,1)node[left]{$1$};
\draw[thick](0.1,2)--(-0.1,2)node[left]{$2$};
\draw[thick](0.1,-1)--(-0.1,-1)node[left]{$-1$};
\draw[thick](0.1,-2)--(-0.1,-2)node[left]{$-2$};

\draw[thick](1,0.1)--(1,-0.1)node[below]{$1$};
\draw[thick](2,0.1)--(2,-0.1)node[below]{$2$};
\draw[thick](-1,0.1)--(-1,-0.1)node[below]{$-1$};
\draw[thick](-2,0.1)--(-2,-0.1)node[below]{$-2$};

\draw [very thick] plot [smooth cycle] coordinates{(-1.92576,-1.22252)(-1.91805,-1.16397)(-1.91033,-1.1399)(-1.90262,-1.12155)(-1.89491,-1.10618)(-1.88719,-1.09273)(-1.87948,-1.08063)(-1.87176,-1.06958)(-1.86405,-1.05936)(-1.85633,-1.04982)(-1.84862,-1.04085)(-1.8409,-1.03237)(-1.83319,-1.02431)(-1.82547,-1.01664)(-1.81776,-1.0093)(-1.81004,-1.00226)(-1.80233,-0.995501)(-1.79461,-0.988989)(-1.7869,-0.982708)(-1.77918,-0.976639)(-1.77147,-0.970767)(-1.76376,-0.96508)(-1.75604,-0.959564)(-1.74833,-0.954211)(-1.74061,-0.94901)(-1.7329,-0.943954)(-1.72518,-0.939034)(-1.71747,-0.934244)(-1.70975,-0.929579)(-1.70204,-0.925032)(-1.69432,-0.920598)(-1.68661,-0.916274)(-1.67889,-0.912054)(-1.67118,-0.907936)(-1.66346,-0.903915)(-1.65575,-0.899988)(-1.64803,-0.896153)(-1.64032,-0.892406)(-1.63261,-0.888746)(-1.62489,-0.88517)(-1.61718,-0.881675)(-1.60946,-0.878261)(-1.60175,-0.874924)(-1.59403,-0.871664)(-1.58632,-0.868479)(-1.5786,-0.865367)(-1.57089,-0.862328)(-1.56317,-0.85936)(-1.55546,-0.856463)(-1.54774,-0.853634)(-1.54003,-0.850874)(-1.53231,-0.848182)(-1.5246,-0.845556)(-1.51689,-0.842997)(-1.50917,-0.840504)(-1.50146,-0.838077)(-1.49374,-0.835715)(-1.48603,-0.833417)(-1.47831,-0.831185)(-1.4706,-0.829017)(-1.46288,-0.826914)(-1.45517,-0.824875)(-1.44745,-0.822902)(-1.43974,-0.820993)(-1.43202,-0.81915)(-1.42431,-0.817372)(-1.41659,-0.81566)(-1.40888,-0.814015)(-1.40116,-0.812438)(-1.39345,-0.810928)(-1.38574,-0.809486)(-1.37802,-0.808115)(-1.37031,-0.806813)(-1.36259,-0.805583)(-1.35488,-0.804426)(-1.34716,-0.803343)(-1.33945,-0.802335)(-1.33173,-0.801403)(-1.32402,-0.800551)(-1.3163,-0.799778)(-1.30859,-0.799088)(-1.30087,-0.798482)(-1.29316,-0.797962)(-1.28544,-0.797532)(-1.27773,-0.797193)(-1.27001,-0.796949)(-1.2623,-0.796803)(-1.25459,-0.796758)(-1.24687,-0.796818)(-1.23916,-0.796987)(-1.23144,-0.797269)(-1.22373,-0.79767)(-1.21601,-0.798194)(-1.2083,-0.798847)(-1.20058,-0.799636)(-1.19287,-0.800568)(-1.18515,-0.801651)(-1.17744,-0.802892)(-1.16972,-0.804302)(-1.16201,-0.805891)(-1.15429,-0.807672)(-1.14658,-0.809658)(-1.13886,-0.811864)(-1.13115,-0.814308)(-1.12344,-0.817011)(-1.11572,-0.819995)(-1.10801,-0.823288)(-1.10029,-0.826923)(-1.09258,-0.830939)(-1.08486,-0.835382)(-1.07715,-0.84031)(-1.06943,-0.845797)(-1.06172,-0.851936)(-1.054,-0.85885)(-1.04629,-0.86671)(-1.03857,-0.875757)(-1.03086,-0.886366)
(-1.02314,-0.89916)(-1.01543,-0.915355)(-1.00771,-0.938075)(-1.,-1.)(-1.,-1.)(-1.00771,-1.07075)(-1.01543,-1.10218)(-1.02314,-1.12698)(-1.03086,-1.14827)(-1.03857,-1.16727)(-1.04629,-1.1846)(-1.054,-1.20064)(-1.06172,-1.21563)(-1.06943,-1.22973)(-1.07715,-1.24308)(-1.08486,-1.25577)(-1.09258,-1.26786)(-1.10029,-1.27943)(-1.10801,-1.29051)(-1.11572,-1.30115)(-1.12344,-1.31138)(-1.13115,-1.32123)(-1.13886,-1.33071)(-1.14658,-1.33986)(-1.15429,-1.34869)(-1.16201,-1.35722)(-1.16972,-1.36545)(-1.17744,-1.37341)(-1.18515,-1.3811)(-1.19287,-1.38854)(-1.20058,-1.39573)(-1.2083,-1.40268)(-1.21601,-1.4094)(-1.22373,-1.4159)(-1.23144,-1.42218)(-1.23916,-1.42825)(-1.24687,-1.43411)(-1.25459,-1.43977)(-1.2623,-1.44524)(-1.27001,-1.45052)(-1.27773,-1.4556)(-1.28544,-1.46051)(-1.29316,-1.46523)(-1.30087,-1.46978)(-1.30859,-1.47415)(-1.3163,-1.47835)(-1.32402,-1.48238)(-1.33173,-1.48625)(-1.33945,-1.48996)(-1.34716,-1.4935)(-1.35488,-1.49688)(-1.36259,-1.50011)(-1.37031,-1.50318)(-1.37802,-1.50611)(-1.38574,-1.50887)(-1.39345,-1.51149)(-1.40116,-1.51396)(-1.40888,-1.51629)(-1.41659,-1.51847)(-1.42431,-1.5205)(-1.43202,-1.52239)(-1.43974,-1.52414)(-1.44745,-1.52575)(-1.45517,-1.52722)(-1.46288,-1.52855)(-1.4706,-1.52974)(-1.47831,-1.5308)(-1.48603,-1.53171)(-1.49374,-1.53249)(-1.50146,-1.53314)(-1.50917,-1.53365)(-1.51689,-1.53402)(-1.5246,-1.53426)(-1.53231,-1.53436)(-1.54003,-1.53433)(-1.54774,-1.53416)(-1.55546,-1.53386)(-1.56317,-1.53343)(-1.57089,-1.53285)(-1.5786,-1.53215)(-1.58632,-1.5313)(-1.59403,-1.53032)(-1.60175,-1.5292)(-1.60946,-1.52794)(-1.61718,-1.52654)(-1.62489,-1.52501)(-1.63261,-1.52333)(-1.64032,-1.5215)(-1.64803,-1.51953)(-1.65575,-1.51741)(-1.66346,-1.51515)(-1.67118,-1.51273)(-1.67889,-1.51015)(-1.68661,-1.50742)(-1.69432,-1.50453)(-1.70204,-1.50148)(-1.70975,-1.49826)(-1.71747,-1.49486)(-1.72518,-1.49129)(-1.7329,-1.48753)(-1.74061,-1.48359)(-1.74833,-1.47945)(-1.75604,-1.4751)(-1.76376,-1.47055)(-1.77147,-1.46577)(-1.77918,-1.46076)(-1.7869,-1.4555)(-1.79461,-1.44998)(-1.80233,-1.44418)(-1.81004,-1.43809)(-1.81776,-1.43167)(-1.82547,-1.42491)(-1.83319,-1.41776)(-1.8409,-1.41019)(-1.84862,-1.40215)(-1.85633,-1.39357)(-1.86405,-1.38438)(-1.87176,-1.37447)(-1.87948,-1.36368)(-1.88719,-1.35181)(-1.89491,-1.33854)(-1.90262,-1.32331)(-1.91033,-1.30506)(-1.91805,-1.28105)(-1.92576,-1.22252)}
;

\draw [very thick] plot [smooth cycle] coordinates{(1.92576,1.22252)(1.91805,1.16397)(1.91033,1.1399)(1.90262,1.12155)(1.89491,1.10618)(1.88719,1.09273)(1.87948,1.08063)(1.87176,1.06958)(1.86405,1.05936)(1.85633,1.04982)(1.84862,1.04085)(1.8409,1.03237)(1.83319,1.02431)(1.82547,1.01664)(1.81776,1.0093)(1.81004,1.00226)(1.80233,0.995501)(1.79461,0.988989)(1.7869,0.982708)(1.77918,0.976639)(1.77147,0.970767)(1.76376,0.96508)(1.75604,0.959564)(1.74833,0.954211)(1.74061,0.94901)(1.7329,0.943954)(1.72518,0.939034)(1.71747,0.934244)(1.70975,0.929579)(1.70204,0.925032)(1.69432,0.920598)(1.68661,0.916274)(1.67889,0.912054)(1.67118,0.907936)(1.66346,0.903915)(1.65575,0.899988)(1.64803,0.896153)(1.64032,0.892406)(1.63261,0.888746)(1.62489,0.88517)(1.61718,0.881675)(1.60946,0.878261)(1.60175,0.874924)(1.59403,0.871664)(1.58632,0.868479)(1.5786,0.865367)(1.57089,0.862328)(1.56317,0.85936)(1.55546,0.856463)(1.54774,0.853634)(1.54003,0.850874)(1.53231,0.848182)(1.5246,0.845556)(1.51689,0.842997)(1.50917,0.840504)(1.50146,0.838077)(1.49374,0.835715)(1.48603,0.833417)(1.47831,0.831185)(1.4706,0.829017)(1.46288,0.826914)(1.45517,0.824875)(1.44745,0.822902)(1.43974,0.820993)(1.43202,0.81915)(1.42431,0.817372)(1.41659,0.81566)(1.40888,0.814015)(1.40116,0.812438)(1.39345,0.810928)(1.38574,0.809486)(1.37802,0.808115)(1.37031,0.806813)(1.36259,0.805583)(1.35488,0.804426)(1.34716,0.803343)(1.33945,0.802335)(1.33173,0.801403)(1.32402,0.800551)(1.3163,0.799778)(1.30859,0.799088)(1.30087,0.798482)(1.29316,0.797962)(1.28544,0.797532)(1.27773,0.797193)(1.27001,0.796949)(1.2623,0.796803)(1.25459,0.796758)(1.24687,0.796818)(1.23916,0.796987)(1.23144,0.797269)(1.22373,0.79767)(1.21601,0.798194)(1.2083,0.798847)(1.20058,0.799636)(1.19287,0.800568)(1.18515,0.801651)(1.17744,0.802892)(1.16972,0.804302)(1.16201,0.805891)(1.15429,0.807672)(1.14658,0.809658)(1.13886,0.811864)(1.13115,0.814308)(1.12344,0.817011)(1.11572,0.819995)(1.10801,0.823288)(1.10029,0.826923)(1.09258,0.830939)(1.08486,0.835382)(1.07715,0.84031)(1.06943,0.845797)(1.06172,0.851936)(1.054,0.85885)(1.04629,0.86671)(1.03857,0.875757)(1.03086,0.886366)(1.02314,0.89916)(1.01543,0.915355)(1.00771,0.938075)(1.,1.)(1.,1.)(1.00771,1.07075)
(1.01543,1.10218)(1.02314,1.12698)(1.03086,1.14827)(1.03857,1.16727)(1.04629,1.1846)(1.054,1.20064)(1.06172,1.21563)(1.06943,1.22973)(1.07715,1.24308)(1.08486,1.25577)(1.09258,1.26786)(1.10029,1.27943)(1.10801,1.29051)(1.11572,1.30115)(1.12344,1.31138)(1.13115,1.32123)(1.13886,1.33071)(1.14658,1.33986)(1.15429,1.34869)(1.16201,1.35722)(1.16972,1.36545)(1.17744,1.37341)(1.18515,1.3811)(1.19287,1.38854)(1.20058,1.39573)(1.2083,1.40268)(1.21601,1.4094)(1.22373,1.4159)(1.23144,1.42218)(1.23916,1.42825)(1.24687,1.43411)(1.25459,1.43977)(1.2623,1.44524)(1.27001,1.45052)(1.27773,1.4556)(1.28544,1.46051)(1.29316,1.46523)(1.30087,1.46978)(1.30859,1.47415)(1.3163,1.47835)(1.32402,1.48238)(1.33173,1.48625)(1.33945,1.48996)(1.34716,1.4935)(1.35488,1.49688)(1.36259,1.50011)(1.37031,1.50318)(1.37802,1.50611)(1.38574,1.50887)(1.39345,1.51149)(1.40116,1.51396)(1.40888,1.51629)(1.41659,1.51847)(1.42431,1.5205)(1.43202,1.52239)(1.43974,1.52414)(1.44745,1.52575)(1.45517,1.52722)(1.46288,1.52855)(1.4706,1.52974)(1.47831,1.5308)(1.48603,1.53171)(1.49374,1.53249)(1.50146,1.53314)(1.50917,1.53365)(1.51689,1.53402)(1.5246,1.53426)(1.53231,1.53436)(1.54003,1.53433)(1.54774,1.53416)(1.55546,1.53386)(1.56317,1.53343)(1.57089,1.53285)(1.5786,1.53215)(1.58632,1.5313)(1.59403,1.53032)(1.60175,1.5292)(1.60946,1.52794)(1.61718,1.52654)(1.62489,1.52501)(1.63261,1.52333)(1.64032,1.5215)(1.64803,1.51953)(1.65575,1.51741)(1.66346,1.51515)(1.67118,1.51273)(1.67889,1.51015)(1.68661,1.50742)(1.69432,1.50453)(1.70204,1.50148)(1.70975,1.49826)(1.71747,1.49486)(1.72518,1.49129)(1.7329,1.48753)(1.74061,1.48359)(1.74833,1.47945)(1.75604,1.4751)(1.76376,1.47055)(1.77147,1.46577)(1.77918,1.46076)(1.7869,1.4555)(1.79461,1.44998)(1.80233,1.44418)(1.81004,1.43809)(1.81776,1.43167)(1.82547,1.42491)(1.83319,1.41776)(1.8409,1.41019)(1.84862,1.40215)(1.85633,1.39357)(1.86405,1.38438)(1.87176,1.37447)(1.87948,1.36368)(1.88719,1.35181)(1.89491,1.33854)(1.90262,1.32331)(1.91033,1.30506)(1.91805,1.28105)(1.92576,1.22252)}
;

	\begin{scope}[scale=4,shift={(3,1.2)}]
	\coordinate(p1)at(-1.7,-1.5023);
	\coordinate(p2)at(-1.7,-0.92385);
	\coordinate(p3)at(-1.01216,-0.92385);
	\coordinate(p4)at(-1.01216,-1.09001);
	\coordinate(p5)at(-1.88552,-1.09001);
	\coordinate(p6)at(-1.88552,-1.35449);
	\coordinate(p7)at(-1.15951,-1.35449);
	\coordinate(p8)at(-1.15951,-0.806446);
	\coordinate(p9)at(-1.36805,-0.806446);
	\coordinate(p10)at(-1.36805,-1.5023);
	
	\draw[fill=black](p1) circle (0.015);		
	\draw[fill=black](p2) circle (0.015);		
	\draw[fill=black](p3) circle (0.015);		
	\draw[fill=black](p4) circle (0.015);		
	\draw[fill=black](p5) circle (0.015);		
	\draw[fill=black](p6) circle (0.015);		
	\draw[fill=black](p7) circle (0.015);		
	\draw[fill=black](p8) circle (0.015);		
	\draw[fill=black](p9) circle (0.015);		
	\draw[fill=black](p10) circle (0.015);

	\draw(p1)--(p2)--(p3)--(p4)--(p5)--(p6)--(p7)--(p8)--(p9)--(p10)--cycle;

\draw [very thick] plot [smooth cycle] coordinates{(-1.92576,-1.22252)(-1.91805,-1.16397)(-1.91033,-1.1399)(-1.90262,-1.12155)(-1.89491,-1.10618)(-1.88719,-1.09273)(-1.87948,-1.08063)(-1.87176,-1.06958)(-1.86405,-1.05936)(-1.85633,-1.04982)(-1.84862,-1.04085)(-1.8409,-1.03237)(-1.83319,-1.02431)(-1.82547,-1.01664)(-1.81776,-1.0093)(-1.81004,-1.00226)(-1.80233,-0.995501)(-1.79461,-0.988989)(-1.7869,-0.982708)(-1.77918,-0.976639)(-1.77147,-0.970767)(-1.76376,-0.96508)(-1.75604,-0.959564)(-1.74833,-0.954211)(-1.74061,-0.94901)(-1.7329,-0.943954)(-1.72518,-0.939034)(-1.71747,-0.934244)(-1.70975,-0.929579)(-1.70204,-0.925032)(-1.69432,-0.920598)(-1.68661,-0.916274)(-1.67889,-0.912054)(-1.67118,-0.907936)(-1.66346,-0.903915)(-1.65575,-0.899988)(-1.64803,-0.896153)(-1.64032,-0.892406)(-1.63261,-0.888746)(-1.62489,-0.88517)(-1.61718,-0.881675)(-1.60946,-0.878261)(-1.60175,-0.874924)(-1.59403,-0.871664)(-1.58632,-0.868479)(-1.5786,-0.865367)(-1.57089,-0.862328)(-1.56317,-0.85936)(-1.55546,-0.856463)(-1.54774,-0.853634)(-1.54003,-0.850874)(-1.53231,-0.848182)(-1.5246,-0.845556)(-1.51689,-0.842997)(-1.50917,-0.840504)(-1.50146,-0.838077)(-1.49374,-0.835715)(-1.48603,-0.833417)(-1.47831,-0.831185)(-1.4706,-0.829017)(-1.46288,-0.826914)(-1.45517,-0.824875)(-1.44745,-0.822902)(-1.43974,-0.820993)(-1.43202,-0.81915)(-1.42431,-0.817372)(-1.41659,-0.81566)(-1.40888,-0.814015)(-1.40116,-0.812438)(-1.39345,-0.810928)(-1.38574,-0.809486)(-1.37802,-0.808115)(-1.37031,-0.806813)(-1.36259,-0.805583)(-1.35488,-0.804426)(-1.34716,-0.803343)(-1.33945,-0.802335)(-1.33173,-0.801403)(-1.32402,-0.800551)(-1.3163,-0.799778)(-1.30859,-0.799088)(-1.30087,-0.798482)(-1.29316,-0.797962)(-1.28544,-0.797532)(-1.27773,-0.797193)(-1.27001,-0.796949)(-1.2623,-0.796803)(-1.25459,-0.796758)(-1.24687,-0.796818)(-1.23916,-0.796987)(-1.23144,-0.797269)(-1.22373,-0.79767)(-1.21601,-0.798194)(-1.2083,-0.798847)(-1.20058,-0.799636)(-1.19287,-0.800568)(-1.18515,-0.801651)(-1.17744,-0.802892)(-1.16972,-0.804302)(-1.16201,-0.805891)(-1.15429,-0.807672)(-1.14658,-0.809658)(-1.13886,-0.811864)(-1.13115,-0.814308)(-1.12344,-0.817011)(-1.11572,-0.819995)(-1.10801,-0.823288)(-1.10029,-0.826923)(-1.09258,-0.830939)(-1.08486,-0.835382)(-1.07715,-0.84031)(-1.06943,-0.845797)(-1.06172,-0.851936)(-1.054,-0.85885)(-1.04629,-0.86671)(-1.03857,-0.875757)(-1.03086,-0.886366)(-1.02314,-0.89916)(-1.01543,-0.915355)(-1.00771,-0.938075)(-1.,-1.)(-1.,-1.)(-1.00771,-1.07075)(-1.01543,-1.10218)(-1.02314,-1.12698)(-1.03086,-1.14827)(-1.03857,-1.16727)(-1.04629,-1.1846)(-1.054,-1.20064)(-1.06172,-1.21563)(-1.06943,-1.22973)(-1.07715,-1.24308)(-1.08486,-1.25577)
(-1.09258,-1.26786)(-1.10029,-1.27943)(-1.10801,-1.29051)(-1.11572,-1.30115)(-1.12344,-1.31138)(-1.13115,-1.32123)(-1.13886,-1.33071)(-1.14658,-1.33986)(-1.15429,-1.34869)(-1.16201,-1.35722)(-1.16972,-1.36545)(-1.17744,-1.37341)(-1.18515,-1.3811)(-1.19287,-1.38854)(-1.20058,-1.39573)(-1.2083,-1.40268)(-1.21601,-1.4094)(-1.22373,-1.4159)(-1.23144,-1.42218)(-1.23916,-1.42825)(-1.24687,-1.43411)(-1.25459,-1.43977)(-1.2623,-1.44524)(-1.27001,-1.45052)(-1.27773,-1.4556)(-1.28544,-1.46051)(-1.29316,-1.46523)(-1.30087,-1.46978)(-1.30859,-1.47415)(-1.3163,-1.47835)(-1.32402,-1.48238)(-1.33173,-1.48625)(-1.33945,-1.48996)(-1.34716,-1.4935)(-1.35488,-1.49688)(-1.36259,-1.50011)(-1.37031,-1.50318)(-1.37802,-1.50611)(-1.38574,-1.50887)(-1.39345,-1.51149)(-1.40116,-1.51396)(-1.40888,-1.51629)(-1.41659,-1.51847)(-1.42431,-1.5205)(-1.43202,-1.52239)(-1.43974,-1.52414)(-1.44745,-1.52575)(-1.45517,-1.52722)(-1.46288,-1.52855)(-1.4706,-1.52974)(-1.47831,-1.5308)(-1.48603,-1.53171)(-1.49374,-1.53249)(-1.50146,-1.53314)(-1.50917,-1.53365)(-1.51689,-1.53402)(-1.5246,-1.53426)(-1.53231,-1.53436)(-1.54003,-1.53433)(-1.54774,-1.53416)(-1.55546,-1.53386)(-1.56317,-1.53343)(-1.57089,-1.53285)(-1.5786,-1.53215)(-1.58632,-1.5313)(-1.59403,-1.53032)(-1.60175,-1.5292)(-1.60946,-1.52794)(-1.61718,-1.52654)(-1.62489,-1.52501)(-1.63261,-1.52333)(-1.64032,-1.5215)(-1.64803,-1.51953)(-1.65575,-1.51741)(-1.66346,-1.51515)(-1.67118,-1.51273)(-1.67889,-1.51015)(-1.68661,-1.50742)(-1.69432,-1.50453)(-1.70204,-1.50148)(-1.70975,-1.49826)(-1.71747,-1.49486)(-1.72518,-1.49129)(-1.7329,-1.48753)(-1.74061,-1.48359)(-1.74833,-1.47945)(-1.75604,-1.4751)(-1.76376,-1.47055)(-1.77147,-1.46577)(-1.77918,-1.46076)(-1.7869,-1.4555)(-1.79461,-1.44998)(-1.80233,-1.44418)(-1.81004,-1.43809)(-1.81776,-1.43167)(-1.82547,-1.42491)(-1.83319,-1.41776)(-1.8409,-1.41019)(-1.84862,-1.40215)(-1.85633,-1.39357)(-1.86405,-1.38438)(-1.87176,-1.37447)(-1.87948,-1.36368)(-1.88719,-1.35181)(-1.89491,-1.33854)(-1.90262,-1.32331)(-1.91033,-1.30506)(-1.91805,-1.28105)(-1.92576,-1.22252)}
	;
\end{scope}

\end{tikzpicture}

	\caption{Example \ref{ex:order10}. In the left: The biquadratic curve \eqref{eq:biquad10}, where $\alpha\in(0,1/4)$ is a root of polynomial $P_5$ given by \eqref{eq:P5}.
	The approximate value of that root is $\alpha\approx0.14381$.
	 In the right: The lower left component of the curve, together with an orbit of the group generated by the horizontal and vertical switches is shown.
		Since the QRT transform is of order $5$, the orbit consists of $10$ points.}
	\label{fig:Q10}
\end{figure}

\begin{example}\label{ex:order10b}
	Consider random walk with the following matrix P \eqref{eq:P}:
	\begin{equation}\label{eq:matrix10}
	P=\left(
	\begin{array}{ccc}
		1/4-\alpha & 1/10 & 1/10 \\
		0 & -9/10 & 0 \\
		1/4+\alpha & 0 & 1/5
	\end{array}
	\right),
	\quad\text{where}\quad |\alpha|\le\frac14.
	\end{equation}
The transition probabilities for such random walks are:
		\begin{equation}\label{eq:prob10}
		p_{11}=\frac14-\alpha,
		\,\,
		p_{10}=p_{1,-1}=p_{00}=\frac1{10},
		\,\, p_{01}=p_{0,-1}=p_{-1,0}=0,
		\,\,
		p_{-1,1}=\frac14+\alpha,
		\,\,
		p_{-1,-1}=\frac15.
		\end{equation}
	The condition for $5$-periodicity of the corresponding QRT map, according to Corollary \ref{cor:small}, is $Q_5(\alpha)=0$, where:
\begin{equation}\label{eq:Q5}
\begin{aligned}
Q_5(x)
=&\
720^6 x^{12}
+6^{11}\cdot20^6\cdot11 x^{11}
-3496\cdot12^9\cdot20^4 x^{10}
-508430088732672000 x^9
\\&
-362150426035814400 x^8+995139332834918400 x^7-671623654047580160 x^6
\\&
-248737097743994880 x^5+211388346237807360 x^4+217796690291328 x^3
\\&
-12588011050241664 x^2+696471951573144 x+137820649612347
.
\end{aligned}
\end{equation}
This polynomial has two roots in $(-1/4,0)$ and one root in $(0,1/4)$.	
For each of those roots, the corresponding biquadratic curve is:
	\begin{equation}\label{eq:biquad10b}
		Q(x,y)
		=
		\left(\frac{1}{4}-\alpha\right) x^2 y^2
		+
		\frac1{10}x^2y
		+
		\frac1{10}x^2
		+
		\left(\alpha+\frac{1}{4}\right) y^2
		-\frac9{10}x y
		+\frac{1}{5}
		=
		0,
	\end{equation}
which is smooth and of genus $1$, thus we got three examples of random walks with the groups of order $10$.
For the root which is in $(0,1/4)$, the curve is illustrated in Figure \ref{fig:Q10b}, together with the orbit of one point for the group of the random walk.
	\begin{figure}[h]
		\centering
		\begin{tikzpicture}[scale=1.5]
			
			\draw[very thin,color=gray!30,dashed] (-5.5,-4.5) grid (1.5,2.5);
			
			\draw[thick,->] (-5.5,0) -- (1.5,0) node[right] {$x$};
			\draw[thick,->] (0,-4.5) -- (0,2.5) node[above] {$y$};
			
			\draw[thick](0.1,1)--(-0.1,1)node[left]{$1$};
			\draw[thick](0.1,2)--(-0.1,2)node[left]{$2$};
			\draw[thick](-0.1,-1)--(0.1,-1)node[right]{$-1$};
			\draw[thick](-0.1,-2)--(0.1,-2)node[right]{$-2$};
			\draw[thick](-0.1,-3)--(0.1,-3)node[right]{$-3$};
			\draw[thick](-0.1,-4)--(0.1,-4)node[right]{$-4$};

			\draw[thick](1,0.1)--(1,-0.1)node[below]{$5$};
			\draw[thick](-1,-0.1)--(-1,0.1)node[above]{$-5$};
			\draw[thick](-2,-0.1)--(-2,0.1)node[above]{$-10$};
			\draw[thick](-3,-0.1)--(-3,0.1)node[above]{$-15$};
\draw[thick](-4,-0.1)--(-4,0.1)node[above]{$-20$};
\draw[thick](-5,-0.1)--(-5,0.1)node[above]{$-25$};

\draw [very thick] plot [smooth cycle] coordinates{
(-5.21654,-1.4913)(-5.16581,-1.59677)(-5.11507,-1.6432)(-5.06433,-1.6802)(-5.01359,-1.71245)(-4.96285,-1.74173)(-4.91211,-1.769)(-4.86137,-1.79478)(-4.81063,-1.81944)(-4.75989,-1.84322)(-4.70915,-1.86631)(-4.65841,-1.88885)(-4.60767,-1.91094)(-4.55694,-1.93267)(-4.5062,-1.95411)(-4.45546,-1.97532)(-4.40472,-1.99635)(-4.35398,-2.01725)(-4.30324,-2.03806)(-4.2525,-2.0588)(-4.20176,-2.0795)(-4.15102,-2.10021)(-4.10028,-2.12094)(-4.04954,-2.14171)(-3.9988,-2.16255)(-3.94807,-2.18348)(-3.89733,-2.20451)(-3.84659,-2.22568)(-3.79585,-2.24699)(-3.74511,-2.26846)(-3.69437,-2.29011)(-3.64363,-2.31196)(-3.59289,-2.33403)(-3.54215,-2.35632)(-3.49141,-2.37886)(-3.44067,-2.40167)(-3.38993,-2.42475)(-3.3392,-2.44814)(-3.28846,-2.47183)(-3.23772,-2.49586)(-3.18698,-2.52023)(-3.13624,-2.54497)(-3.0855,-2.57009)(-3.03476,-2.59561)(-2.98402,-2.62156)(-2.93328,-2.64793)(-2.88254,-2.67477)(-2.8318,-2.70208)(-2.78106,-2.72989)(-2.73033,-2.75821)(-2.67959,-2.78706)(-2.62885,-2.81648)(-2.57811,-2.84647)(-2.52737,-2.87706)(-2.47663,-2.90826)(-2.42589,-2.94011)(-2.37515,-2.97262)(-2.32441,-3.00581)(-2.27367,-3.03969)(-2.22293,-3.0743)(-2.17219,-3.10964)(-2.12146,-3.14574)(-2.07072,-3.1826)(-2.01998,-3.22023)(-1.96924,-3.25864)(-1.9185,-3.29784)(-1.86776,-3.33782)(-1.81702,-3.37856)(-1.76628,-3.42005)(-1.71554,-3.46224)(-1.6648,-3.5051)(-1.61406,-3.54856)(-1.56332,-3.59253)(-1.51259,-3.63689)(-1.46185,-3.68151)(-1.41111,-3.72618)(-1.36037,-3.77067)(-1.30963,-3.81469)(-1.25889,-3.85786)(-1.20815,-3.89971)(-1.15741,-3.93967)(-1.10667,-3.97703)(-1.05593,-4.01089)(-1.00519,-4.04018)(-0.954454,-4.06358)(-0.903715,-4.07947)(-0.852976,-4.08588)(-0.802237,-4.08044)(-0.751497,-4.06032)(-0.700758,-4.02216)(-0.650019,-3.962)(-0.59928,-3.87534)(-0.548541,-3.75706)(-0.497802,-3.60161)(-0.447062,-3.40318)(-0.396323,-3.15594)(-0.345584,-2.85443)(-0.294845,-2.49357)(-0.244106,-2.06733)(-0.193367,-1.56006)(-0.142627,-0.724307)(-0.193367,-0.378432)(-0.244106,-0.323526)(-0.294845,-0.303125)(-0.345584,-0.296939)(-0.396323,-0.298078)(-0.447062,-0.303429)(-0.497802,-0.311357)(-0.548541,-0.320921)(-0.59928,-0.331543)(-0.650019,-0.342851)(-0.700758,-0.354598)(-0.751497,-0.366615)(-0.802237,-0.378784)(-0.852976,-0.391022)
(-0.903715,-0.403269)(-0.954454,-0.415484)(-1.00519,-0.427635)(-1.05593,-0.439702)(-1.10667,-0.451669)(-1.15741,-0.463527)(-1.20815,-0.475268)(-1.25889,-0.486888)(-1.30963,-0.498387)(-1.36037,-0.509762)(-1.41111,-0.521015)(-1.46185,-0.532148)(-1.51259,-0.543164)(-1.56332,-0.554064)(-1.61406,-0.564852)(-1.6648,-0.575532)(-1.71554,-0.586107)(-1.76628,-0.596582)(-1.81702,-0.60696)(-1.86776,-0.617245)(-1.9185,-0.627442)(-1.96924,-0.637554)(-2.01998,-0.647585)(-2.07072,-0.65754)(-2.12146,-0.667423)(-2.17219,-0.677237)(-2.22293,-0.686987)(-2.27367,-0.696677)(-2.32441,-0.706309)(-2.37515,-0.71589)(-2.42589,-0.725421)(-2.47663,-0.734908)(-2.52737,-0.744354)(-2.57811,-0.753763)(-2.62885,-0.763138)(-2.67959,-0.772485)(-2.73033,-0.781806)(-2.78106,-0.791105)(-2.8318,-0.800387)(-2.88254,-0.809655)(-2.93328,-0.818914)(-2.98402,-0.828168)(-3.03476,-0.837421)(-3.0855,-0.846677)(-3.13624,-0.85594)(-3.18698,-0.865216)(-3.23772,-0.874509)(-3.28846,-0.883823)(-3.3392,-0.893165)(-3.38993,-0.902538)(-3.44067,-0.911949)(-3.49141,-0.921404)(-3.54215,-0.930908)(-3.59289,-0.940468)(-3.64363,-0.950091)(-3.69437,-0.959784)(-3.74511,-0.969554)(-3.79585,-0.979411)(-3.84659,-0.989363)(-3.89733,-0.999419)(-3.94807,-1.00959)(-3.9988,-1.01989)(-4.04954,-1.03033)(-4.10028,-1.04092)(-4.15102,-1.05168)(-4.20176,-1.06263)(-4.2525,-1.07378)(-4.30324,-1.08516)(-4.35398,-1.09679)(-4.40472,-1.1087)(-4.45546,-1.12093)(-4.5062,-1.13351)(-4.55694,-1.14649)(-4.60767,-1.15992)(-4.65841,-1.17387)(-4.70915,-1.18843)(-4.75989,-1.20369)(-4.81063,-1.21979)(-4.86137,-1.2369)(-4.91211,-1.25528)(-4.96285,-1.27528)(-5.01359,-1.29743)(-5.06433,-1.32267)(-5.11507,-1.35279)(-5.16581,-1.39247)};
			
\draw [very thick] plot [smooth cycle] coordinates{
	(0.18693,0.762394)(0.193983,0.635372)(0.201035,0.595285)(0.208088,0.569412)(0.21514,0.550834)(0.222193,0.536877)(0.229245,0.526176)(0.236297,0.517927)(0.24335,0.511608)(0.250402,0.506863)(0.257455,0.503434)(0.264507,0.50113)(0.271559,0.499806)(0.278612,0.499349)(0.285664,0.499669)(0.292717,0.500695)(0.299769,0.50237)(0.306822,0.504647)(0.313874,0.507489)(0.320926,0.510865)(0.327979,0.514753)(0.335031,0.519133)(0.342084,0.523991)(0.349136,0.529318)(0.356188,0.535108)(0.363241,0.541357)(0.370293,0.548068)(0.377346,0.555244)(0.384398,0.562893)(0.391451,0.571027)(0.398503,0.57966)(0.405555,0.588813)(0.412608,0.598508)(0.41966,0.608775)(0.426713,0.61965)(0.433765,0.631174)(0.440817,0.6434)(0.44787,0.65639)(0.454922,0.67022)(0.461975,0.684982)(0.469027,0.700795)(0.47608,0.717806)(0.483132,0.736207)(0.490184,0.756254)(0.497237,0.778296)(0.504289,0.802842)(0.511342,0.830668)(0.518394,0.863082)(0.525446,0.902632)(0.532499,0.955871)(0.539551,1.09139)(0.532499,1.23405)(0.525446,1.29393)(0.518394,1.33962)(0.511342,1.37764)(0.504289,1.41054)(0.497237,1.43958)(0.490184,1.46554)(0.483132,1.48889)(0.47608,1.50997)(0.469027,1.52901)(0.461975,1.54618)(0.454922,1.5616)(0.44787,1.57538)(0.440817,1.58758)(0.433765,1.59825)(0.426713,1.60744)(0.41966,1.61516)(0.412608,1.62145)(0.405555,1.62632)(0.398503,1.62976)(0.391451,1.63178)(0.384398,1.63238)(0.377346,1.63155)(0.370293,1.62927)(0.363241,1.62553)(0.356188,1.62031)(0.349136,1.61359)(0.342084,1.60534)(0.335031,1.59555)(0.327979,1.58416)(0.320926,1.57116)(0.313874,1.5565)(0.306822,1.54014)(0.299769,1.52202)(0.292717,1.5021)(0.285664,1.4803)(0.278612,1.45656)(0.271559,1.43079)(0.264507,1.40288)(0.257455,1.37271)(0.250402,1.34011)(0.24335,1.30489)(0.236297,1.26678)(0.229245,1.22542)(0.222193,1.18027)(0.21514,1.13053)(0.208088,1.07482)(0.201035,1.01048)(0.193983,0.930573)};

\coordinate(q1)at(-4,-2.16206);
\coordinate(q2)at(-4,-1.02013);
\coordinate(q3)at(-0.15244,-1.02013);
\coordinate(q4)at(-0.15244,-0.525572);
\coordinate(q5)at(-1.43181,-0.525572);
\coordinate(q6)at(-1.43181,-3.70796);
\coordinate(q7)at(-0.531049,-3.70796);
\coordinate(q8)at(-0.531049,-0.31748);
\coordinate(q9)at(-0.254685,-0.31748);
\coordinate(q10)at(-0.254685,-2.16206);

\draw[fill=black](q1) circle (0.05);		
\draw[fill=black](q2) circle (0.05);		
\draw[fill=black](q3) circle (0.05);		
\draw[fill=black](q4) circle (0.05);		
\draw[fill=black](q5) circle (0.05);		
\draw[fill=black](q6) circle (0.05);		
\draw[fill=black](q7) circle (0.05);		
\draw[fill=black](q8) circle (0.05);		
\draw[fill=black](q9) circle (0.05);		
\draw[fill=black](q10) circle (0.05);

\draw(q1)--(q2)--(q3)--(q4)--(q5)--(q6)--(q7)--(q8)--(q9)--(q10)--cycle;

		\end{tikzpicture}
		\caption{Example \ref{ex:order10b}. The biquadratic curve \eqref{eq:biquad10b}, where $\alpha\in(0,1/4)$ is a root of polynomial $Q_5$ given by \eqref{eq:Q5}. The value of that root is approximately $\alpha\approx0.20557$.			
	The orbit of a point in the group generated by horizontal and vertical switches is also shown, and it consists of $10$ points, which means that the corresponding group of random walk is of order $10$. Note that the scales along the axes are different.}
		\label{fig:Q10b}
	\end{figure}
	\end{example}

It seems that examples of random walks with the groups of orders higher than $10$
were not known in the literature. Our method can easily be used to construct explicit examples of random walks with the groups of any given order.
We illustrate this by constructing random walks with the groups of orders $12$, $14$, and  $16$ in the next Section \ref{sec:1214}.

\subsection{Groups of order 12, 14, 16}\label{sec:1214}

In this section, we construct examples of the random walk with the groups of order higher than $10$.

\begin{example}[Random walks with group of order 12]\label{ex:order12}
Consider random walks with the following matrix P \eqref{eq:P}:
\begin{equation}\label{eq:matrix}
	P=\left(
	\begin{array}{ccc}
		\frac{1}{4}-\alpha & 0 & \frac{3}{10} \\
		0 & -1 & 0 \\
		\frac14+\alpha & 0 & \frac{1}{5} \\
	\end{array}
	\right),
	\quad\text{where}\quad |\alpha|\le\frac14.
\end{equation}
The transition probabilities for such random walks are:
$$
p_{11}=\frac14-\alpha,
\quad
p_{10}=p_{01}=p_{0,-1}=p_{-1,0}=p_{00}=0,
\quad
p_{1,-1}=\frac3{10},
\quad
p_{-1,-1}=\frac15,
\quad
p_{-1,1}=\frac14+\alpha.
$$

The condition for $6$-periodicity of the corresponding QRT map, given in Corollary \ref{cor:small}, factorizes to the product of the conditions for $2$- and $3$-periodicity, and $P_{6}(\alpha)=0$, with:
\begin{equation}\label{eq:P6}
\begin{aligned}
P_6(x)
=\ &
256\cdot10^8 x^8
+1024\cdot10^{7} x^7
+149248\cdot10^{6}x^6
-96512\cdot10^5 x^5
\\&
-32552\cdot10^6 x^4+5505472000 x^3+6006692800 x^2-305663840 x-284217599.
\end{aligned}
\end{equation}
This polynomial has exactly two real roots: one in each of the intervals $(-1/4,0)$ and $(0,1/4)$.
Since for those choices of roots, the corresponding biquadratic curve:
\begin{equation}\label{eq:biquad12}
	Q(x,y)
	=
	\left(\frac{1}{4}-\alpha\right) x^2 y^2
	+
	\left(\alpha+\frac{1}{4}\right) y^2
	+\frac{3x^2}{10}
	-x y
	+\frac{1}{5}
	=
	0
\end{equation}
is smooth and of genus $1$, we get here two examples of random walks with the groups of order $12$.

In both examples, the real part of the biquadratic curve \eqref{eq:biquad12} has two connected components, which are symmetric to each other with respect to the origin, as shown in Figure \ref{fig:Q12}.
Notice that each of the two components will be invariant for the QRT transformation.
	In the figure, two orbits are shown, each in one of the connected components.
\end{example}
\begin{figure}[h]
	\centering
	\begin{tikzpicture}[scale=0.3]
		
		\draw[very thin,color=gray!30,dashed,step=10] (-25,-21) grid (25,21);
		
		\draw[thick,->] (-25,0) -- (25,0) node[right] {$x$};
		\draw[thick,->] (0,-21) -- (0,21) node[above] {$y$};
		
		\draw[thick](0.5,10)--(-0.5,10)node[left]{$10$};
		\draw[thick](0.5,20)--(-0.5,20)node[left]{$20$};
		\draw[thick](0.5,-10)--(-0.5,-10)node[left]{$-10$};
		\draw[thick](0.5,-20)--(-0.5,-20)node[left]{$-20$};
		
		\draw[thick](10,0.5)--(10,-0.5)node[below]{$10$};
		\draw[thick](20,0.5)--(20,-0.5)node[below]{$20$};
		\draw[thick](-10,0.5)--(-10,-0.5)node[below]{$-10$};
		\draw[thick](-20,0.5)--(-20,-0.5)node[below]{$-20$};

		\draw [very thick] plot [smooth cycle]
	 coordinates{
	 	(1.,1.)(1.20844,0.779791)(1.41687,0.782719)(1.62531,0.815909)(1.83374,0.862871)(2.04218,0.917691)(2.25061,0.977552)(2.45905,1.0409)(2.66748,1.1068)(2.87592,1.17463)(3.08435,1.244)(3.29279,1.31459)(3.50122,1.38622)(3.70966,1.45872)(3.91809,1.53197)(4.12653,1.60589)(4.33496,1.68041)(4.5434,1.75547)(4.75184,1.83104)(4.96027,1.90709)(5.16871,1.98358)(5.37714,2.0605)(5.58558,2.13784)(5.79401,2.21558)(6.00245,2.29373)(6.21088,2.37227)(6.41932,2.4512)(6.62775,2.53052)(6.83619,2.61024)(7.04462,2.69036)(7.25306,2.77088)(7.46149,2.85181)(7.66993,2.93316)(7.87836,3.01494)(8.0868,3.09714)(8.29524,3.1798)(8.50367,3.2629)(8.71211,3.34648)(8.92054,3.43053)(9.12898,3.51508)(9.33741,3.60014)(9.54585,3.68572)(9.75428,3.77184)(9.96272,3.85852)(10.1712,3.94577)(10.3796,4.03362)(10.588,4.12208)(10.7965,4.21118)(11.0049,4.30093)(11.2133,4.39136)(11.4218,4.4825)(11.6302,4.57437)(11.8386,4.66699)(12.0471,4.76041)(12.2555,4.85463)(12.4639,4.94971)(12.6724,5.04566)(12.8808,5.14253)(13.0892,5.24036)(13.2977,5.33917)(13.5061,5.43903)(13.7146,5.53996)(13.923,5.64202)(14.1314,5.74525)(14.3399,5.84972)(14.5483,5.95547)(14.7567,6.06257)(14.9652,6.17109)(15.1736,6.2811)(15.382,6.39267)(15.5905,6.50589)(15.7989,6.62084)(16.0073,6.73762)(16.2158,6.85634)(16.4242,6.97712)(16.6326,7.10007)(16.8411,7.22534)(17.0495,7.35309)(17.258,7.48347)(17.4664,7.61669)(17.6748,7.75294)(17.8833,7.89248)(18.0917,8.03557)(18.3001,8.18251)(18.5086,8.33367)(18.717,8.48943)(18.9254,8.65029)(19.1339,8.81681)(19.3423,8.98964)(19.5507,9.16961)(19.7592,9.35771)(19.9676,9.55519)(20.176,9.76367)(20.3845,9.98527)(20.5929,10.2229)(20.8014,10.4808)(21.0098,10.7652)
(21.2182,11.087)
	 	(21.4267,11.4671)
	 	(21.6351,11.959)
	 	(21.8435,13.1244)
	 	%(21.8435,13.1244)
	 	(21.6351,14.2384)(21.4267,14.6762)(21.2182,14.9993)(21.0098,15.2613)(20.8014,15.4831)(20.5929,15.6754)(20.3845,15.8444)(20.176,15.9944)(19.9676,16.1282)(19.7592,16.2478)(19.5507,16.3549)(19.3423,16.4506)(19.1339,16.5358)(18.9254,16.6114)(18.717,16.6779)(18.5086,16.7359)(18.3001,16.7857)(18.0917,16.8278)(17.8833,16.8624)(17.6748,16.8898)(17.4664,16.9102)(17.258,16.9238)(17.0495,16.9307)(16.8411,16.9311)(16.6326,16.9252)(16.4242,16.9129)(16.2158,16.8945)(16.0073,16.8699)(15.7989,16.8393)(15.5905,16.8028)(15.382,16.7602)(15.1736,16.7119)(14.9652,16.6576)(14.7567,16.5976)(14.5483,16.5318)(14.3399,16.4602)(14.1314,16.3829)(13.923,16.2999)(13.7146,16.2112)(13.5061,16.1169)(13.2977,16.017)(13.0892,15.9114)(12.8808,15.8002)(12.6724,15.6834)(12.4639,15.561)(12.2555,15.4331)(12.0471,15.2996)(11.8386,15.1606)(11.6302,15.0161)(11.4218,14.8662)(11.2133,14.7107)(11.0049,14.5498)(10.7965,14.3835)(10.588,14.2118)(10.3796,14.0347)(10.1712,13.8523)(9.96272,13.6646)(9.75428,13.4716)(9.54585,13.2734)(9.33741,13.07)(9.12898,12.8615)(8.92054,12.6478)(8.71211,12.4291)(8.50367,12.2054)(8.29524,11.9767)(8.0868,11.7431)(7.87836,11.5046)(7.66993,11.2614)(7.46149,11.0135)(7.25306,10.7609)(7.04462,10.5037)(6.83619,10.242)(6.62775,9.97592)(6.41932,9.70545)(6.21088,9.4307)(6.00245,9.15177)(5.79401,8.86872)(5.58558,8.58165)(5.37714,8.29064)(5.16871,7.99577)(4.96027,7.69715)(4.75184,7.39484)(4.5434,7.08893)(4.33496,6.77951)(4.12653,6.46664)(3.91809,6.15039)(3.70966,5.83081)(3.50122,5.50793)(3.29279,5.18177)(3.08435,4.8523)(2.87592,4.51944)(2.66748,4.18303)(2.45905,3.8428)(2.25061,3.49828)(2.04218,3.14867)(1.83374,2.79256)(1.62531,2.42727)(1.41687,2.04704)(1.20844,1.63553)}
		;

		\draw [very thick] plot [smooth cycle]
coordinates{
	(-1.,-1.)(-1.20844,-0.779791)(-1.41687,-0.782719)(-1.62531,-0.815909)(-1.83374,-0.862871)(-2.04218,-0.917691)(-2.25061,-0.977552)(-2.45905,-1.0409)(-2.66748,-1.1068)(-2.87592,-1.17463)(-3.08435,-1.244)(-3.29279,-1.31459)(-3.50122,-1.38622)(-3.70966,-1.45872)(-3.91809,-1.53197)(-4.12653,-1.60589)(-4.33496,-1.68041)(-4.5434,-1.75547)(-4.75184,-1.83104)(-4.96027,-1.90709)(-5.16871,-1.98358)(-5.37714,-2.0605)(-5.58558,-2.13784)(-5.79401,-2.21558)(-6.00245,-2.29373)(-6.21088,-2.37227)(-6.41932,-2.4512)(-6.62775,-2.53052)(-6.83619,-2.61024)(-7.04462,-2.69036)(-7.25306,-2.77088)(-7.46149,-2.85181)(-7.66993,-2.93316)(-7.87836,-3.01494)(-8.0868,-3.09714)(-8.29524,-3.1798)(-8.50367,-3.2629)(-8.71211,-3.34648)(-8.92054,-3.43053)(-9.12898,-3.51508)(-9.33741,-3.60014)(-9.54585,-3.68572)(-9.75428,-3.77184)(-9.96272,-3.85852)(-10.1712,-3.94577)(-10.3796,-4.03362)(-10.588,-4.12208)(-10.7965,-4.21118)(-11.0049,-4.30093)(-11.2133,-4.39136)(-11.4218,-4.4825)(-11.6302,-4.57437)(-11.8386,-4.66699)(-12.0471,-4.76041)(-12.2555,-4.85463)(-12.4639,-4.94971)(-12.6724,-5.04566)(-12.8808,-5.14253)(-13.0892,-5.24036)(-13.2977,-5.33917)(-13.5061,-5.43903)(-13.7146,-5.53996)(-13.923,-5.64202)(-14.1314,-5.74525)(-14.3399,-5.84972)(-14.5483,-5.95547)(-14.7567,-6.06257)(-14.9652,-6.17109)(-15.1736,-6.2811)(-15.382,-6.39267)(-15.5905,-6.50589)(-15.7989,-6.62084)(-16.0073,-6.73762)(-16.2158,-6.85634)(-16.4242,-6.97712)(-16.6326,-7.10007)(-16.8411,-7.22534)(-17.0495,-7.35309)(-17.258,-7.48347)(-17.4664,-7.61669)(-17.6748,-7.75294)(-17.8833,-7.89248)(-18.0917,-8.03557)(-18.3001,-8.18251)(-18.5086,-8.33367)(-18.717,-8.48943)(-18.9254,-8.65029)(-19.1339,-8.81681)(-19.3423,-8.98964)(-19.5507,-9.16961)(-19.7592,-9.35771)(-19.9676,-9.55519)(-20.176,-9.76367)(-20.3845,-9.98527)(-20.5929,-10.2229)(-20.8014,-10.4808)(-21.0098,-10.7652)(-21.2182,-11.087)(-21.4267,-11.4671)(-21.6351,-11.959)
	(-21.8435,-13.1244)(-21.6351,-14.2384)(-21.4267,-14.6762)(-21.2182,-14.9993)(-21.0098,-15.2613)(-20.8014,-15.4831)(-20.5929,-15.6754)(-20.3845,-15.8444)(-20.176,-15.9944)(-19.9676,-16.1282)(-19.7592,-16.2478)(-19.5507,-16.3549)(-19.3423,-16.4506)(-19.1339,-16.5358)(-18.9254,-16.6114)(-18.717,-16.6779)(-18.5086,-16.7359)(-18.3001,-16.7857)(-18.0917,-16.8278)(-17.8833,-16.8624)(-17.6748,-16.8898)(-17.4664,-16.9102)(-17.258,-16.9238)(-17.0495,-16.9307)(-16.8411,-16.9311)(-16.6326,-16.9252)(-16.4242,-16.9129)(-16.2158,-16.8945)(-16.0073,-16.8699)(-15.7989,-16.8393)(-15.5905,-16.8028)(-15.382,-16.7602)(-15.1736,-16.7119)(-14.9652,-16.6576)(-14.7567,-16.5976)(-14.5483,-16.5318)(-14.3399,-16.4602)(-14.1314,-16.3829)(-13.923,-16.2999)(-13.7146,-16.2112)(-13.5061,-16.1169)(-13.2977,-16.017)(-13.0892,-15.9114)(-12.8808,-15.8002)(-12.6724,-15.6834)(-12.4639,-15.561)(-12.2555,-15.4331)(-12.0471,-15.2996)(-11.8386,-15.1606)(-11.6302,-15.0161)(-11.4218,-14.8662)(-11.2133,-14.7107)(-11.0049,-14.5498)(-10.7965,-14.3835)(-10.588,-14.2118)(-10.3796,-14.0347)(-10.1712,-13.8523)(-9.96272,-13.6646)(-9.75428,-13.4716)(-9.54585,-13.2734)(-9.33741,-13.07)(-9.12898,-12.8615)(-8.92054,-12.6478)(-8.71211,-12.4291)(-8.50367,-12.2054)(-8.29524,-11.9767)(-8.0868,-11.7431)(-7.87836,-11.5046)(-7.66993,-11.2614)(-7.46149,-11.0135)(-7.25306,-10.7609)(-7.04462,-10.5037)(-6.83619,-10.242)(-6.62775,-9.97592)(-6.41932,-9.70545)(-6.21088,-9.4307)(-6.00245,-9.15177)(-5.79401,-8.86872)(-5.58558,-8.58165)(-5.37714,-8.29064)(-5.16871,-7.99577)(-4.96027,-7.69715)(-4.75184,-7.39484)(-4.5434,-7.08893)(-4.33496,-6.77951)(-4.12653,-6.46664)(-3.91809,-6.15039)(-3.70966,-5.83081)(-3.50122,-5.50793)(-3.29279,-5.18177)(-3.08435,-4.8523)(-2.87592,-4.51944)(-2.66748,-4.18303)(-2.45905,-3.8428)(-2.25061,-3.49828)(-2.04218,-3.14867)(-1.83374,-2.79256)(-1.62531,-2.42727)(-1.41687,-2.04704)(-1.20844,-1.63553)}
;

\coordinate(a1)at(20,9.58682);
\coordinate(a2)at(20,16.1084);
\coordinate(a3)at(13.4879,16.1084);
\coordinate(a4)at(13.4879,5.43023);
\coordinate(a5)at(3.45138,5.43023);
\coordinate(a6)at(3.45138,1.36901);
\coordinate(a7)at(1.09218,1.36901);
\coordinate(a8)at(1.09218,0.814758);
\coordinate(a9)at(1.6195,0.814758);
\coordinate(a10)at(1.6195,2.41692);
\coordinate(a11)at(6.32893,2.41692);
\coordinate(a12)at(6.32893,9.58682);

\draw[fill=black](a1) circle (0.25);		
\draw[fill=black](a2) circle (0.25);		
\draw[fill=black](a3) circle (0.25);		
\draw[fill=black](a4) circle (0.25);		
\draw[fill=black](a5) circle (0.25);		
\draw[fill=black](a6) circle (0.25);		
\draw[fill=black](a7) circle (0.25);		
\draw[fill=black](a8) circle (0.25);		
\draw[fill=black](a9) circle (0.25);		
\draw[fill=black](a10) circle (0.25);		
\draw[fill=black](a11) circle (0.25);		
\draw[fill=black](a12) circle (0.25);		

\draw(a1)--(a2)--(a3)--(a4)--(a5)--(a6)--(a7)--(a8)--(a9)--(a10)--(a11)--(a12)--cycle;

\coordinate(b1)at(-21,-15.2725);
\coordinate(b2)at(-21,-10.7511);
\coordinate(b3)at(-7.24509,-10.7511);
\coordinate(b4)at(-7.24509,-2.7678);
\coordinate(b5)at(-1.81942,-2.7678);
\coordinate(b6)at(-1.81942,-0.859349);
\coordinate(b7)at(-1.04017,-0.859349);
\coordinate(b8)at(-1.04017,-1.22075);
\coordinate(b9)at(-3.01494,-1.22075);
\coordinate(b10)at(-3.01494,-4.74183);
\coordinate(b11)at(-12.0058,-4.74183);
\coordinate(b12)at(-12.0058,-15.2725);

\draw[fill=black](b1) circle (0.25);		
\draw[fill=black](b2) circle (0.25);		
\draw[fill=black](b3) circle (0.25);		
\draw[fill=black](b4) circle (0.25);		
\draw[fill=black](b5) circle (0.25);		
\draw[fill=black](b6) circle (0.25);		
\draw[fill=black](b7) circle (0.25);		
\draw[fill=black](b8) circle (0.25);		
\draw[fill=black](b9) circle (0.25);		
\draw[fill=black](b10) circle (0.25);		
\draw[fill=black](b11) circle (0.25);		
\draw[fill=black](b12) circle (0.25);		

\draw(b1)--(b2)--(b3)--(b4)--(b5)--(b6)--(b7)--(b8)--(b9)--(b10)--(b11)--(b12)--cycle;
		
	\end{tikzpicture}
	\caption{Example \ref{ex:order12}. The biquadratic curve \eqref{eq:biquad12}, for $\alpha\in(0,1/4)$ being a root of the polynomial $P_6$ given by \eqref{eq:P6}.  The root has approximate value $\alpha\approx0.24930$.  Since the QRT transformation is of order $6$, the orbit of each point in the corresponding group of random walk consists of $12$ points. Here, two such orbits are shown. }
	\label{fig:Q12}
\end{figure}
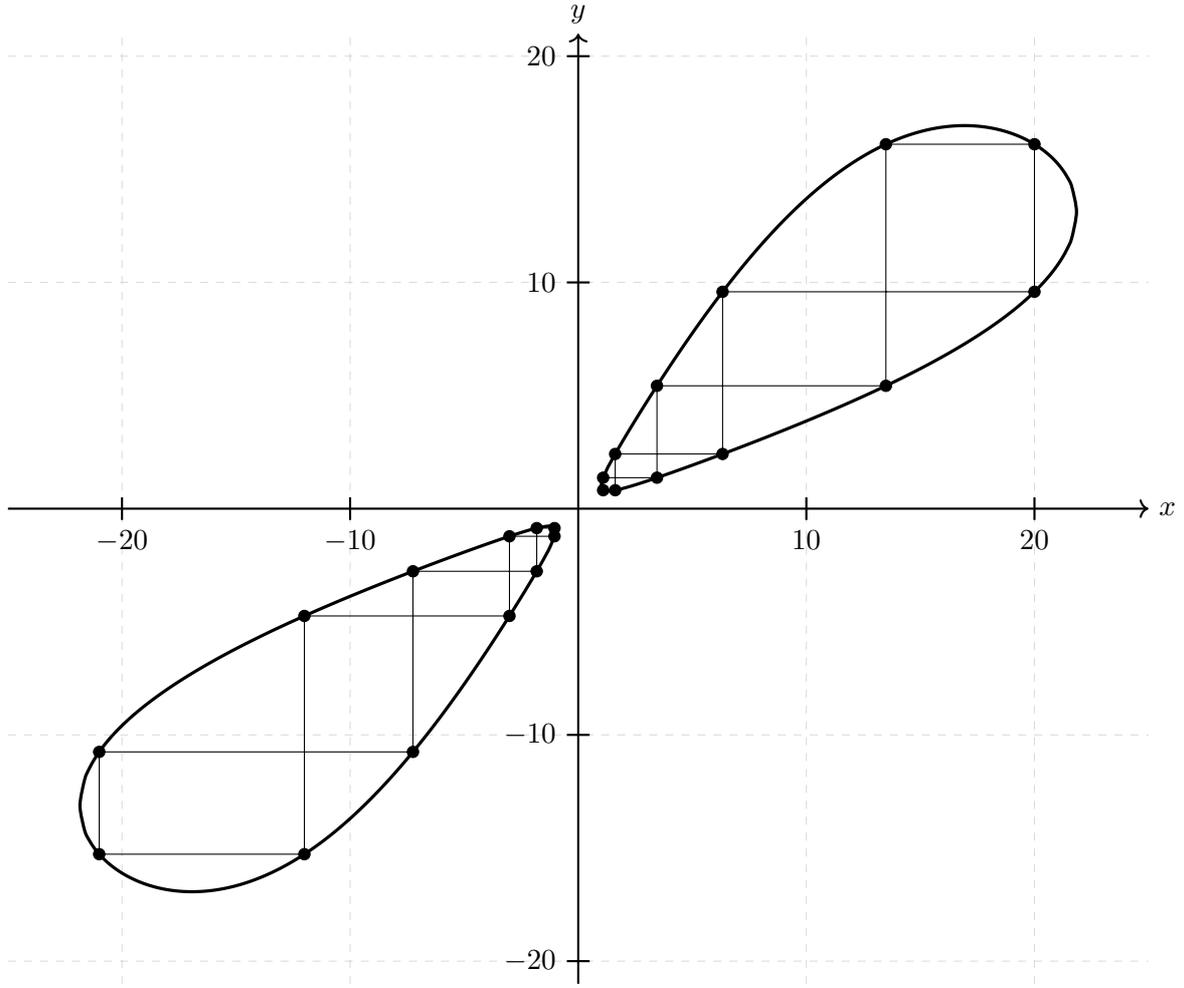

\begin{example}[Random walks with group of order 14]\label{ex:order14}
Consider random walks with the matrix \eqref{eq:matrix}.
The condition for $7$-periodicity of the corresponding QRT map, obtained using Theorem \ref{th:cayley}, is of the form $P_{7}(\alpha)\cdot Q_7(\alpha)=0$, with:
\begin{align*}
P_7(x)
=&\
20^{12} x^{12}
-132\cdot20^{11} x^{11}
-3054\cdot20^{10} x^{10}
-1537\cdot20^{10} x^9
+33220992\cdot10^{9} x^8
\\&
+5169528\cdot20^{7} x^7
-30266436\cdot{20}^6 x^6
+479310648\cdot{20}^5 x^5
+2661358104\cdot10^5 x^4
\\&
-1641246361\cdot20^4 x^3
-39080282181600 x^2+8579231661360 x+1275445223281,
\end{align*}
and
\begin{align*}
Q_7(x)
=&\
20^{12} x^{12}
+108\cdot20^{11} x^{11}
-2014\cdot20^{10} x^{10}
+2675\cdot20^{10} x^9
+1188495\cdot20^8 x^8
-3606312\cdot20^7 x^7
\\&
-77738916\cdot20^6 x^6
-227633832\cdot20^5 x^5
+644039196\cdot10^6x^4
+1077020579\cdot20^4 x^3
\\&
-36482843685600 x^2-7213762397840 x+530637311521.
\end{align*}
One can check that $P_7$ has six real zeros, four of which are in $(-1/4,1/4)$, while $Q_7$ also has six real zeros with three of them in $(-1/4,1/4)$.
The biquadratic curve \eqref{eq:biquad12} is smooth and of genus $1$ for all of those roots, thus here we have seven distinct examples of random walks with the group of order $14$.
In Figure \ref{fig:Q14}, one of those examples is shown.
\end{example}
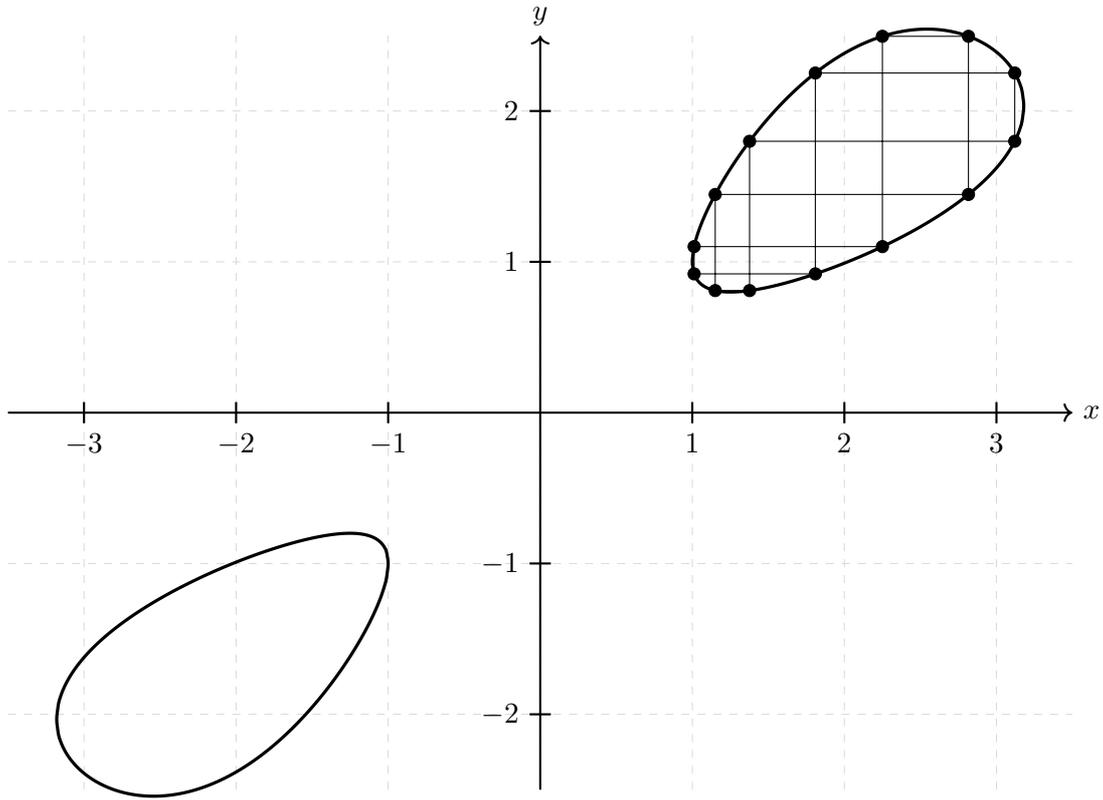
\begin{figure}[h]
	\centering
	\begin{tikzpicture}[scale=2]
		
	\draw[very thin,color=gray!30,dashed,step=1] (-3.5,-2.5) grid (3.5,2.5);
		
	\draw[thick,->] (-3.5,0) -- (3.5,0) node[right] {$x$};
	\draw[thick,->] (0,-2.5) -- (0,2.5) node[above] {$y$};
		
	\draw[thick](0.07,1)--(-0.07,1)node[left]{$1$};
	\draw[thick](0.07,2)--(-0.07,2)node[left]{$2$};
	\draw[thick](0.07,-1)--(-0.07,-1)node[left]{$-1$};
	\draw[thick](0.07,-2)--(-0.07,-2)node[left]{$-2$};
		
		\draw[thick](1,0.07)--(1,-0.07)node[below]{$1$};
		\draw[thick](2,0.07)--(2,-0.07)node[below]{$2$};
		\draw[thick](-1,0.07)--(-1,-0.07)node[below]{$-1$};
		\draw[thick](-2,0.07)--(-2,-0.07)node[below]{$-2$};
		\draw[thick](3,0.07)--(3,-0.07)node[below]{$3$};
\draw[thick](-3,0.07)--(-3,-0.07)node[below]{$-3$};

\coordinate(c1)at(2.25,1.1009);
\coordinate(c2)at(2.25,2.49537);
\coordinate(c3)at(2.81583,2.49537);
\coordinate(c4)at(2.81583,1.4467);
\coordinate(c5)at(1.15064,1.4467);
\coordinate(c6)at(1.15064,0.809376);
\coordinate(c7)at(1.37659,0.809376);
\coordinate(c8)at(1.37659,1.79938);
\coordinate(c9)at(3.12029,1.79938);
\coordinate(c10)at(3.12029,2.25189);
\coordinate(c11)at(1.80899,2.25189);
\coordinate(c12)at(1.80899,0.920199);
\coordinate(c13)at(1.01207,0.920199);
\coordinate(c14)at(1.01207,1.1009);

\draw(c1)--(c2)--(c3)--(c4)--(c5)--(c6)--(c7)--(c8)--(c9)--(c10)--(c11)--(c12)--(c13)--(c14)--cycle;

\draw[fill=black](c1) circle (0.04);		
\draw[fill=black](c2) circle (0.04);		
\draw[fill=black](c3) circle (0.04);		
\draw[fill=black](c4) circle (0.04);		
\draw[fill=black](c5) circle (0.04);		
\draw[fill=black](c6) circle (0.04);		
\draw[fill=black](c7) circle (0.04);		
\draw[fill=black](c8) circle (0.04);		
\draw[fill=black](c9) circle (0.04);		
\draw[fill=black](c10) circle (0.04);		
\draw[fill=black](c11) circle (0.04);		
\draw[fill=black](c12) circle (0.04);		
\draw[fill=black](c13) circle (0.04);		
\draw[fill=black](c14) circle (0.04);

		\draw [very thick] plot [smooth cycle]
coordinates{(1.,1.)(1.0109,0.923664)(1.0218,0.897573)(1.0327,0.879716)(1.0436,0.866106)(1.05449,0.8552)(1.06539,0.846207)(1.07629,0.838658)(1.08719,0.832247)
(1.09809,0.826762)(1.10899,0.82205)(1.11989,0.817994)(1.13079,0.814504)(1.14168,0.811508)(1.15258,0.80895)(1.16348,0.806781)(1.17438,0.804962)(1.18528,0.803459)(1.19618,0.802243)(1.20708,0.801291)(1.21798,0.800581)(1.22887,0.800094)(1.23977,0.799814)(1.25067,0.799726)(1.26157,0.799818)(1.27247,0.800077)(1.28337,0.800495)(1.29427,0.801061)(1.30517,0.801766)(1.31606,0.802604)(1.32696,0.803568)(1.33786,0.80465)(1.34876,0.805845)(1.35966,0.807148)(1.37056,0.808554)(1.38146,0.810059)(1.39236,0.811657)(1.40325,0.813346)(1.41415,0.815121)(1.42505,0.81698)(1.43595,0.818919)(1.44685,0.820935)(1.45775,0.823026)(1.46865,0.825188)(1.47955,0.827421)(1.49044,0.829721)(1.50134,0.832086)(1.51224,0.834515)(1.52314,0.837005)(1.53404,0.839556)(1.54494,0.842164)(1.55584,0.84483)(1.56674,0.847551)(1.57763,0.850326)(1.58853,0.853153)(1.59943,0.856032)(1.61033,0.858962)(1.62123,0.861941)(1.63213,0.864968)(1.64303,0.868043)(1.65393,0.871164)(1.66482,0.874331)(1.67572,0.877542)(1.68662,0.880798)(1.69752,0.884097)(1.70842,0.887439)(1.71932,0.890823)(1.73022,0.894249)(1.74112,0.897716)(1.75201,0.901223)(1.76291,0.90477)(1.77381,0.908357)(1.78471,0.911983)(1.79561,0.915647)
(1.80651,0.91935)(1.81741,0.923091)(1.82831,0.926869)(1.8392,0.930685)(1.8501,0.934538)(1.861,0.938428)(1.8719,0.942354)(1.8828,0.946317)(1.8937,0.950316)(1.9046,0.954352)(1.9155,0.958423)(1.92639,0.96253)(1.93729,0.966672)(1.94819,0.970851)(1.95909,0.975065)(1.96999,0.979314)(1.98089,0.983599)(1.99179,0.987919)(2.00269,0.992275)(2.01358,0.996667)(2.02448,1.00109)(2.03538,1.00556)(2.04628,1.01005)(2.05718,1.01459)(2.06808,1.01916)(2.07898,1.02376)(2.08988,1.0284)(2.10077,1.03308)(2.11167,1.0378)(2.12257,1.04255)(2.13347,1.04734)(2.14437,1.05216)(2.15527,1.05702)(2.16617,1.06192)(2.17707,1.06686)(2.18796,1.07184)(2.19886,1.07685)(2.20976,1.0819)(2.22066,1.08699)(2.23156,1.09212)(2.24246,1.0973)(2.25336,1.10251)(2.26426,1.10776)(2.27515,1.11305)(2.28605,1.11839)(2.29695,1.12376)(2.30785,1.12918)(2.31875,1.13464)(2.32965,1.14015)(2.34055,1.1457)(2.35145,1.1513)(2.36234,1.15694)(2.37324,1.16262)(2.38414,1.16836)(2.39504,1.17414)(2.40594,1.17997)(2.41684,1.18585)(2.42774,1.19178)(2.43864,1.19776)(2.44953,1.20379)(2.46043,1.20988)(2.47133,1.21601)(2.48223,1.22221)(2.49313,1.22846)
(2.50403,1.23476)(2.51493,1.24113)(2.52583,1.24755)(2.53672,1.25404)(2.54762,1.26058)(2.55852,1.26719)(2.56942,1.27387)(2.58032,1.28061)(2.59122,1.28742)(2.60212,1.2943)(2.61302,1.30125)(2.62391,1.30827)(2.63481,1.31537)(2.64571,1.32255)(2.65661,1.32981)(2.66751,1.33715)(2.67841,1.34457)(2.68931,1.35208)(2.70021,1.35968)(2.7111,1.36737)(2.722,1.37516)(2.7329,1.38305)(2.7438,1.39104)(2.7547,1.39913)(2.7656,1.40734)(2.7765,1.41566)(2.7874,1.42409)(2.79829,1.43265)(2.80919,1.44134)(2.82009,1.45017)(2.83099,1.45913)(2.84189,1.46823)(2.85279,1.47749)(2.86369,1.48691)(2.87459,1.4965)(2.88548,1.50627)(2.89638,1.51622)(2.90728,1.52637)(2.91818,1.53672)(2.92908,1.5473)(2.93998,1.55811)(2.95088,1.56918)(2.96178,1.58051)(2.97267,1.59213)(2.98357,1.60406)(2.99447,1.61633)(3.00537,1.62897)(3.01627,1.64201)(3.02717,1.65549)(3.03807,1.66945)(3.04897,1.68396)(3.05986,1.69908)(3.07076,1.7149)(3.08166,1.73152)(3.09256,1.74908)(3.10346,1.76775)(3.11436,1.78779)(3.12526,1.80954)(3.13616,1.83356)(3.14705,1.86074)(3.15795,1.89287)(3.16885,1.93452)
(3.17975, 2.03365)
(3.16885,2.12996)(3.15795,2.16874)(3.14705,2.19794)(3.13616,2.22215)(3.12526,2.24313)(3.11436,2.2618)(3.10346,2.27869)(3.09256,2.29417)(3.08166,2.30847)(3.07076,2.32177)(3.05986,2.33422)(3.04897,2.34592)(3.03807,2.35694)(3.02717,2.36736)(3.01627,2.37724)(3.00537,2.38661)(2.99447,2.39553)(2.98357,2.40402)(2.97267,2.41211)(2.96178,2.41983)(2.95088,2.4272)(2.93998,2.43424)(2.92908,2.44097)(2.91818,2.4474)(2.90728,2.45354)(2.89638,2.45942)(2.88548,2.46503)(2.87459,2.4704)(2.86369,2.47552)(2.85279,2.48042)(2.84189,2.48508)(2.83099,2.48953)(2.82009,2.49377)(2.80919,2.49781)(2.79829,2.50164)(2.7874,2.50528)(2.7765,2.50873)(2.7656,2.51199)(2.7547,2.51507)(2.7438,2.51798)(2.7329,2.52071)(2.722,2.52326)(2.7111,2.52566)(2.70021,2.52788)(2.68931,2.52994)(2.67841,2.53185)(2.66751,2.53359)(2.65661,2.53518)(2.64571,2.53662)(2.63481,2.5379)(2.62391,2.53903)(2.61302,2.54002)(2.60212,2.54086)(2.59122,2.54156)(2.58032,2.54211)(2.56942,2.54252)(2.55852,2.54279)(2.54762,2.54292)(2.53672,2.54291)(2.52583,2.54276)(2.51493,2.54248)(2.50403,2.54206)(2.49313,2.5415)(2.48223,2.54081)(2.47133,2.53999)(2.46043,2.53904)(2.44953,2.53795)(2.43864,2.53674)(2.42774,2.53539)(2.41684,2.53391)(2.40594,2.53231)(2.39504,2.53057)(2.38414,2.52871)(2.37324,2.52671)(2.36234,2.5246)(2.35145,2.52235)(2.34055,2.51997)(2.32965,2.51747)(2.31875,2.51485)(2.30785,2.51209)(2.29695,2.50921)(2.28605,2.50621)(2.27515,2.50308)(2.26426,2.49982)(2.25336,2.49644)(2.24246,2.49294)(2.23156,2.4893)(2.22066,2.48555)(2.20976,2.48167)(2.19886,2.47766)(2.18796,2.47354)(2.17707,2.46928)(2.16617,2.4649)(2.15527,2.4604)(2.14437,2.45577)(2.13347,2.45102)(2.12257,2.44614)(2.11167,2.44114)(2.10077,2.43602)(2.08988,2.43077)(2.07898,2.42539)(2.06808,2.41989)(2.05718,2.41427)(2.04628,2.40852)(2.03538,2.40264)(2.02448,2.39664)(2.01358,2.39051)(2.00269,2.38426)(1.99179,2.37788)(1.98089,2.37138)(1.96999,2.36475)(1.95909,2.35799)(1.94819,2.3511)(1.93729,2.34409)(1.92639,2.33695)(1.9155,2.32969)(1.9046,2.32229)(1.8937,2.31477)(1.8828,2.30712)(1.8719,2.29934)(1.861,2.29143)(1.8501,2.28339)(1.8392,2.27522)(1.82831,2.26692)(1.81741,2.25849)(1.80651,2.24993)(1.79561,2.24123)(1.78471,2.23241)(1.77381,2.22345)(1.76291,2.21436)(1.75201,2.20514)(1.74112,2.19578)(1.73022,2.18629)(1.71932,2.17666)(1.70842,2.16689)(1.69752,2.15699)(1.68662,2.14695)(1.67572,2.13678)(1.66482,2.12646)(1.65393,2.11601)(1.64303,2.10541)(1.63213,2.09468)(1.62123,2.0838)(1.61033,2.07278)(1.59943,2.06161)(1.58853,2.0503)(1.57763,2.03885)(1.56674,2.02724)(1.55584,2.01549)(1.54494,2.00359)(1.53404,1.99153)(1.52314,1.97932)(1.51224,1.96696)(1.50134,1.95445)(1.49044,1.94177)(1.47955,1.92893)(1.46865,1.91594)(1.45775,1.90278)(1.44685,1.88945)(1.43595,1.87596)(1.42505,1.86229)(1.41415,1.84845)(1.40325,1.83443)(1.39236,1.82024)(1.38146,1.80586)(1.37056,1.79129)(1.35966,1.77653)(1.34876,1.76158)(1.33786,1.74643)(1.32696,1.73107)(1.31606,1.7155)(1.30517,1.69971)(1.29427,1.6837)(1.28337,1.66746)(1.27247,1.65098)(1.26157,1.63425)(1.25067,1.61727)(1.23977,1.60001)(1.22887,1.58247)(1.21798,1.56464)(1.20708,1.54649)(1.19618,1.52801)(1.18528,1.50918)(1.17438,1.48997)(1.16348,1.47036)(1.15258,1.45031)(1.14168,1.42979)(1.13079,1.40874)(1.11989,1.38711)(1.10899,1.36483)(1.09809,1.34181)(1.08719,1.31793)(1.07629,1.29304)(1.06539,1.26692)(1.05449,1.23928)(1.0436,1.20964)(1.0327,1.17722)(1.0218,1.14047)(1.0109,1.0954)
};

		\draw [very thick] plot [smooth cycle]
coordinates{(-1.,-1.)(-1.0109,-0.923664)(-1.0218,-0.897573)(-1.0327,-0.879716)(-1.0436,-0.866106)(-1.05449,-0.8552)(-1.06539,-0.846207)(-1.07629,-0.838658)(-1.08719,-0.832247)
	(-1.09809,-0.826762)(-1.10899,-0.82205)(-1.11989,-0.817994)(-1.13079,-0.814504)(-1.14168,-0.811508)(-1.15258,-0.80895)(-1.16348,-0.806781)(-1.17438,-0.804962)(-1.18528,-0.803459)(-1.19618,-0.802243)(-1.20708,-0.801291)(-1.21798,-0.800581)(-1.22887,-0.800094)(-1.23977,-0.799814)(-1.25067,-0.799726)(-1.26157,-0.799818)(-1.27247,-0.800077)(-1.28337,-0.800495)(-1.29427,-0.801061)(-1.30517,-0.801766)(-1.31606,-0.802604)(-1.32696,-0.803568)(-1.33786,-0.80465)(-1.34876,-0.805845)(-1.35966,-0.807148)(-1.37056,-0.808554)(-1.38146,-0.810059)(-1.39236,-0.811657)(-1.40325,-0.813346)(-1.41415,-0.815121)(-1.42505,-0.81698)(-1.43595,-0.818919)(-1.44685,-0.820935)(-1.45775,-0.823026)(-1.46865,-0.825188)(-1.47955,-0.827421)(-1.49044,-0.829721)(-1.50134,-0.832086)(-1.51224,-0.834515)(-1.52314,-0.837005)(-1.53404,-0.839556)(-1.54494,-0.842164)(-1.55584,-0.84483)(-1.56674,-0.847551)(-1.57763,-0.850326)(-1.58853,-0.853153)(-1.59943,-0.856032)(-1.61033,-0.858962)(-1.62123,-0.861941)(-1.63213,-0.864968)(-1.64303,-0.868043)(-1.65393,-0.871164)(-1.66482,-0.874331)(-1.67572,-0.877542)(-1.68662,-0.880798)(-1.69752,-0.884097)(-1.70842,-0.887439)(-1.71932,-0.890823)(-1.73022,-0.894249)(-1.74112,-0.897716)(-1.75201,-0.901223)(-1.76291,-0.90477)(-1.77381,-0.908357)(-1.78471,-0.911983)(-1.79561,-0.915647)
	(-1.80651,-0.91935)(-1.81741,-0.923091)(-1.82831,-0.926869)(-1.8392,-0.930685)(-1.8501,-0.934538)(-1.861,-0.938428)(-1.8719,-0.942354)(-1.8828,-0.946317)(-1.8937,-0.950316)(-1.9046,-0.954352)(-1.9155,-0.958423)(-1.92639,-0.96253)(-1.93729,-0.966672)(-1.94819,-0.970851)(-1.95909,-0.975065)(-1.96999,-0.979314)(-1.98089,-0.983599)(-1.99179,-0.987919)(-2.00269,-0.992275)(-2.01358,-0.996667)(-2.02448,-1.00109)(-2.03538,-1.00556)(-2.04628,-1.01005)(-2.05718,-1.01459)(-2.06808,-1.01916)(-2.07898,-1.02376)(-2.08988,-1.0284)(-2.10077,-1.03308)(-2.11167,-1.0378)(-2.12257,-1.04255)(-2.13347,-1.04734)(-2.14437,-1.05216)(-2.15527,-1.05702)(-2.16617,-1.06192)(-2.17707,-1.06686)(-2.18796,-1.07184)(-2.19886,-1.07685)(-2.20976,-1.0819)(-2.22066,-1.08699)(-2.23156,-1.09212)(-2.24246,-1.0973)(-2.25336,-1.10251)(-2.26426,-1.10776)(-2.27515,-1.11305)(-2.28605,-1.11839)(-2.29695,-1.12376)(-2.30785,-1.12918)(-2.31875,-1.13464)(-2.32965,-1.14015)(-2.34055,-1.1457)(-2.35145,-1.1513)(-2.36234,-1.15694)(-2.37324,-1.16262)(-2.38414,-1.16836)(-2.39504,-1.17414)(-2.40594,-1.17997)(-2.41684,-1.18585)(-2.42774,-1.19178)(-2.43864,-1.19776)(-2.44953,-1.20379)(-2.46043,-1.20988)(-2.47133,-1.21601)(-2.48223,-1.22221)(-2.49313,-1.22846)
	(-2.50403,-1.23476)(-2.51493,-1.24113)(-2.52583,-1.24755)(-2.53672,-1.25404)(-2.54762,-1.26058)(-2.55852,-1.26719)(-2.56942,-1.27387)(-2.58032,-1.28061)(-2.59122,-1.28742)(-2.60212,-1.2943)(-2.61302,-1.30125)(-2.62391,-1.30827)(-2.63481,-1.31537)(-2.64571,-1.32255)(-2.65661,-1.32981)(-2.66751,-1.33715)(-2.67841,-1.34457)(-2.68931,-1.35208)(-2.70021,-1.35968)(-2.7111,-1.36737)(-2.722,-1.37516)(-2.7329,-1.38305)(-2.7438,-1.39104)(-2.7547,-1.39913)(-2.7656,-1.40734)(-2.7765,-1.41566)(-2.7874,-1.42409)(-2.79829,-1.43265)(-2.80919,-1.44134)(-2.82009,-1.45017)(-2.83099,-1.45913)(-2.84189,-1.46823)(-2.85279,-1.47749)(-2.86369,-1.48691)(-2.87459,-1.4965)(-2.88548,-1.50627)(-2.89638,-1.51622)(-2.90728,-1.52637)(-2.91818,-1.53672)(-2.92908,-1.5473)(-2.93998,-1.55811)(-2.95088,-1.56918)(-2.96178,-1.58051)(-2.97267,-1.59213)(-2.98357,-1.60406)(-2.99447,-1.61633)(-3.00537,-1.62897)(-3.01627,-1.64201)(-3.02717,-1.65549)(-3.03807,-1.66945)(-3.04897,-1.68396)(-3.05986,-1.69908)(-3.07076,-1.7149)(-3.08166,-1.73152)(-3.09256,-1.74908)(-3.10346,-1.76775)(-3.11436,-1.78779)(-3.12526,-1.80954)(-3.13616,-1.83356)(-3.14705,-1.86074)(-3.15795,-1.89287)
	(-3.16885,-1.93452)
	(-3.17975, -2.03365)
	(-3.16885,-2.12996)(-3.15795,-2.16874)(-3.14705,-2.19794)(-3.13616,-2.22215)(-3.12526,-2.24313)(-3.11436,-2.2618)(-3.10346,-2.27869)(-3.09256,-2.29417)(-3.08166,-2.30847)(-3.07076,-2.32177)(-3.05986,-2.33422)(-3.04897,-2.34592)(-3.03807,-2.35694)(-3.02717,-2.36736)(-3.01627,-2.37724)(-3.00537,-2.38661)(-2.99447,-2.39553)(-2.98357,-2.40402)(-2.97267,-2.41211)(-2.96178,-2.41983)(-2.95088,-2.4272)(-2.93998,-2.43424)(-2.92908,-2.44097)(-2.91818,-2.4474)(-2.90728,-2.45354)(-2.89638,-2.45942)(-2.88548,-2.46503)(-2.87459,-2.4704)(-2.86369,-2.47552)(-2.85279,-2.48042)(-2.84189,-2.48508)(-2.83099,-2.48953)(-2.82009,-2.49377)(-2.80919,-2.49781)(-2.79829,-2.50164)(-2.7874,-2.50528)(-2.7765,-2.50873)(-2.7656,-2.51199)(-2.7547,-2.51507)(-2.7438,-2.51798)(-2.7329,-2.52071)(-2.722,-2.52326)(-2.7111,-2.52566)(-2.70021,-2.52788)(-2.68931,-2.52994)(-2.67841,-2.53185)(-2.66751,-2.53359)(-2.65661,-2.53518)(-2.64571,-2.53662)(-2.63481,-2.5379)(-2.62391,-2.53903)(-2.61302,-2.54002)(-2.60212,-2.54086)(-2.59122,-2.54156)(-2.58032,-2.54211)(-2.56942,-2.54252)(-2.55852,-2.54279)(-2.54762,-2.54292)(-2.53672,-2.54291)(-2.52583,-2.54276)(-2.51493,-2.54248)(-2.50403,-2.54206)(-2.49313,-2.5415)(-2.48223,-2.54081)(-2.47133,-2.53999)(-2.46043,-2.53904)(-2.44953,-2.53795)(-2.43864,-2.53674)(-2.42774,-2.53539)(-2.41684,-2.53391)(-2.40594,-2.53231)(-2.39504,-2.53057)(-2.38414,-2.52871)(-2.37324,-2.52671)(-2.36234,-2.5246)(-2.35145,-2.52235)(-2.34055,-2.51997)(-2.32965,-2.51747)(-2.31875,-2.51485)(-2.30785,-2.51209)(-2.29695,-2.50921)(-2.28605,-2.50621)(-2.27515,-2.50308)(-2.26426,-2.49982)(-2.25336,-2.49644)(-2.24246,-2.49294)(-2.23156,-2.4893)(-2.22066,-2.48555)(-2.20976,-2.48167)(-2.19886,-2.47766)(-2.18796,-2.47354)(-2.17707,-2.46928)(-2.16617,-2.4649)(-2.15527,-2.4604)(-2.14437,-2.45577)(-2.13347,-2.45102)(-2.12257,-2.44614)(-2.11167,-2.44114)(-2.10077,-2.43602)(-2.08988,-2.43077)(-2.07898,-2.42539)(-2.06808,-2.41989)(-2.05718,-2.41427)(-2.04628,-2.40852)(-2.03538,-2.40264)(-2.02448,-2.39664)(-2.01358,-2.39051)(-2.00269,-2.38426)(-1.99179,-2.37788)(-1.98089,-2.37138)(-1.96999,-2.36475)(-1.95909,-2.35799)(-1.94819,-2.3511)(-1.93729,-2.34409)(-1.92639,-2.33695)(-1.9155,-2.32969)(-1.9046,-2.32229)(-1.8937,-2.31477)(-1.8828,-2.30712)(-1.8719,-2.29934)(-1.861,-2.29143)(-1.8501,-2.28339)(-1.8392,-2.27522)(-1.82831,-2.26692)(-1.81741,-2.25849)(-1.80651,-2.24993)(-1.79561,-2.24123)(-1.78471,-2.23241)(-1.77381,-2.22345)(-1.76291,-2.21436)(-1.75201,-2.20514)(-1.74112,-2.19578)(-1.73022,-2.18629)(-1.71932,-2.17666)(-1.70842,-2.16689)(-1.69752,-2.15699)(-1.68662,-2.14695)(-1.67572,-2.13678)(-1.66482,-2.12646)(-1.65393,-2.11601)(-1.64303,-2.10541)(-1.63213,-2.09468)(-1.62123,-2.0838)(-1.61033,-2.07278)(-1.59943,-2.06161)(-1.58853,-2.0503)(-1.57763,-2.03885)(-1.56674,-2.02724)(-1.55584,-2.01549)(-1.54494,-2.00359)(-1.53404,-1.99153)(-1.52314,-1.97932)(-1.51224,-1.96696)(-1.50134,-1.95445)(-1.49044,-1.94177)(-1.47955,-1.92893)(-1.46865,-1.91594)(-1.45775,-1.90278)(-1.44685,-1.88945)(-1.43595,-1.87596)(-1.42505,-1.86229)(-1.41415,-1.84845)(-1.40325,-1.83443)(-1.39236,-1.82024)(-1.38146,-1.80586)(-1.37056,-1.79129)(-1.35966,-1.77653)(-1.34876,-1.76158)(-1.33786,-1.74643)(-1.32696,-1.73107)(-1.31606,-1.7155)(-1.30517,-1.69971)(-1.29427,-1.6837)(-1.28337,-1.66746)(-1.27247,-1.65098)(-1.26157,-1.63425)(-1.25067,-1.61727)(-1.23977,-1.60001)(-1.22887,-1.58247)(-1.21798,-1.56464)(-1.20708,-1.54649)(-1.19618,-1.52801)(-1.18528,-1.50918)(-1.17438,-1.48997)(-1.16348,-1.47036)(-1.15258,-1.45031)(-1.14168,-1.42979)(-1.13079,-1.40874)(-1.11989,-1.38711)(-1.10899,-1.36483)(-1.09809,-1.34181)(-1.08719,-1.31793)(-1.07629,-1.29304)(-1.06539,-1.26692)(-1.05449,-1.23928)(-1.0436,-1.20964)(-1.0327,-1.17722)(-1.0218,-1.14047)(-1.0109,-1.0954)
};

	\end{tikzpicture}
	\caption{The biquadratic curve \eqref{eq:biquad12}, for $\alpha$ being the largest root of the polynomial $P_7$ which lies in $(-1/4,1/4)$. The approximate value of that root is $\alpha\approx0.21907$. Since the QRT transformation is of order $7$, the orbit of each point in the corresponding group of random walk consists of $14$ points. One orbit is shown. }
	\label{fig:Q14}
\end{figure}

\begin{example}[Random walks with group of order 16]\label{ex:order16}
Consider random walks with the matrix of the form as in Example \ref{ex:order10b}, see \eqref{eq:matrix10}.
Theorem \ref{th:cayley} gives that the condition of $8$-periodicity of the QRT transformation is equivalent to $P_8(\alpha)=0$, with
\begin{align*}
P_8(x)=&\
720^{12} x^{24}
+48\cdot720^{11} x^{23}
+12173088\cdot720^{10}x^{22}
+261726093\cdot6^5\cdot240^{10}x^{21}
\\&
-302090057\cdot162^2\cdot240^{10} x^{20}
-342221260129056\cdot1440^{7} x^{19}
+296459946313775748\cdot480^7 x^{18}
\\&
+1118545314702907752\cdot5760^5 x^{17}
-280331580296674545363\cdot11520^4x^{16}
\\&
-7373756024596349563986\cdot3840^4x^{15}
-252915006814253850702967235149824000 x^{14}
\\&
+852630401570800608206184851177472000 x^{13}
+1380087227897215680351747388211200 x^{12}
\\&
-94156379331381933840523329955430400 x^{11}
+17920120671486543086006490051379200 x^{10}
\\&
-7336759131940740122237302996992000 x^9
-1659061223344425240385415387873280 x^8
\\&
+3081652542867863291605737456992256 x^7
-336812496532651996260994153930752 x^6
\\&
-290159641899959711486508873031680 x^5
+56921711506423502988585171870720 x^4
\\&
+9632396294650986598746653863680 x^3
-2557401128040288900995832717216 x^2
\\&
-65643464566444059311731288128 x+29930352989898719170714742775.
\end{align*}
Among real roots of this polynomial, there are three which lie in $(-1/4,1/4)$, and the biquadratic curve \eqref{eq:biquad10b} is smooth of genus $1$ for each of them, thus here we have three examples of random walks with the group of order $16$, with the transition probabilities given in \eqref{eq:prob10}.
See illustration in Figure \ref{fig:Q16}.
\end{example}
	\begin{figure}[h]
	\centering
	\begin{tikzpicture}[scale=3]
		
\draw[very thin,color=gray!30,dashed] (-2.5,-1.5) grid (1.5,1.5);
		
		\draw[thick,->] (-2.5,0) -- (1.5,0) node[right] {$x$};
		\draw[thick,->] (0,-1.5) -- (0,1.5) node[above] {$y$};
		
	\draw[thick](-0.05,1)--(0.05,1)node[right]{$1$};
	\draw[thick](-0.05,-1)--(0.01,-1)node[right]{$-1$};

		\draw[thick](1,-0.05)--(1,0.05)node[above]{$1$};
		\draw[thick](-1,-0.05)--(-1,0.05)node[above]{$-1$};
		\draw[thick](-2,-0.05)--(-2,0.05)node[above]{$-2$};

		\draw [very thick] plot [smooth cycle] coordinates{
(-2.54679,-0.577146)(-2.52577,-0.632762)(-2.50475,-0.658088)(-2.48372,-0.678569)(-2.4627,-0.696606)(-2.44168,-0.713126)(-2.42066,-0.728603)(-2.39963,-0.743318)(-2.37861,-0.757455)(-2.35759,-0.77114)(-2.33656,-0.784466)(-2.31554,-0.797503)(-2.29452,-0.810307)(-2.2735,-0.822919)(-2.25247,-0.835377)(-2.23145,-0.847709)(-2.21043,-0.85994)(-2.18941,-0.872093)(-2.16838,-0.884184)(-2.14736,-0.89623)(-2.12634,-0.908246)(-2.10532,-0.920244)(-2.08429,-0.932234)(-2.06327,-0.944228)(-2.04225,-0.956234)(-2.02122,-0.968262)(-2.0002,-0.980318)(-1.97918,-0.99241)(-1.95816,-1.00454)(-1.93713,-1.01673)(-1.91611,-1.02896)(-1.89509,-1.04126)(-1.87407,-1.05362)(-1.85304,-1.06605)(-1.83202,-1.07856)(-1.811,-1.09114)(-1.78998,-1.1038)(-1.76895,-1.11655)(-1.74793,-1.12939)(-1.72691,-1.14232)(-1.70588,-1.15534)(-1.68486,-1.16846)(-1.66384,-1.18167)(-1.64282,-1.19499)(-1.62179,-1.2084)(-1.60077,-1.22192)(-1.57975,-1.23555)(-1.55873,-1.24927)(-1.5377,-1.2631)(-1.51668,-1.27704)(-1.49566,-1.29107)(-1.47464,-1.3052)(-1.45361,-1.31944)(-1.43259,-1.33376)(-1.41157,-1.34818)(-1.39055,-1.36268)(-1.36952,-1.37727)(-1.3485,-1.39193)(-1.32748,-1.40665)(-1.30645,-1.42143)(-1.28543,-1.43626)(-1.26441,-1.45113)(-1.24339,-1.46601)(-1.22236,-1.4809)(-1.20134,-1.49578)(-1.18032,-1.51063)(-1.1593,-1.52542)(-1.13827,-1.54013)(-1.11725,-1.55474)(-1.09623,-1.5692)(-1.07521,-1.58348)(-1.05418,-1.59755)(-1.03316,-1.61135)(-1.01214,-1.62484)(-0.991115,-1.63795)(-0.970092,-1.65062)(-0.949069,-1.66278)(-0.928047,-1.67436)(-0.907024,-1.68525)(-0.886001,-1.69535)(-0.864979,-1.70456)(-0.843956,-1.71274)(-0.822933,-1.71974)(-0.801911,-1.72542)(-0.780888,-1.72957)(-0.759865,-1.732)(-0.738843,-1.73247)(-0.71782,-1.73072)(-0.696797,-1.72642)(-0.675775,-1.71923)(-0.654752,-1.70874)(-0.63373,-1.69445)(-0.612707,-1.67577)(-0.591684,-1.65198)(-0.570662,-1.62215)(-0.549639,-1.58505)(-0.528616,-1.53893)(-0.507594,-1.48104)(-0.486571,-1.40635)(-0.465548,-1.30257)
(-0.444526,-1.04689)(-0.465548,-0.819538)(-0.486571,-0.739317)(-0.507594,-0.683801)(-0.528616,-0.641048)(-0.549639,-0.60638)(-0.570662,-0.577384)(-0.591684,-0.552617)(-0.612707,-0.531142)(-0.63373,-0.512307)(-0.654752,-0.49564)(-0.675775,-0.480785)(-0.696797,-0.467467)(-0.71782,-0.455471)(-0.738843,-0.444622)(-0.759865,-0.434777)(-0.780888,-0.425819)(-0.801911,-0.41765)(-0.822933,-0.410185)(-0.843956,-0.403353)(-0.864979,-0.397095)(-0.886001,-0.391356)(-0.907024,-0.38609)(-0.928047,-0.381258)(-0.949069,-0.376824)(-0.970092,-0.372755)(-0.991115,-0.369025)(-1.01214,-0.365609)(-1.03316,-0.362484)(-1.05418,-0.35963)(-1.07521,-0.357029)(-1.09623,-0.354666)(-1.11725,-0.352525)(-1.13827,-0.350593)(-1.1593,-0.348858)(-1.18032,-0.347309)(-1.20134,-0.345936)(-1.22236,-0.34473)(-1.24339,-0.343683)(-1.26441,-0.342787)(-1.28543,-0.342034)(-1.30645,-0.341419)(-1.32748,-0.340936)(-1.3485,-0.340578)(-1.36952,-0.340343)(-1.39055,-0.340224)(-1.41157,-0.340217)(-1.43259,-0.34032)(-1.45361,-0.340528)(-1.47464,-0.340839)(-1.49566,-0.341249)(-1.51668,-0.341756)(-1.5377,-0.342358)(-1.55873,-0.343052)(-1.57975,-0.343838)(-1.60077,-0.344713)(-1.62179,-0.345676)(-1.64282,-0.346726)(-1.66384,-0.347861)(-1.68486,-0.349083)(-1.70588,-0.350389)(-1.72691,-0.351779)(-1.74793,-0.353255)(-1.76895,-0.354814)(-1.78998,-0.356459)(-1.811,-0.35819)(-1.83202,-0.360007)(-1.85304,-0.361912)(-1.87407,-0.363906)(-1.89509,-0.36599)(-1.91611,-0.368167)(-1.93713,-0.370439)(-1.95816,-0.372808)(-1.97918,-0.375279)(-2.0002,-0.377853)(-2.02122,-0.380535)(-2.04225,-0.38333)(-2.06327,-0.386243)(-2.08429,-0.38928)(-2.10532,-0.392447)(-2.12634,-0.395753)(-2.14736,-0.399206)(-2.16838,-0.402816)(-2.18941,-0.406595)(-2.21043,-0.410557)(-2.23145,-0.414718)(-2.25247,-0.419097)(-2.2735,-0.423716)(-2.29452,-0.428604)(-2.31554,-0.433792)(-2.33656,-0.439324)(-2.35759,-0.445252)(-2.37861,-0.451643)(-2.39963,-0.458589)(-2.42066,-0.466214)(-2.44168,-0.4747)(-2.4627,-0.484326)(-2.48372,-0.495564)(-2.50475,-0.50934)(-2.52577,-0.528053)			
};
		
	\draw [very thick] plot [smooth cycle] coordinates{
(0.569046,0.96874)(0.574656,0.893962)(0.580266,0.863778)(0.585876,0.841178)(0.591486,0.822584)(0.597096,0.806595)(0.602706,0.792484)(0.608316,0.779816)(0.613927,0.768304)(0.619537,0.757747)(0.625147,0.747998)(0.630757,0.738945)(0.636367,0.730499)(0.641977,0.722592)(0.647587,0.715164)(0.653197,0.708169)(0.658807,0.701566)(0.664417,0.695322)(0.670028,0.689407)(0.675638,0.683796)(0.681248,0.678467)(0.686858,0.6734)(0.692468,0.668579)(0.698078,0.663987)(0.703688,0.659612)(0.709298,0.65544)(0.714908,0.651462)(0.720518,0.647666)(0.726129,0.644044)(0.731739,0.640588)(0.737349,0.637289)(0.742959,0.634142)(0.748569,0.631139)(0.754179,0.628275)(0.759789,0.625544)(0.765399,0.622942)(0.771009,0.620464)(0.776619,0.618106)(0.78223,0.615864)(0.78784,0.613734)(0.79345,0.611713)(0.79906,0.609798)(0.80467,0.607987)(0.81028,0.606276)(0.81589,0.604663)(0.8215,0.603146)(0.82711,0.601723)(0.832721,0.600392)(0.838331,0.599152)(0.843941,0.598)(0.849551,0.596936)(0.855161,0.595958)(0.860771,0.595066)(0.866381,0.594258)(0.871991,0.593533)(0.877601,0.592891)(0.883211,0.592331)(0.888822,0.591854)(0.894432,0.591457)(0.900042,0.591143)(0.905652,0.590909)(0.911262,0.590758)(0.916872,0.590688)(0.922482,0.590701)(0.928092,0.590797)(0.933702,0.590976)(0.939312,0.59124)(0.944923,0.591591)(0.950533,0.592028)(0.956143,0.592555)(0.961753,0.593172)(0.967363,0.593882)(0.972973,0.594687)(0.978583,0.595591)(0.984193,0.596595)(0.989803,0.597704)(0.995413,0.598921)(1.00102,0.600251)(1.00663,0.6017)(1.01224,0.603272)(1.01785,0.604974)(1.02346,0.606815)(1.02907,0.608801)(1.03468,0.610945)(1.04029,0.613256)(1.0459,0.615749)(1.05151,0.618439)(1.05712,0.621345)(1.06273,0.624491)(1.06834,0.627904)
(1.07395,0.631618)(1.07957,0.635677)(1.08518,0.640137)(1.09079,0.645074)(1.0964,0.650588)(1.10201,0.65683)(1.10762,0.664027)(1.11323,0.672563)(1.11884,0.683189)(1.12445,0.697842)(1.13006,0.736951)(1.12445,0.78136)(1.11884,0.801329)(1.11323,0.817286)(1.10762,0.831167)(1.10201,0.843723)(1.0964,0.855337)(1.09079,0.866238)(1.08518,0.876574)(1.07957,0.886446)(1.07395,0.89593)(1.06834,0.90508)(1.06273,0.913941)(1.05712,0.922546)(1.05151,0.930922)(1.0459,0.939092)(1.04029,0.947074)(1.03468,0.954885)(1.02907,0.962536)(1.02346,0.970038)(1.01785,0.977403)(1.01224,0.984636)(1.00663,0.991745)(1.00102,0.998737)(0.995413,1.00562)(0.989803,1.01239)(0.984193,1.01905)(0.978583,1.02562)(0.972973,1.03208)(0.967363,1.03845)(0.961753,1.04473)(0.956143,1.05091)(0.950533,1.057)(0.944923,1.063)(0.939312,1.06891)(0.933702,1.07473)(0.928092,1.08047)(0.922482,1.08611)(0.916872,1.09166)(0.911262,1.09713)(0.905652,1.1025)(0.900042,1.10778)(0.894432,1.11297)(0.888822,1.11807)(0.883211,1.12307)(0.877601,1.12797)(0.871991,1.13278)(0.866381,1.13749)(0.860771,1.14209)(0.855161,1.14659)(0.849551,1.15098)(0.843941,1.15526)(0.838331,1.15943)(0.832721,1.16349)(0.82711,1.16743)(0.8215,1.17124)(0.81589,1.17493)(0.81028,1.17849)(0.80467,1.18191)(0.79906,1.1852)(0.79345,1.18834)(0.78784,1.19134)(0.78223,1.19418)(0.776619,1.19686)(0.771009,1.19938)(0.765399,1.20173)(0.759789,1.20391)(0.754179,1.20589)(0.748569,1.20769)(0.742959,1.20929)(0.737349,1.21068)(0.731739,1.21185)(0.726129,1.21279)(0.720518,1.21349)(0.714908,1.21395)(0.709298,1.21415)(0.703688,1.21407)(0.698078,1.2137)(0.692468,1.21303)(0.686858,1.21203)(0.681248,1.2107)(0.675638,1.209)(0.670028,1.20692)(0.664417,1.20443)(0.658807,1.2015)(0.653197,1.1981)(0.647587,1.19419)(0.641977,1.18972)(0.636367,1.18464)(0.630757,1.17889)(0.625147,1.1724)(0.619537,1.16508)(0.613927,1.1568)(0.608316,1.14741)(0.602706,1.13672)(0.597096,1.12442)(0.591486,1.11008)(0.585876,1.09296)(0.580266,1.07166)(0.574656,1.04259)
};
		
\coordinate(r1)at(-1.5,-1.28816);
\coordinate(r2)at(-1.5,-0.341345);
\coordinate(r3)at(-1.30932,-0.341345);
\coordinate(r4)at(-1.30932,-1.41941);
\coordinate(r5)at(-0.489841,-1.41941);
\coordinate(r6)at(-0.489841,-0.72952);
\coordinate(r7)at(-2.41937,-0.72952);
\coordinate(r8)at(-2.41937,-0.465727);
\coordinate(r9)at(-0.699717,-0.465727);
\coordinate(r10)at(-0.699717,-1.72718);
\coordinate(r11)at(-0.793935,-1.72718);
\coordinate(r12)at(-0.793935,-0.420662);
\coordinate(r13)at(-2.25973,-0.420662);
\coordinate(r14)at(-2.25973,-0.831095);
\coordinate(r15)at(-0.463253,-0.831095);
\coordinate(r16)at(-0.463253,-1.28816);

\draw[fill=black](r1) circle (0.02);		
\draw[fill=black](r2) circle (0.02);		
\draw[fill=black](r3) circle (0.02);		
\draw[fill=black](r4) circle (0.02);		
\draw[fill=black](r5) circle (0.02);		
\draw[fill=black](r6) circle (0.02);		
\draw[fill=black](r7) circle (0.02);		
\draw[fill=black](r8) circle (0.02);		
\draw[fill=black](r9) circle (0.02);		
\draw[fill=black](r10) circle (0.02);		
\draw[fill=black](r11) circle (0.02);		
\draw[fill=black](r12) circle (0.02);		
\draw[fill=black](r13) circle (0.02);		
\draw[fill=black](r14) circle (0.02);		
\draw[fill=black](r15) circle (0.02);		
\draw[fill=black](r16) circle (0.02);

\draw(r1)--(r2)--(r3)--(r4)--(r5)--(r6)--(r7)--(r8)--(r9)--(r10)--(r11)--(r12)--(r13)--(r14)--(r15)--(r16)--cycle;

\coordinate(ra1)at(1,0.6);
\coordinate(ra2)at(1,1.);
\coordinate(ra3)at(0.570025,1.);
\coordinate(ra4)at(0.570025,0.937332);
\coordinate(ra5)at(1.04712,0.937332);
\coordinate(ra6)at(1.04712,0.616316);
\coordinate(ra7)at(0.781075,0.616316);
\coordinate(ra8)at(0.781075,1.19475);
\coordinate(ra9)at(0.648346,1.19475);
\coordinate(ra10)at(0.648346,0.714194);
\coordinate(ra11)at(1.12826,0.714194);
\coordinate(ra12)at(1.12826,0.7614);
\coordinate(ra13)at(0.617543,0.7614);
\coordinate(ra14)at(0.617543,1.16225);
\coordinate(ra15)at(0.83445,1.16225);
\coordinate(ra16)at(0.83445,0.6);

\draw[fill=black](ra1) circle (0.02);		
\draw[fill=black](ra2) circle (0.02);		
\draw[fill=black](ra3) circle (0.02);		
\draw[fill=black](ra4) circle (0.02);		
\draw[fill=black](ra5) circle (0.02);		
\draw[fill=black](ra6) circle (0.02);		
\draw[fill=black](ra7) circle (0.02);		
\draw[fill=black](ra8) circle (0.02);		
\draw[fill=black](ra9) circle (0.02);		
\draw[fill=black](ra10) circle (0.02);		
\draw[fill=black](ra11) circle (0.02);		
\draw[fill=black](ra12) circle (0.02);		
\draw[fill=black](ra13) circle (0.02);		
\draw[fill=black](ra14) circle (0.02);		
\draw[fill=black](ra15) circle (0.02);		
\draw[fill=black](ra16) circle (0.02);

\draw(ra1)--(ra2)--(ra3)--(ra4)--(ra5)--(ra6)--(ra7)--(ra8)--(ra9)--(ra10)--(ra11)--(ra12)--(ra13)--(ra14)--(ra15)--(ra16)--cycle;

	\end{tikzpicture}
	\caption{Example \ref{ex:order16}. The biquadratic curve \eqref{eq:biquad10b} and two orbits of the group generated by horizontal and vertical switches.
	Each orbit consists of $16$ points, since $\alpha$ is a root of polynomial $P_8$.
		That polynomial has three roots in $(-1/4,1/4)$ and for this illustration we chose the middle one, which has approximate value $\alpha\approx-0.12324$.
}
	\label{fig:Q16}
\end{figure}
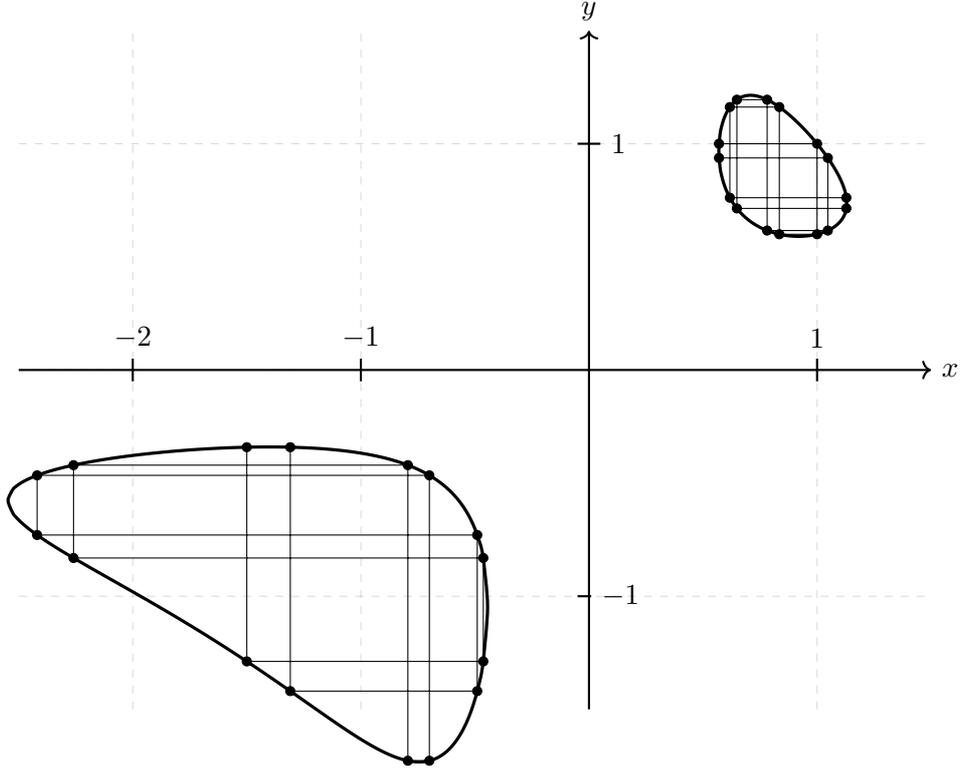

\section{Singular cases: QRT transformations and groups of random walks of a finite order}\label{sec:singular}

The discussion in Sections \ref{sec:smooth} and \ref{sec:small} relates to the \emph{nonsingular $(2,2)$ correspondences}, i.e.~the cases when the biquadratic curve is a smooth elliptic curve.
In order to complete that discussion, here we focus to singular cases, which happen when the curve is either an irreducible non-smooth biquadratic curve, or a reducible biquadratic curve.

\subsection{General considerations of singular cases}

A biquadratic curve is singular if it either contains some singular points or is reducible.
The possible case of irreducible singular biquadratics are:
\begin{itemize}
\item[(i)] irreducible curve with one ordinary double point. In this case we have: $d_1=d_2=(2, 1,1)$;
\item[(ii)] irreducible curve with one cusp point; $d_1=d_2=(3, 1)$.
\end{itemize}

Each reducible biquadratic in $\mathbb P^1\times \mathbb P^1$ is the union of the curves which can be of the following types:

\begin{itemize}
	\item $(1,1)$ correspondence in $\mathbb P^1\times \mathbb P^1$.
	Such a correspondence is the graph of a M\"obius transformation in $\mathbb P^1$ and the corresponding curve is a conic, which we denote by $\mathcal L$.
	Note that two such conics always have two common points, counting multiplicities, i.e.~they either intersect transversally at two point or have one point of tangency.
	\item If in \eqref{eq:biquad1}, we have $a(x)\equiv 0$, then the correspondence is  $(2,1)$ in  $\mathbb P^1\times \mathbb P^1$ and the we denote the corresponding twisted cubic curve by $\mathcal T_1$.
	\item Similarly, if in \eqref{eq:biquad1}, the coefficient $\tilde a(y)$ vanishes, then the correspondence is  $(1,2)$ in  $\mathbb P^1\times \mathbb P^1$ and we denote the corresponding twisted cubic curve by $\mathcal T_2$.
	\item Vertical lines $\{x_0\}\times \mathbb P^1$, which we denote by $\mathcal V$.
	\item Horizontal lines $\mathbb P^1\times \{y_0\}$, denoted by $\mathcal H$.
\end{itemize}

Now we can provide the exhaustive list of reducible biquadratics in
$\mathbb P^1\times \mathbb P^1$ with the types of their critical divisors:

\begin{itemize}
\item[(iii)] the  union of two conics $\mathcal{L}_1\cup\mathcal{L}_2$ with two common points; $d_1=d_2=(2,2)$;
\item[(iv)] the union of two conics $\mathcal{L}_1\cup\mathcal{L}_2$ with a point of tangency; $d_1=d_2=(4)$;
\item[(v)] a double conic $2\mathcal{L}$; $d_1$ and $d_2$ are undefined;
\item[(vi)] the union of a conic, and a horizontal and a vertical line $\mathcal{L}\cup\mathcal{H}\cup\mathcal{V}$, such that the intersection point of the lines does not belong to the conic; $d_1=d_2=(2,2)$;
\item[(vii)] the union of a conic, and a horizontal and a vertical line $\mathcal{L}\cup\mathcal{H}\cup\mathcal{V}$, such that the intersection point of the lines lies on the conic; $d_1=d_2=(4)$;
\item[(viii)] the union of two distinct horizontal and two distinct  vertical lines
$\mathcal{H}_1\cup\mathcal{H}_2\cup\mathcal{V}_1\cup\mathcal{V}_2$; $d_1=d_2=(2,2)$;
\item[(ix)] the union of one double horizontal and one double vertical line $2\mathcal{H}\cup2\mathcal{V}$; $d_1$ and $d_2$ are undefined;
\item[(x)] the union of two distinct horizontal and one double vertical line $\mathcal{H}_1\cup\mathcal{H}_2\cup2\mathcal{V}$; $d_1=(4)$, $d_2$ is undefined;
\item[(xi)] the union of a double horizontal and two distinct  vertical lines $2\mathcal{H}\cup\mathcal{V}_1\cup\mathcal{V}_2$; $d_1$ is undefined, $d_2=(4)$;
\item[(xii)] the union of a horizontal line and a $(1,2)$ correspondence defined by a twisted cubic
$\mathcal{H}\cup\mathcal T_2$, where $\mathcal{H}$ is not tangent to the cubic; $d_1=(2,2)$, $d_2=(2,1,1)$;
\item[(xiii)] the union of a horizontal line and a $(1,2)$ correspondence defined by a twisted cubic
$\mathcal{H}\cup\mathcal T_2$, where the line is tangent to the cubic; $d_1=(4)$, $d_2=(3,1)$;
\item[(xiv)] the union of a vertical line and a $(2,1)$ correspondence defined by a twisted cubic $\mathcal{V}\cup\mathcal T_1$, where the line is not tangent to the cubic; $d_1=(2,1,1)$, $d_2=(2,2)$;
\item[(xv)] the union of a vertical line and a $(2,1)$ correspondence defined by a twisted cubic $\mathcal{V}\cup\mathcal T_1$, where $\mathcal{V}$ is tangent to the cubic; $d_1=(3,1)$, $d_2=(4)$.
\end{itemize}

The following result goes back to Frobenius:

\begin{proposition}[\cites{Fr,Sam}] A $(2,2)$ correspondence in $\mathbb P^1\times \mathbb P^1$ that defines a singular curve is symmetrizable in cases (i)--(ix) and not symmetrizable in the remaining cases (x)--(xv).
\end{proposition}

\begin{example}
Let us consider a general irreducible nonsymmetric $(2,2)$ correspondence \eqref{eq:biquad} with a double point.

First, we demonstrate how to symmetrize such a curve.
Without losing generality, we can assume that the double point is placed at $(0,0)$, which means:
	$$
	a_{00}=a_{10}=a_{01}=0.
	$$
Applying the M\"obius transformation $y_1=y/(ry+s)$ we get:
	$$
	a_{22}x^2y^2+a_{21}x^2y(ry+s)+a_{20}x^2(ry+s)^2+a_{12}xy^2+a_{11}xy(ry+s)=0.
	$$
	The last correspondence is symmetric if the following relations hold:
	$$
	a_{21}s+2a_{20}rs-a_{12}-a_{11}r=0, \quad a_{20}s^2=a_{02}.
	$$
Notice that both $a_{20}$ and $a_{02}$ are nonzero, as otherwise the correspondence would be reducible.
Thus there are two distinct solutions in $s$ of the last equation, which are opposite to each other.
The first equation is linear in $r$, with the coefficient with $r$ equal to $2a_{20}s-a_{11}$.
This coefficient is nonzero for at least one of the two values for $s=\pm \sqrt{a_{02}/a_{20}}$, so we conclude that there is a pair $(r,s)$ such that the M\"obius transformation maps the curve to a symmetric one.
	
For $2a_{20}s=a_{11}$, we get $4a_{20}a_{02}=a_{11}^2$, which gives that the singularity at the origin is a cusp.
Otherwise, the curve has an ordinary double point there.
\end{example}

\begin{theorem}\label{th:doublepoint}
Consider a $(2,2)$ correspondence with a double point at $(0,0)$:
\begin{equation}\label{eq:singular}
\mathcal{C}:\ a_{22}x^2y^2+a_{21}x^2y+a_{20}x^2+a_{12}xy^2+a_{11}xy=0.
\end{equation}
Then we have:
\begin{itemize}
\item
If $a_{11}^2\ne4a_{20}a_{02}$, then the singularity at the origin is an ordinary double point and the QRT transformation of $\mathcal C$ is $n$-periodic if and only if there exists a natural number $m$, such that
\begin{equation}\label{eq:nsing}
\frac{a_{11}^2}{4a_{20}a_{02}}=\cos^2\Big(\frac{m}{n}\pi\Big).
\end{equation}
\item
If $a_{11}^2=4a_{20}a_{02}$, then the singularity is a cusp and the QRT transformation of $\mathcal C$ is not periodic.
\end{itemize}
\end{theorem}

\begin{proof}
We calculate the invariants of the biquadratic curve \eqref{eq:singular}:
$$
D=\frac{1}{12} \left(a_{11}^2-4 a_{02} a_{20}\right)^2,
\quad
E=\frac{1}{216} \left(a_{11}^2-4 a_{02} a_{20}\right)^3,
$$
so we obtain the equation of the corresponding cubic:
$$
\Gamma_s: Y^2
=
\frac{1}{216}
\left(12X+4 a_{02} a_{20}-a_{11}^2\right)^2
\left(6 X-4 a_{02} a_{20}+a_{11}^2\right).
$$
Note that, indeed, this curve has a double point with the coordinates $(X_d,Y_d)=((a_{11}^2-4a_{02}a_{20})/12,0)$, which is a cusp for $a_{11}^2-4a_{02}a_{20}=0$ or an ordinary double point otherwise.

Now, the QRT-transformation on the original biquadratic curve \eqref{eq:singular} is $n$-periodic if and only if the shift by the divisor $Q_0-Q_{\infty}$ on the obtained cubic $\Gamma_s$ is of order $n$, where $Q_0$ is the point with coordinates $(X_0,Y_0)$:
$$
X_0=\frac{1}{12} \left(8 a_{02} a_{20}+a_{11}^2\right),
\quad
Y_0=-a_{02} a_{11} a_{20},
$$	
while $Q_{\infty}$ is the point at infinity.

The normalization of  $\Gamma_s$  is:
$$
\tilde{Y}^2=6\tilde{X}-4 a_{02} a_{20}+a_{11}^2,
$$
with
$
\pi:
(\tilde{X},\tilde{Y})
\mapsto
(X,Y)
=
\left(
\tilde X,\tilde{Y}(12\tilde{X}+4a_{02}a_{20}-a_{11}^2)/\sqrt{216}
\right).
$
The point $(X_0,Y_0)$ is the image of
$$
(\tilde{X}_0,\tilde{Y}_0)
=
\left(
X_0,\frac{Y_0\sqrt{216}}{12X_0+4a_{02}a_{20}-a_{11}^2}
\right)
=
\left(
\frac{1}{12} \left(8 a_{02} a_{20}+a_{11}^2\right),
\frac{-a_{11}\sqrt{6}}{2}
\right),
$$
while the preimages of $(X_d,Y_d)$ are
$
\left(
(a_{11}^2-4a_{02}a_{20})/12,\pm\sqrt{3(a_{11}^2-4a_{02}a_{20})/2}
\right)
$.
Thus, the $n$-pe\-riodi\-ci\-ty of the QRT map is equivalent to the condition:
$$
\left(
\frac{a_{11}\sqrt{6}}{2}
+
\sqrt{\frac{3(a_{11}^2-4a_{02}a_{20})}2}
\right)^n
=
\left(
\frac{a_{11}\sqrt{6}}{2}
-\sqrt{\frac{3(a_{11}^2-4a_{02}a_{20})}2}
\right)^n,
$$
which is equivalent to:
$$
\left(
a_{11}
+
\sqrt{a_{11}^2-4a_{02}a_{20}}
\right)^n
=
\left(
a_{11}
-
\sqrt{a_{11}^2-4a_{02}a_{20}}
\right)^n.
$$
A direct calculation gives \eqref{eq:nsing}.
The second item follows from there as well.
\end{proof}

\begin{remark}
It is interesting to note that in  Theorem \ref{th:doublepoint}, the conditions, such as  \eqref{eq:nsing}, do not depend on $a_{22}, a_{21}$, and $a_{12}$.
\end{remark}

\begin{example}
Consider the following biquadratic curve:
$$
x^2y^2+2x^2y+x^2+3xy^2-xy+y^2=0.
$$
We have $a_{00}=a_{01}=a_{10}=0$ and
$\dfrac{a_{11}^2}{4a_{20}a_{02}}=\cos^2(\pi/3)$, thus Theorem \ref{th:doublepoint} implies that the QRT-transformation is $3$-periodic.
Indeed, by a direct calculation one obtains that consecutive iterations of the QRT-map give the following points:
$$
\left(-1,\frac{3-\sqrt{13}}{2} \right),
\left(\frac{-\sqrt{13}-7}{18},
\frac{3+\sqrt{13}}{2}\right),
\left(\frac{\sqrt{13}-7}{18},-\frac{1}{4}\right),
\left(-1,\frac{3-\sqrt{13}}{2} \right),
\dots.
$$
\end{example}

\begin{example}\label{ex:double-point}
Consider a cubic curve of the form $4\mu^2=(1+\lambda)(1+\alpha\lambda)^2$, where  $\alpha\neq1$ is a constant.

Denote by $P_0$ and $P_{\infty}$ the points with coordinates $(\lambda,\mu)=(0,1/4)$ and $(\lambda,\mu)=(\infty,\infty)$ respectively.
According to \cite{FlattoBOOK} (see also \cite{DR2025rcd}*{Theorem 2.7}), the shift by the divisor $P_0-P_{\infty}$ is of order $n$ if and only if $\alpha=\cos^2\dfrac{\pi m}{n}$, where $m$ and $n$ are positive integers.

The homogeneous equation in the projective plane of the curve is:
$4\mu^2\nu=(\nu+\lambda)(\nu+\alpha\lambda)^2$.
Taking the change:
$[\lambda:\mu:\nu]\mapsto[x:y:t]$, with $x=\nu$, $y=2\mu-\nu$, $t=\lambda$,
we get, in the affine chart $t=1$, the following curve:
$2 x^2 y + x y^2-(2\alpha+1)x^2  -\alpha(\alpha+2)x-\alpha^2  =0$, which is irreducible curve with ordinary double point $(-\alpha,\alpha)$, thus of type (i).
Writing the equation in the chart $(1/x,1/y)$, the transformation from Proposition \ref{prop:transformation} then gives:
$$
Y^2
=
 -(4/27) (\alpha^2 -\alpha  - 3 X)^2
(2 \alpha^2-2 \alpha  + 3 X).
$$

Denote by $Q_0$ one of the points with of the curve with $X$-coordinate equal to zero and by $Q_{\infty}$ the points at the infinity.
According to Proposition \ref{prop:deltagamma}, the QRT transformation is equivalent to the shift by the divisor $Q_0-Q_{\infty}$.
Note that the mapping
$
(\mu,\nu)
=
\left(
\dfrac{X}{\alpha^2}
-
\dfrac{\alpha+2}{3 \alpha}
,
\dfrac{Y}{4i\alpha^2}
\right)
$
transforms the last cubic curve into the initial one, while taking points $Q_0$, $Q_{\infty}$ to $P_0$, $P_{\infty}$ respectively.

\end{example}

\begin{example}
Now, consider the cubic curve from Example \ref{ex:double-point}, but with $\alpha=1$: $4\mu^2=(1+\lambda)^3$.
The transformation from that example gives: $2 x^2 y-3 x^2+x y^2-3 x-1=0$, which is irreducible curve with cusp at $(-1,1)$, thus it is of type (ii).
Translating the coordinate system so that the cusp will be at the origin, we get the equation: $2 x^2 y+xy^2-x^2-2 x y-y^2=0$, where $a_{02}=a_{20}=-1$ and $a_{11}=-2$, so, as expected, this example agrees with the case (ii) of Theorem \ref{th:doublepoint}, so the QRT-transformation is not periodic there.
That also agrees with
 \cite{FlattoBOOK}*{Theorem 11.7}.
\end{example}

Theorem \ref{th:doublepoint} can be used to find the condition for finite groups of some special cases of random walks.
For example, some such cases were discussed in
\cite{FayRas}*{Theorem 1.4}, and we show in Example \ref{ex:rangen0} and Corollary \ref{cor:connection} how that condition can be derived from our Theorem \ref{th:doublepoint}.

\begin{example}[Random walks with zero drift]\label{ex:rangen0}
\emph{The drift $M$} of a random walk is defined as:
$$
{\bf M}=\left(\sum_{-1\le j,k\le 1}jp_{jk}, \sum_{-1\le j,k\le 1}kp_{jk}\right).
$$
The condition ${\bf M}={\bf 0}$ implies that the underlying biquadratic is of genus $0$, see \cite{FayRas,RandomWalks}.
Denote by $R$ the correlation coefficient of the random walk, defined by
$$
R=\frac{\sum_{-1\le j,k\le 1}jkp_{jk}}{(\sum_{-1\le j,k\le 1}j^2p_{jk})^{1/2}(\sum_{-1\le j,k\le 1}k^2p_{jk})^{1/2}}
$$
and the angle $\theta$:
$$
\theta=\arccos(-R).
$$
Then \cite{FayRas}*{Theorem 1.4}, see also \cite{RandomWalks}*{Theorem 7.1}, states that for ${\bf M}=0$, the group of random walk is finite if and only if $\theta/\pi$ is rational and in that case the order is equal to
$$
2\min\{\ell\in\mathbb Z^+\mid\ell\theta/\pi \in \mathbb Z\}.
$$

In Corollary \ref{cor:connection} below, we show that these results about random walks with the zero drift follow from Theorem \ref{th:doublepoint}.
A random walk is called singular in \cite{RandomWalks} if its biquadratic is either reducible or of degree $1$ in at least one of the variables. In the nonsingular case, according to \cite{RandomWalks}*{Lemma 2.3.10}, the biquadratic is of genus zero if and only if one of the following conditions is satisfied:
\begin{itemize}
\item ${\bf M}=0$;
\item $p_{10}=p_{11}=p_{01}=0$;
\item
$p_{10}=p_{1,-1}=p_{0,-1}=0$;
\item
$p_{-1,0}=p_{-1,-1}=p_{0,-1}=0$;
\item
$p_{01}=p_{-1,0}=p_{-1,1}=0$.
\end{itemize}
According to \cite{RandomWalks}*{Theorem 7.1}, in all four listed cases  with ${\bf M}\ne {\bf 0}$, the groups of random walks are of infinite order.
\end{example}

\begin{corollary}\label{cor:connection}
	Consider a biquadratic curve \eqref{eq:biquad} satisfying:
	\begin{equation}\label{eq:drift0}
		\begin{aligned}
			&a_{00}+a_{01}+a_{02}+a_{10}+a_{11}+a_{12}+a_{20}+a_{21}+a_{22}=0,
			\\
			&a_{00}+a_{01}+a_{02}=a_{20}+a_{21}+a_{22},
			\\
			&a_{00}+a_{10}+a_{20}=a_{02}+a_{12}+a_{22}.
		\end{aligned}
	\end{equation}
	Then, the QRT map is $n$-periodic if and only if:
	$$
	\frac{(a_{00}-a_{02}-a_{20}+a_{22})^2}
	{4(a_{20}+a_{21}+a_{22})(a_{02}+a_{12}+a_{22})}
	=
	\cos^2\left(\frac{m\pi}n\right).
	$$
\end{corollary}

\begin{proof}
	It can be straightforwardly calculated that $(x,y)=(1,1)$ is a double point of that biquadratic curve.
	
	Moreover, let $(\tilde{x},\tilde{y})=(x-1,y-1)$ be a new coordinate system.
	Then the equation of the biquadratic curve in that coordinate system is:
	\begin{align*}
		0=&(a_{11} +2 a_{12}+2 a_{21}+4 a_{22}) \tilde{x} \tilde{y}
		+
		(a_{20}+a_{21}+a_{22}) \tilde{x}^2
		+
		(a_{02}+a_{12}+a_{22}) \tilde{y}^2
		\\&
		+
		(a_{12}+2 a_{22}) \tilde{x} \tilde{y}^2
		+
		(a_{21}+2 a_{22}) \tilde{x}^2 \tilde{y}
		+
		a_{22} \tilde{x}^2 \tilde{y}^2.
	\end{align*}
	The relations \eqref{eq:drift0} imply the following for the coefficient multiplying $\tilde{x}\tilde{y}$:
	$$
	a_{11} +2 a_{12}+2 a_{21}+4 a_{22}
	=
	a_{00}-a_{02}-a_{20}+a_{22}.
	$$
	Now, applying Theorem \ref{th:doublepoint}, we get that the QRT map is $n$-periodic if and only if:
	$$
	\frac{(a_{00}-a_{02}-a_{20}+a_{22})^2}
	{4(a_{20}+a_{21}+a_{22})(a_{02}+a_{12}+a_{22})}
	=
	\cos^2\left(\frac{m\pi}n\right).
	$$
\end{proof}

	\begin{example}
		In the projective plane, suppose that a smooth conic $\mathcal{C}$ and two points $C_1$, $C_2$ are given, such that the points do not lie on the conic.
		Consider polygons inscribed in $\mathcal{C}$ such that its sides alternately contain $C_1$ and $C_2$.
		Such polygons were discussed and their periodicity analysed in \cite{DR2025rcd}*{Section 3.1}.
		
		Here, we choose a coordinate system such that $C_1$, $C_2$ are the points at the infinity corresponding to the horizontal and vertical directions.
		Then, notice that the conic $\mathcal{C}$ can be considered as a biquadratic curve and that the sides of the polygons correspond to the horizontal and vertical switches on $\mathcal{C}$.
		The matrix of that biquadratic curve is of the form:
		$$
		\begin{pmatrix}
			0 & 0 & a_{20}\\
			0 & a_{11} & a_{10}\\
			a_{02} & a_{01} & a_{00}
		\end{pmatrix},
		$$
		with $a_{20}a_{02}\neq0$.
		According to Theorem \ref{th:doublepoint}, the condition for $n$-periodicity of the QRT-transformation is given by \eqref{eq:nsing}.
		Notice that this equation also implies $a_{11}^2<4a_{20}a_{02}$, which means that, in the chosen coordinate system, conic $\mathcal{C}$ is an ellipse, or equivalently, that the line $C_1C_2$ is disjoint with $\mathcal{C}$, which is in agreement with \cite{DR2025rcd}*{Proposition 3.4}.

		Now, let us apply an affine transformation which fixes point $C_1$ and maps the conic $\mathcal{C}$ to a circle.
		For example, the transformations of the following form fix the point $C_1$:
		$$
		x=\alpha x_1+\beta y_1,
		\quad
		y=y_1,
		\quad\text{with }\alpha\neq0.
		$$
The quadratic terms in the equation of $\mathcal{C}$ after the transformation are:
$$
a_{20}(\alpha x_1+\beta y_1)^2
+
a_{11}(\alpha x_1+\beta y_1)y_1
+
a_{02}y_1^2
=
a_{20}\alpha^2 x_1^2
+
\alpha(2a_{20}\beta+a_{11})x_1y_1
+
(a_{20}\beta^2+a_{11}\beta+a_{02})y_1^2,
$$
so we got a circle if and only if:
$$
a_{20}\alpha^2=a_{20}\beta^2+a_{11}\beta+a_{02}
\quad\text{and}\quad
2a_{20}\beta+a_{11}=0,
$$
so we get:
$$
\alpha=\pm\frac{\sqrt{4 a_{02} a_{20}-a_{11}^2}}{2 a_{20}},
\quad
\beta=-\frac{a_{11}}{2a_{20}}.
$$
That transformation maps point $C_2$, which has projective coordinates $[0:1:0]$ to the point with coordinates:
$$
\left[\pm\frac{a_{11}}{\sqrt{4 a_{02} a_{20}-a_{11}^2}}:1:0\right].
$$
Thus, according to \cite{DR2025rcd}*{Proposition 3.7}, Poncelet polygons inscribed in $\mathcal{C}$ and circumscribed about the pair of points $C_1,C_2$ have $2n$ sides if and only if:
$$
\arctan\left(\frac{\sqrt{4 a_{02} a_{20}-a_{11}^2}}{|a_{11}|}\right)
\in
\left\{
\frac{k\pi}{n}
\mid
1\le k<2n,\ (k,2n)=1
\right\}.
$$
Thus we have:
$$
\tan^2\left(\frac{k\pi}{n}\right)
=
\frac{4a_{02}a_{20}}{a_{11}^2}-1.
$$
Applying a trigonometric identity, we get:
$$
\cos^2\left(\frac{k\pi}n\right)
=
\frac{1}
{1+\tan^2\left(\frac{k\pi}{n}\right)}
=
\frac{a_{11}^2}{4a_{02}a_{20}},
$$
which is exactly \eqref{eq:nsing}.
\end{example}

We covered by the previous statements and examples in this section the cases of irreducible biquadratic curves.
Now, we will discuss the reducible ones.

All cases (vi)--(xv) in the list above contain a vertical or a horizontal line. Thus, in each of these cases, there is no meaningful way to define a QRT transformation, and consequently, it is not possible to study periodicity of the QRT transformation in such cases.
Thus, we consider here separately each of the cases (iii)--(v).

\paragraph*{Case (iii).} Consider two M\"obius transformations
$$
\phi_j(u)=\frac{\alpha_j u +\beta_j}{\gamma_j u + \delta_j}, \quad j=1,2,
$$
associated with the singular $(2,2)$ correspondence
\begin{equation}\label{eq:iii}
\mathcal C_{iii}: (\alpha_1 u +\beta_1-\gamma_1 uv -\delta_1 v)(\alpha_2 u +\beta_2-\gamma_2 uv -\delta_2 v)=0.
\end{equation}

The QRT transformation in this case is:
\begin{equation*}%\label{eq:iiiQRT}
(u_1, v_1)
=
(u_1,\phi_1(u_1))
\mapsto
\left(\phi_2^{-1}(\phi_1(u_1)), \phi_1(u_1)\right)
\mapsto
\left(
\phi_2^{-1}(\phi_1(u_1)),
\left(\phi_1\left(\phi_2^{-1}(\phi_1(u_1))\right)\right)
\right).
\end{equation*}
\begin{proposition}\label{prop:ordQRTsing} The QRT transformation in  case (iii) has a period $N$ for all $u_1$ if and only if $(\phi_2^{-1}\circ \phi_1)^N=\Id$.
\end{proposition}
\begin{proof}
We have:
\begin{equation}\label{eq:equiv}
(\phi_2^{-1}\circ \phi_1)^N(u_1)
=
u_1
\quad \Longleftrightarrow \quad
\phi_1\left((\phi_2^{-1}\circ \phi_1)^N(u_1)\right)
=
\phi_1(u_1)
\quad \Longleftrightarrow \quad
(\phi_1\circ \phi_2^{-1})^N(\phi_1(u_1))
=
\phi_1(u_1).
\end{equation}
Thus, $(\phi_2^{-1}\circ \phi_1)^N=\Id$ if and only if $(\phi_1\circ \phi_2^{-1})^N=\Id$. The last two equivalent conditions are necessary and sufficient that the $N$-th iteration of the QRT map is the identity.
\end{proof}

From \eqref{eq:equiv} we see that in case (iii) it is possible for an $N$-th iteration of the QRT transformation to have a fixed point $(u_1,\phi_1(u_1))$ as soon as $u_1$ is a fixed point of $(\phi_2^{-1}\circ \phi_1)^N$.

\begin{example}\label{ex:xy-axes}
We consider \cite{DR2025rcd}*{Section 3.3}.
We fix two lines in the real plane, say $x$-axis and $y$-axis and two points $C_1=(p_1, q_1)$, $C_2=(p_2, q_2)$ which do not lie on the axes, $p_1q_1p_2q_2\neq0$.
We define two transformations $\phi_1$ and $\phi_2$ as the projections of the $x$-axis to the $y$-axis from $C_1$ and $C_2$ respectively.

Then
$$
\phi_j(x)=\frac{q_jx}{x-p_j}, \quad j=1, 2,
$$
and
$$
\phi_2^{-1}(y)=\frac{p_2y}{y-q_2}.
$$
The associated biquadratic curve is decomposable:
$$
{\mathcal C}_{iii}\ :\ (xy-p_1y-q_1x)(xy-p_2y-q_2x)=0.
$$
According to Proposition \ref{prop:ordQRTsing}, the QRT transformation has period $N$ if and only if
$$
\delta(x)=(\phi_2^{-1}\circ\phi_1)(x)=\frac{p_2q_1x}{(q_1-q_2)x+q_2p_1},
$$
has the order $N$.
Since one can show by induction that:
$$
\delta^N(x)
=
\frac{(p_2q_1)^Nx}
{(q_1-q_2)\frac{(p_2q_1)^N-(q_2p_1)^N}{p_2q_1-q_2p_1}x+(q_2p_1)^N},
$$
we have that $\delta^N=\Id$ implies $(p_2q_1)^N=(q_2p_1)^N$.
Since we deal here with the real numbers, that implies $(p_2q_1)^2=(q_2p_1)^2$, thus $\delta^2=\Id$,
which
is in alignment with \cite{DR2025rcd}*{Section 3.3}.
We came to the conclusion that
the QRT transformation is an identity if and only if $q_1=q_2$ and $p_2=p_1$, while it is of order $N=2$ if and only if $p_2q_1=\pm q_2p_1$ and $(p_1,q_1)\neq(p_2,q_2)$.

Consider the case $p_1q_2=1=-p_2q_1$.
The intersections of the line $C_1C_2$ with the coordinate axes are points:
$$
D_1=\Big(\frac{2p_1}{1-q_1p_1}, 0\Big)
\quad\text{and}\quad
D_2=\Big(0, \frac{2q_1}{q_1p_1+1}\Big),
$$
which are shown in Figure \ref{fig:xy-C1C2}.
	\begin{figure}[h]
	\begin{center}
		\begin{tikzpicture}[scale=1]

\draw[thick,gray!50] (-2,1)--(0,-1.5)--(1,0.5);

\draw[thick,->](-3,0)--(6,0)node[right]{$x$};
\draw[thick,->](0,-2)--(0,2)node[above]{$y$};
\draw[very thick](5,-0.166667)--(-3,1.166667);

\draw[fill=black] (1,0.5) circle [radius=0.07] node [above]{$C_1$};
\draw[fill=black] (-2,1) circle [radius=0.07] node [above]{$C_2$};
\draw[fill=black] (0,0.666667) circle [radius=0.07] node [below left]{$D_2$};
\draw[fill=black] (4,0) circle [radius=0.07] node [below]{$D_1$};

		\end{tikzpicture}
	\end{center}\caption{Example \ref{ex:xy-axes}: The pairs $C_1$, $C_2$ and $D_1$, $D_2$ are harmonically conjugated.}\label{fig:xy-C1C2}
\end{figure}
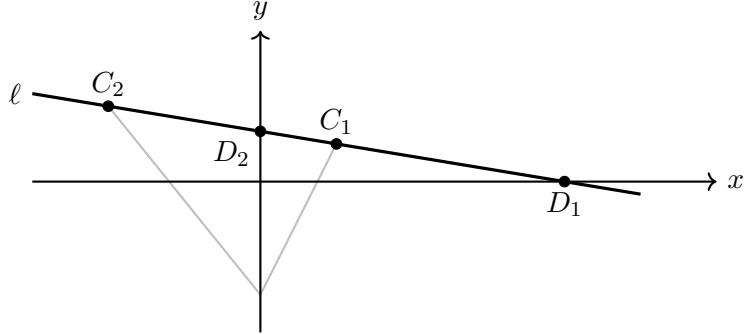

In the accordance with \cite{DR2025rcd}*{Section 3.3}, we can verify by direct calculation that the pair $(D_1, D_2)$ is harmonically-conjugated with the pair $(C_1, C_2)$.
Finally, in this case the biquadratic curve $\mathcal C_{iii}$ takes the form:
$$
x^2y^2-\frac{p_1q_1+1}{p_1}x^2y-\frac{p_1q_1-1}{q_1}xy^2 +\frac{q_1}{p_1}x^2-\frac{p_1}{q_1}y^2=0.
$$
That is not a random walk biquadratic, because, for example, the coefficients with $x^2$ and $y^2$ are of opposite signs.
\end{example}

\paragraph*{Case (iv).}
We consider the case where two conics given by \eqref{eq:iii} are tangent to each other.
Without loss of generality, we may assume that a point of their intersection is the origin.
Under that assumption, the two conics are:
$$
\alpha_1 u-uv-\delta_1 v=0
\quad\text{and}\quad
\alpha_2 u-uv-\delta_2 v=0.
$$
The tangency condition additionally gives: $\alpha_1/\delta_1=\alpha_2/\delta_2$, i.e.~there is $\lambda\neq0$ such that $\alpha_2=\lambda\alpha_1$ and $\delta_2=\lambda\delta_1$.
Thus the corresponding biquadratic curve is:
$$
\alpha_1^2 \lambda u^2
-2 \alpha_1 \delta_1 \lambda u v
-\alpha_1 (\lambda+1) u^2 v
+\delta_1^2 \lambda v^2
+\delta_1(\lambda+1)u v^2
+u^2 v^2
=0.
$$
A direct verification shows that $d_1=d_2=(4)$.

\paragraph*{Case (v).} In this case $\phi_1=\phi_2$ and the QRT map is the identity.

\subsection{Applications to enumerative combinatorics}\label{sec:comb}

The enumeration  of lattice walks occupies an important part of enumerative combinatorics. There has been a significant progress made recently in the quite complex study of lattice walks in the quarter plane. The kernel method and the group of random walks played prominent roles in this advancement. From \cite{BMM}, it is known that there are $79$ nonequivalent nontrivial walks with small steps in the quarter plane. We are going to provide new independent proofs of some of the results; see \cites{BMM, KuRas,KaYa}.

Denote by $\mathcal{S}\subset \{-1, 0, 1\}^2\setminus \{(0,0)\}$, the set of one-step vectors that defines a given walk in the quarter plane.
\emph{The generating polynomial} of the steps in $\mathcal{S}$ is:
$$
S(x,y)=\sum_{(i,j)\in\mathcal{S}}x^iy^j.
$$
One can assign three groups with every $S$:
\begin{itemize}
\item $G(\mathcal S)$ as the group generated by two involutions $\alpha$ and $\beta$ that keep the function $S(x, y)$ invariant:
$$
\alpha(x, y) = (x, y^{-1}A_{-1}(x)A^{-1}_1(x)),
\quad
\beta(x,y)=(x^{-1}B_{-1}(y)B^{-1}_1(y), y),
$$
where
$$
S(x, y)=A_{-1}(x)y^{-1}+A_0+A_1(x)y=B_{-1}(y)x^{-1}+B_0(y)+B_1(y)x.
$$
In \cite{BMM}, $G(\mathcal S)$ is called \emph{the group of the walk};
\item
$W(\mathcal S)$ as the group generated by the horizontal and vertical switches of the curve $xyS(x, y)=0$; and
\item
$\mathcal H(\mathcal S, t)$ as the group generated by the horizontal and vertical switches of the curve $\mathcal K_t: xy(1-tS(x, y))=0$.
\end{itemize}
As mentioned in \cite{RandomWalks}, the order of $\mathcal H(\mathcal{S}, t)$, for $t\ne 0$ is less or equal to the order of $G(S)$.
It was shown in  \cite{BMM}, that there are $23$ out of $79$ walks $S$ for which the group $G(S)$ is finite.
We will refer to these walks as $\mathcal S_j$, $j=1,\dots, 23$, according to \cite{BMM}*{Tables 1, 2, and 3}.

\begin{proposition}\label{prop:finite}
For each $j\in\{1,\dots,23\}$:
\begin{itemize}
\item[(i)] \cite{BMM}
The group $G(\mathcal S_j)$ of the walks in the quarter plane is finite. Moreover, its order is:
$$
|G(\mathcal S_j)|
=
\begin{cases}
4, &\text{for }1\le j\le16;\\
6, &\text{for }17\le j\le21;\\
8, &\text{for } j\in\{22,23\}.
\end{cases}
$$
\item[(ii)]
For the walks in the quarter plane, the group $W(\mathcal S_j)$ is finite and the orders of the groups $\mathcal H(\mathcal S_j, t)$ for $t\ne 0$ do not depend on $t$.
The orders of those groups are:
$$
|W(\mathcal S_j)|=|\mathcal{H}(\mathcal S_j,t)|
=
\begin{cases}
4, &\text{for }1\le j\le16;\\
6, &\text{for }17\le j\le21;\\
8, &\text{for }22\le j\le23.
\end{cases}
$$
Note that $j\in\{22,23\}$, the curve $xyS_j(x,y)=0$ has a horizontal component,  thus the group $W(\mathcal S_j)$ is defined on the other component of the curve.
All the curves ${\mathcal K}_{j,0}=xy$, for $j=1, \dots, 23$ consist of a horizontal and a vertical component; thus the QRT transformation is not defined for $t=0$.
\end{itemize}
\end{proposition}

\begin{proof}
(i)
The proofs for orders of $G(\mathcal S_j)$, for $j=1, \dots, 23$ are contained in \cite{BMM}, see Tables 1--3 therein.

(ii)
\emph{Case $1\le j\le16$.} We present the case $j=1$ here. The cases $j=2,\dots, 23$ are analogous. The curves for $j=1$ are:
$$
S_1(x, y)=x+y+x^{-1}+y^{-1}, \quad xyS_1(x,y)=x^2y+xy^2+x+y,
$$
and
$$
\mathcal K_{1,t}(x, y)=xy-t(x^2y+xy^2+x+y),
$$
The corresponding matrices from Theorem \ref{th:cayley} (b(i)) are
$$
	M_{S_1}=\begin{pmatrix}
		0 & 1 & 0\\
		1 & 0 & 1\\
		0 & 1 & 0
	\end{pmatrix},
    \quad M_{\mathcal K_1}=\begin{pmatrix}
		0 & -t & 0\\
		-t& 1 & -t\\
		0 & -t & 0
	\end{pmatrix}.
	$$

Obviously, $\det(M_{S_1})=0$ and $\det(M_{\mathcal K_1})=0$.
Thus, the orders of the groups $W(\mathcal S_1)$ and $\mathcal H(\mathcal S_1, t)$ are equal to four.
The same calculation applies to $j=2, \dots, 16.$

\emph{Case $17\le j\le21$.}
We have:
\begin{align*}
S_{17}(x, y)&=y + x^{-1} + xy^{-1},
\\
xyS_{17}(x, y)&=xy^2 + y + x^2,
\\
\mathcal{K}_{17,t}(x,y)&=xy-t(xy^2 + y + x^2).
\end{align*}
The corresponding matrices are:
$$
M_{S_{17}}
=
\begin{pmatrix}
0 & 0 & 1\\
1 & 0 & 0\\
0 & 1 & 0
\end{pmatrix},
\quad
M_{\mathcal{K}_{17}}
=
\begin{pmatrix}
	0 & 0 & -t\\
	-t & 1 & 0\\
	0 & -t & 0
\end{pmatrix},
$$
and
$$
\Delta_{S_{17}}=
\left(
\begin{array}{cccc}
	0 & 1 & 0 & 0 \\
	0 & 0 & 1 & 0 \\
	1 & 0 & 0 & 1 \\
	0 & 1 & 0 & 0 \\
\end{array}
\right),
\quad
\Delta_{\mathcal{K}_{17}}
=
\left(
\begin{array}{cccc}
	0 & t^2 & 0 & 0 \\
	0 & 0 & t^2 & 0 \\
	t^2 & t & 0 & t^2 \\
	0 & t^2 & 0 & 0 \\
\end{array}
\right).
$$
It is easy to see that
$\det(\Delta_{S_{17}})=\det(\Delta_{\mathcal{K}_{17}})=0$.

Then:
\begin{align*}
S_{18}(x, y)&=x+ y + x^{-1} +y^{-1}+ xy^{-1}+ x^{-1}y,
\\
xyS_{18}(x,y)&=x^2y+ xy^2 + y +x+ x^2+ y^2,
\\
\mathcal{K}_{18,t}(x,y)&=xy-t(x^2y+ xy^2 + y +x+ x^2+ y^2),
\end{align*}
with the matrices:
\begin{gather*}
M_{S_{18}}
=
\begin{pmatrix}
	0 & 1 & 1\\
	1 & 0 & 1\\
	1 & 1 & 0
\end{pmatrix},
\quad
M_{\mathcal{K}_{18}}
=
\begin{pmatrix}
	0 & -t & -t\\
	-t & 1 & -t\\
	-t & -t & 0
\end{pmatrix},
\\
\Delta_{S_{18}}=
\left(
\begin{array}{cccc}
	-1 & 1 & 1 & -1 \\
	1 & -1 & 1 & 1 \\
	1 & 1 & -1 & 1 \\
	-1 & 1 & 1 & -1 \\
\end{array}
\right),
\quad
\Delta_{\mathcal{K}_{18}}
=
\left(
\begin{array}{cccc}
	-t^2 & t^2 & t^2 & -t^2 \\
	t^2 & -t^2 & t^2+t & t^2 \\
	t^2 & t^2+t & -t^2 & t^2 \\
	-t^2 & t^2 & t^2 & -t^2 \\
\end{array}
\right).
\end{gather*}
Again, the determinants of the last two matrices are zero.

Next:
\begin{align*}
S_{19}(x, y)&=y^{-1} + x^{-1} + xy,\\
xyS_{19}(x, y)&=x + y + x^2y^2,\\
\mathcal{K}_{19,t}(x,y)&=xy-t(x + y + x^2y^2).
\end{align*}
The matrices:
\begin{gather*}
	M_{S_{19}}
	=
	\begin{pmatrix}
		1 & 0 & 0\\
		0 & 0 & 1\\
		0 & 1 & 0
	\end{pmatrix},
	\quad
	M_{\mathcal{K}_{19}}
	=
	\begin{pmatrix}
		-t & 0 & 0\\
		0 & 1 & -t\\
		0 & -t & 0
	\end{pmatrix},
	\\
	\Delta_{S_{19}}=
	\left(
	\begin{array}{cccc}
		-1 & 0 & 0 & 0 \\
		0 & 0 & 0 & -1 \\
		0 & 0 & 0 & -1 \\
		0 & -1 & -1 & 0 \\
	\end{array}
	\right),
	\quad
	\Delta_{\mathcal{K}_{19}}
	=
	\left(
	\begin{array}{cccc}
		-t^2 & 0 & 0 & 0 \\
		0 & 0 & 0 & -t^2 \\
		0 & 0 & 0 & -t^2 \\
		0 & -t^2 & -t^2 & -t \\
	\end{array}
	\right).
\end{gather*}
The last two matrices have zero determinants.

Then:
\begin{align*}
S_{20}(x, y)&=y + x + x^{-1}y^{-1},\\
xyS_{20}(x, y)&=xy^2 + x^2y + 1,\\
\mathcal{K}_{20,t}(x,y)&=xy-t(xy^2 + x^2y + 1),	
\end{align*}
with the matrices:
\begin{gather*}
	M_{S_{20}}
	=
	\begin{pmatrix}
		0 & 1 & 0\\
		1 & 0 & 0\\
		0 & 0 & 1
	\end{pmatrix},
	\quad
	M_{\mathcal{K}_{20}}
	=
	\begin{pmatrix}
		0 & -t & 0\\
		-t & 1 & 0\\
		0 & 0 & -t
	\end{pmatrix},
	\\
	\Delta_{S_{20}}=
	\left(
	\begin{array}{cccc}
		0 & -1 & -1 & 0 \\
		-1 & 0 & 0 & 0 \\
		-1 & 0 & 0 & 0 \\
		0 & 0 & 0 & -1 \\
	\end{array}
	\right),
	\quad
	\Delta_{\mathcal{K}_{20}}
	=
	\left(
	\begin{array}{cccc}
		-t & -t^2 & -t^2 & 0 \\
		-t^2 & 0 & 0 & 0 \\
		-t^2 & 0 & 0 & 0 \\
		0 & 0 & 0 & -t^2 \\
	\end{array}
	\right).
\end{gather*}
The determinants of the last two matrices are obviously zero.

Finally, we have:
\begin{align*}
S_{21}(x, y)&=x+ y + x^{-1} +y^{-1}+ xy+ x^{-1}y^{-1},
\\
xyS_{21}(x, y)&=x^2y+ xy^2 + y +x+ x^2y^2+ 1,
\\
\mathcal{K}_{21,t}(x,y)&=xy-t(x^2y+ xy^2 + y +x+ x^2y^2+ 1),
\end{align*}
with the matrices:
\begin{gather*}
	M_{S_{21}}
	=
	\begin{pmatrix}
		1 & 1 & 0\\
		1 & 0 & 1\\
		0 & 1 & 1
	\end{pmatrix},
	\quad
	M_{\mathcal{K}_{21}}
	=
	\begin{pmatrix}
		-t & -t & 0\\
		-t & 1 & -t\\
		0 & -t & -t
	\end{pmatrix},
	\\
	\Delta_{S_{21}}=
	\left(
	\begin{array}{cccc}
		-1 & -1 & -1 & 1 \\
		-1 & 1 & 1 & -1 \\
		-1 & 1 & 1 & -1 \\
		1 & -1 & -1 & -1 \\
	\end{array}
	\right),
	\quad
	\Delta_{\mathcal{K}_{21}}
	=
	\left(
	\begin{array}{cccc}
		-t^2-t & -t^2 & -t^2 & t^2 \\
		-t^2 & t^2 & t^2 & -t^2 \\
		-t^2 & t^2 & t^2 & -t^2 \\
		t^2 & -t^2 & -t^2 & -t^2-t \\
	\end{array}
	\right).
\end{gather*}
Again, the determinants of the last two matrices are equal to zero, so we can conclude that all groups $W(\mathcal S_j)$ and $\mathcal H(\mathcal S_j, t)$ for $j=17, \dots, 21$ are of order six.

\emph{Case $j\in\{22,23\}$.}
For $j=22$, we have:
\begin{align*}
	S_{22}(x, y)&=x + x^{-1}+  xy^{-1}+  x^{-1}y;\\
	xyS_{22}(x,y)&=x^2y+y+x^2+y^2=(x^2+y)(y+1);\\
	\mathcal{K}_{22,t}(x,y)&=xy-t(x^2y+y+x^2+y^2),
\end{align*}
so the corresponding matrices are:
\begin{gather*}
M_{S_{22}}=
\begin{pmatrix}
	0&1&1\\
	0&0&0\\
	1&1&0
\end{pmatrix},
\quad
M_{\mathcal{K}_{22}}=
\begin{pmatrix}
	0&-t&-t\\
	0&1&0\\
	-t&-t&0
\end{pmatrix},
\\
\Delta_{S_{22}}=
\left(
\begin{array}{cccc}
	0 & 1 & 0 & -1 \\
	0 & -1 & 0 & 1 \\
	1 & 0 & -1 & 0 \\
	-1 & 0 & 1 & 0 \\
\end{array}
\right),
\quad
\Delta_{\mathcal{K}_{22}}=
\left(
\begin{array}{cccc}
	0 & t^2 & 0 & -t^2 \\
	0 & -t^2 & t & t^2 \\
	t^2 & t & -t^2 & 0 \\
	-t^2 & 0 & t^2 & 0 \\
\end{array}
\right).
\end{gather*}
We have: $\det(M_{S_{22}})=\det(\Delta_{S_{22}})=0$,
$\det(M_{\mathcal{K}_{22}})=-t^2$, $\det(\Delta_{\mathcal{K}_{22}})=t^6$.

We note that all curves $\mathcal{K}_{22,t}=0$ are smooth, except for $t\in\{1/4,-1/4,0\}$.
For $t=0$, the curve is $xy=0$.
For $t=1/4$, the curve has double point at $(x,y)=(1,1)$, while for $t=-1/4$, it has double point at $(x,y)=(-1,1)$.

The Eisenstein invariants for $\mathcal{K}_{22,t}$ are:
$$
D=\frac{1}{12} \left(16 t^4-16 t^2+1\right),
\quad
E=\frac{1}{216} \left(64 t^6+120 t^4-24 t^2+1\right).
$$
The condition of Theorem \ref{th:cayley} for the translation of order $4$ is satisfied for all $t\notin\{-1/4, 0, 1/4\}$, thus, the group $\mathcal H(\mathcal S_{22}, t)$ is of order $8$ whenever the curve is smooth.

We check that the order of $\mathcal H(\mathcal S_{22}, t)$ for $t=\pm1/4$ is $8$,  using Theorem \ref{th:doublepoint}.

For $j=23$, the consideration repeats the one for $j=22$. We have:
\begin{align*}
S_{23}(x, y)&=x+x^{-1} + xy+ x^{-1}y^{-1},
\\
xyS_{23}(x, y)&=x^2y+y + x^2y^2+ 1=(x^2y+1)(y+1),
\\
\mathcal{K}_{23,t}(x,y)&=xy-t(x^2y+y + x^2y^2+ 1).
\end{align*}
The matrices are:
\begin{gather*}
	M_{S_{23}}=
	\begin{pmatrix}
		1&1&0\\
		0&0&0\\
		0&1&1
	\end{pmatrix},
	\quad
	M_{\mathcal{K}_{23}}=
	\begin{pmatrix}
		-t&-t&0\\
		0&1&0\\
		0&-t&-t
	\end{pmatrix},
	\\
	\Delta_{S_{23}}=
	\left(
	\begin{array}{cccc}
		0 & -1 & 0 & 1 \\
		0 & 1 & 0 & -1 \\
		-1 & 0 & 1 & 0 \\
		1 & 0 & -1 & 0 \\
	\end{array}
	\right),
	\quad
	\Delta_{\mathcal{K}_{23}}=
	\left(
	\begin{array}{cccc}
		-t & -t^2 & 0 & t^2 \\
		0 & t^2 & 0 & -t^2 \\
		-t^2 & 0 & t^2 & 0 \\
		t^2 & 0 & -t^2 & -t \\
	\end{array}
	\right).
\end{gather*}
The condition of Theorem of \ref{th:cayley} for a translation of order $4$ is always satisfied for $t\notin\{-1/4, 0, 1/4\}$ and for $t=\pm 1/4$ Theorem \ref{th:doublepoint} gives that that the translation is of order $4$ as well.

Thus, the orders of the groups $W(\mathcal S_j)$ and $\mathcal H(\mathcal S_j, t)$, for $t\ne 0$, for $j=22, 23$ are equal to eight.
\end{proof}

\section{Planar four-bar links}
\subsection{Four-bar links and their planar configurations}
In this section, following Darboux \cite{Dar},  we consider $4$-bar links and their configurations in the Euclidean plane, see also \cite{GN} and \cite{DuistermaatBOOK}.

\begin{figure}[h]
	\centering
	\includegraphics[width=5cm, height=5 cm]{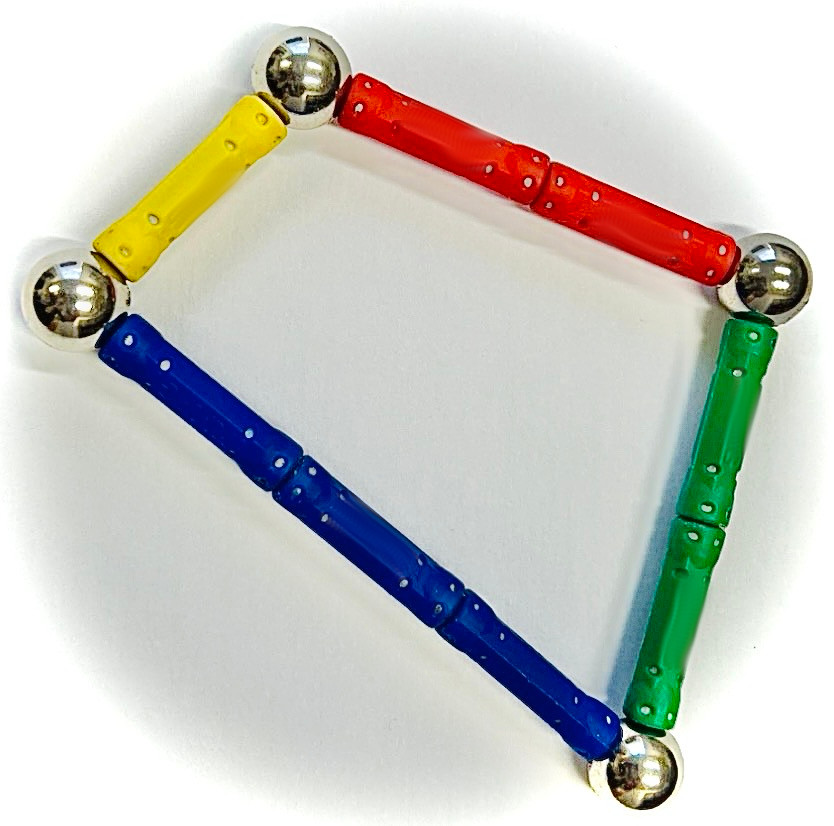}
	\caption{A planar configuration for the four-bar link $(3, 2, 2, 1)$.}\label{fig:4barlink}
\end{figure}

\begin{definition}\label{def:4bar}
A \emph{$4$-bar link} is a $4$-string of positive numbers $(a, b, c, d)$.
A \emph{planar configuration} of a $4$-bar link $(a,b,c,d)$ is a  closed planar polygonal line $V_1V_2V_3V_4$ whose edges have lengths $a=|V_1V_2|$, $b=|V_2V_3|$, $c=|V_3V_4|$, and $d=|V_4V_1|$, see Figure \ref{fig:4barlink}.
\end{definition}

\begin{remark}
A necessary and sufficient condition for the existence of a planar configuration of a $4$-bar link $(a,b,c,d)$ is the ``triangle inequality'':
$$
\max\{a, b, c, d\}<\frac{1}{2}(a+b+c+d).
$$
\end{remark}

Following Darboux \cite{Dar}, we consider planar configurations of  a $4$-bar link $(a,b,c,d)$ up to orientation-preserving isometries of the Euclidean plane (see also \cite{Far}).
Thus, choosing the appropriate coordinate system, we can assume that $V_1=(0,0)$ and $V_2=(a, 0)$. Then $V_3$ lies on the circle $C(V_2, b)$, centered at $V_2$ with radius $b$ and similarly, $V_4$
lies on the circle $C(V_1, d)$, centered at the origin with radius $d$. The pair $(V_3, V_4)\in C(V_2, b)\times C(V_1, d)$ satisfies an additional distance relation: $c=|V_3V_4|$.
We will use that to parametrize all planar configurations of a given $4$-bar link.

First, denote by $\varphi$, $\psi$ the angles between the sides $V_2V_3$, $V_1V_4$ with the line $V_1V_2$, as shown in Figure \ref{fig:4bar-angles}
\begin{figure}[h]
	\centering
	\begin{tikzpicture}[scale=1.3]
		\coordinate (V1) at (0,0);
		\coordinate (V2) at (2,0);
		\coordinate (V2b) at (3,0);
		\coordinate (V3) at (2.5,0.866025);
		\coordinate (V4) at (1.17962, 2.36823);

		\draw[very thick] (V1)--(V2)node[midway, below]{$a$}--(V3)node[midway, left]{$b$}--(V4)node[midway, right]{$c$}--cyclenode[midway, left]{$d$};
		\draw[thick,gray!30](V2)--(V2b);
		
		\pic [draw, ->, thick, "$\psi$", angle eccentricity=1.5] {angle = V2--V1--V4};
		\pic [draw, ->, thick, "$\varphi$", angle eccentricity=1.5] {angle = V2b--V2--V3};
		
		\draw[fill=black] (V1) circle [radius=0.05] node [below left]{$V_1$};
		\draw[fill=black] (V2) circle [radius=0.05] node [below right]{$V_2$};
		\draw[fill=black] (V3) circle [radius=0.05] node [above right]{$V_3$};
		\draw[fill=black] (V4) circle [radius=0.05] node [above]{$V_4$};

	\end{tikzpicture}
	\caption{Parametrization of a configuration $V_1V_2V_3V_4$ by the angles $\varphi$, $\psi$.}\label{fig:4bar-angles}
\end{figure}
Then $V_3=(a+b\cos \varphi, b\sin\varphi)$, $V_4=(d\cos \psi, d\sin\psi)$.
Denoting $x=\tan(\varphi/2)$, $y=-\tan(\psi/2)$, we have:
\begin{equation}\label{eq:4bar-param}
\cos \psi =\frac{x^2-1}{x^2+1},\quad \sin \psi =\frac{2x}{x^2+1},%\label{eq:4link1}
	\quad
	\cos \varphi =-\frac{y^2-1}{y^2+1}, \quad \sin \varphi  =-\frac{2y}{y^2+1}.%\label{eq:4link2}
\end{equation}
The distance relation $c=|V_3V_4|$ then gives the following $(2,2)$ correspondence:
\begin{equation}\label{eq:4bar22}
%	\begin{aligned}
		L:\ ((a+b+d)^2-c^2)x^2y^2+((a+b-d)^2-c^2)x^2 +((a-b+d)^2-c^2)y^2%\\
		+8bdxy +(a-b-d)^2-c^2=0.
%	\end{aligned}
\end{equation}

\begin{remark}
Unless $b=d$, the $(2,2)$ correspondence \eqref{eq:4bar22} is non-symmetric.
\end{remark}

\begin{remark}\label{rem:opor}
The correspondence $L$ is centrally symmetric with respect to the origin, i.e.~if $(x,y)\in L$ then $(-x,-y)\in L$.
We note that the two configurations corresponding to the points $(x,y)$ and $(-x,-y)$ are symmetric to each other with respect to the line $V_1V_2$.
Thus, they are congruent, but of the opposite orientations. This induces a natural involution among the configurations of $4$-bar links, which plays an important role in the general theory, in particular in topological considerations, see \cite{Far}. We will return to it in Section \ref{sec:semiper}.
\end{remark}

\begin{remark}\label{rem:singular}
	The discriminant $F_L$ of the biquadratic $L$ given by \eqref{eq:4bar22} is
	\begin{align*}
		F_L=&\,2^{24}(abcd)^4(a+b+c+d)(a-b+c+d)(a+b-c+d)(a+b+c-d)
		\times\\
		&\times
		(a-b-c+d)(a-b+c-d)(a+b-c-d)(a-b-c-d).
	\end{align*}
We note that $F_L=0$ if and only if
either
one of the quantities $a$, $b$, $c$, $d$ equals zero, or one of those quantities equals the sum of the remaining three, or the sum of two of them equals the sum of the remaining two.
Assuming that $a$, $b$, $c$, $d$ are all positive, we can see that the discriminant vanishes in the limit cases of the triangle inequality or when  the sum of two sides equals the sum of the remaining pair of sides.
\end{remark}

Next, in Section \ref{sec:4bar-darboux} we introduce involutions and so-called \emph{Darboux transformations}, which occur naturally on the variety of the configurations of a $4$-bar link, then  we investigate their periodicity in in Section \ref{sec:4bar-periodic} and other properties in the remaining sections.

\subsection{Darboux transformations and their periodicity}\label{sec:4bar-darboux}

Let $V_1V_2V_3V_4$ be a configuration of a given $4$-bar link.
Suppose that $V_4'$ is the point symmetric to $V_4$ with respect to the line $V_1V_3$.
Then the map $h:V_1V_2V_3V_4\mapsto V_1V_2V_3V_4'$ is an involution on the variety of all configurations of the $4$-bar link.
Similarly, if $V_3'$ is the point symmetric to $V_3$ with respect to $V_2V_4$, the map $v:V_1V_2V_3V_4\mapsto V_1V_2V_3'V_4$ is also an involution.
See Figure \ref{fig:4bar-involutions}.
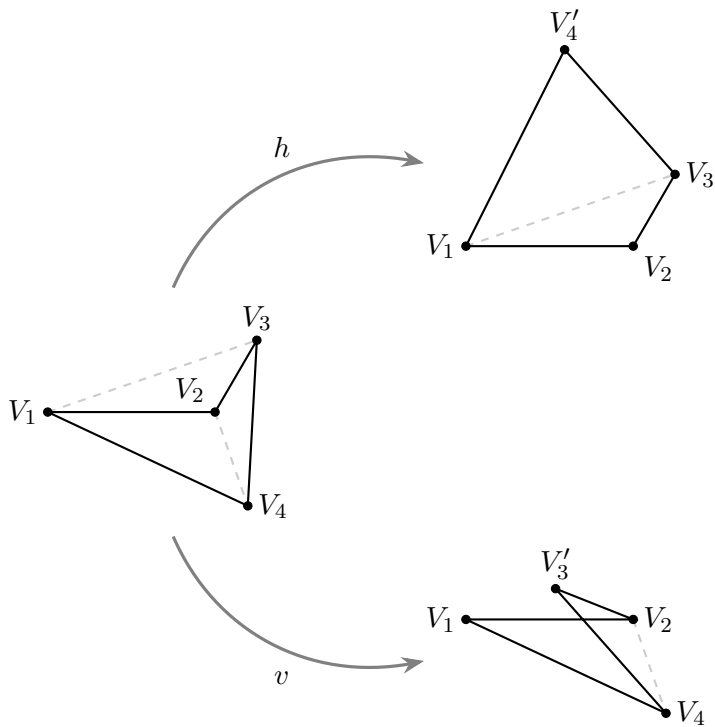
\begin{figure}[h]
\centering
	\begin{tikzpicture}[scale=1.1]

 \draw [-Stealth,very thick, gray] (1.5,1.5) to [bend left=40] (4.5,3);
\node at (2.8,3.2){$h$};

 \draw [-Stealth,very thick, gray] (1.5,-1.5) to [bend left=-40] (4.5,-3);
\node at (2.8,-3.2){$v$};

%pocetni 4bar
\draw[thick,dashed,gray!40](0,0)--(2.5,0.866025);
\draw[thick,dashed,gray!40](2,0)--(2.39181, -1.13105);

\draw[thick](0,0)--(2,0)--(2.5,0.866025)--(2.39181, -1.13105)--cycle;

\draw[fill=black] (0,0) circle [radius=0.05] node [left]{$V_1$};
\draw[fill=black] (2,0) circle [radius=0.05] node [above left]{$V_2$};
\draw[fill=black] (2.5,0.866025) circle [radius=0.05] node [above]{$V_3$};
\draw[fill=black] (2.39181, -1.13105) circle [radius=0.05] node [right]{$V_4$};

%drugi 4bar
\begin{scope}[shift={(5,2)}]
\draw[thick,dashed,gray!40](0,0)--(2.5,0.866025);

\draw[thick](0,0)--(2,0)--(2.5,0.866025)--(1.17962, 2.36823)--cycle;

\draw[fill=black] (0,0) circle [radius=0.05] node [left]{$V_1$};
\draw[fill=black] (2,0) circle [radius=0.05] node [below right]{$V_2$};
\draw[fill=black] (2.5,0.866025) circle [radius=0.05] node [right]{$V_3$};
\draw[fill=black] (1.17962, 2.36823) circle [radius=0.05] node [above]{$V_4'$};
\end{scope}

%treci 4bar
\begin{scope}[shift={(5,-2.5)}]
\draw[thick,dashed,gray!40](2,0)--(2.39181, -1.13105);

\draw[thick](0,0)--(2,0)--(1.07143, 0.371154)--(2.39181, -1.13105)--cycle;

\draw[fill=black] (0,0) circle [radius=0.05] node [left]{$V_1$};
\draw[fill=black] (2,0) circle [radius=0.05] node [right]{$V_2$};
\draw[fill=black] (1.07143, 0.371154) circle [radius=0.05] node [above]{$V_3'$};
\draw[fill=black] (2.39181, -1.13105) circle [radius=0.05] node [right]{$V_4$};
\end{scope}

	\end{tikzpicture}
	\caption{Involutions $h$ and $v$ on the configurations of  $4$-bar link.}\label{fig:4bar-involutions}
\end{figure}

\begin{definition}
The \emph{Darboux} transformation $\delta$ of the $4$-bar link configurations is the composition of the involutions $h$ and $v$: $\delta=v\circ h$.
\end{definition}

\begin{proposition}
The involutions $h$ and $v$ correspond to the horizontal and vertical switches on the biquadratic curve $L$ given by \eqref{eq:4bar22}.
Therefore, the Darboux transformation is an instance of the QRT transformations.
\end{proposition}

The dynamics of a different type of iterations of quadrilaterals was studied in \cite{BH}.

\subsubsection{Periodic Darboux transformations}\label{sec:4bar-periodic}

Following Darboux \cite{Dar}, we say that the Darboux transformation is $n$ periodic, if after $n$ iterations, a quadrilateral maps to a quadrilateral congruent to the initial one and of the same orientation.
Now, we are going to describe  all periodic Darboux transformations,  with the biquadratic $L$ being an elliptic curve. Darboux proved in \cite{Dar}, {\it the poristic property} of periodicity of Darboux transformations: the period of the Darboux transformation is a universal property of a given link, not dependent on the choice of a particular polygonal configuration.

The following proposition describes $2$-periodic $4$-bar links and it goes back to the original paper of Darboux.
\begin{proposition}[\cite{Dar}]\label{prop:2Darboux}
 The four-bar link $(a, b,c,d)$ has a $2$-periodic Darboux transformation if and only if:
\begin{equation}\label{eq:2Darbouxperiodic}
a^2+c^2=b^2+d^2.
\end{equation}
\end{proposition}

\begin{proof}
We will prove this in two ways.
	
\emph{First way, different from the proofs from \cite{Dar} and \cite{Izm}}.
The following matrix corresponds to the biquadratic \eqref{eq:4bar22}:
\begin{equation}\label{eq:ML}
M_L=
\begin{pmatrix}
	(a+b+d)^2-c^2 & 0 & (a+b-d)^2-c^2 \\
	0 & 8 b d & 0 \\
	(a-b+d)^2-c^2 & 0 & (a-b-d)^2-c^2
\end{pmatrix}.
\end{equation}
The statement follows immediately from:
$
\det(M_L)=-64 b^2 d^2 \left(a^2-b^2+c^2-d^2\right).
$

\smallskip

\emph{Second way.}
For another proof, which was also indicated in \cite{Izm}, recall the following known statement from elementary geometry:
For a given quadrilateral, the sum of the squares of one pair of opposite sides equals the sum of the squares of the other pair of opposite sides if and only if its diagonals are orthogonal to each other.

Using that, the statement follows from the fact that two axial symmetries with non-parallel axes commute if and only if their axes are orthogonal to each other.
\end{proof}

\begin{example}
Proposition \ref{prop:2Darboux} is illustrated in Figure \ref{fig:2Darboux}.
The diagonals of each quadrilateral in the sequence are orthogonal to each other.
As an interesting consequence, we have that thus their intersection points remains unchanged by the involutions.
\end{example}
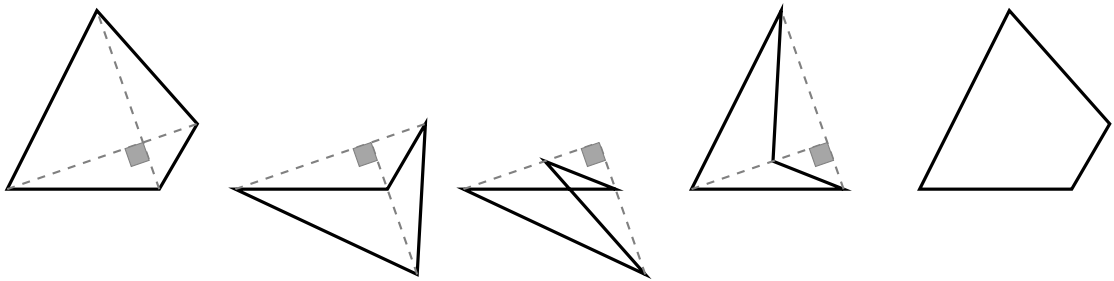
\begin{figure}[h]
	\begin{center}
		\begin{tikzpicture}%[scale=0.8]
\draw[very thick] (0,0) -- (2,0)--(2.5,0.866025)--(1.17962, 2.36823)--cycle;
			
\draw[thick,gray,dashed](0,0) -- (2.5,0.866025);
\draw[thick,gray,dashed](2,0)--(1.17962, 2.36823);

\coordinate (A) at (2.,0) {};
\coordinate (B) at (0, 0) {};
\coordinate (0) at (1.78571, 0.618589) {};

% left angle
\tkzMarkRightAngle[draw=gray,size=.25,fill=gray!70](A,0,B);

\draw[very thick,shift={(3,0)}] (0,0) -- (2,0)--(2.5,0.866025)--(2.39181, -1.13105)--cycle;
			
\draw[thick,gray,dashed,shift={(3,0)}](0,0) -- (2.5,0.866025);
\draw[thick,gray,dashed,,shift={(3,0)}](2.39181, -1.13105)--(1.78571, 0.618589);
			
\coordinate (A1) at (5,0) {};
\coordinate (B1) at (3, 0) {};
\coordinate (01) at (4.78571, 0.618589) {};

\tkzMarkRightAngle[draw=gray,size=.25,fill=gray!70](A1,01,B1);

\draw[very thick,shift={(6,0)}] (0,0) -- (2,0)--(1.07143, 0.371154)--(2.39181, -1.13105)--cycle;
	
\draw[thick,gray,dashed,shift={(6,0)}](0,0) -- (1.78571, 0.618589);
\draw[thick,gray,dashed,,shift={(6,0)}](2.39181, -1.13105)--(1.78571, 0.618589);
			
\coordinate (A2) at (8,0) {};
\coordinate (B2) at (6, 0) {};
\coordinate (02) at (7.78571, 0.618589) {};

\tkzMarkRightAngle[draw=gray,size=.25,fill=gray!70](A2,02,B2);

\draw[very thick,shift={(9,0)}] (0,0) -- (2,0)--(1.07143, 0.371154)--(1.17962, 2.36823)--cycle;

\draw[thick,gray,dashed,shift={(9,0)}](0,0) -- (1.78571, 0.618589);
\draw[thick,gray,dashed,,shift={(9,0)}](2, 0)--(1.17962, 2.36823);

\coordinate (A3) at (11,0) {};
\coordinate (B3) at (9, 0) {};
\coordinate (03) at (10.78571, 0.618589) {};

\tkzMarkRightAngle[draw=gray,size=.25,fill=gray!70](A3,03,B3);

			\draw[very thick,shift={(12,0)}] (0,0) -- (2,0)--(2.5,0.866025)--(1.17962, 2.36823)--cycle;
			
		\end{tikzpicture}
	\end{center}\caption{A $2$-periodic Darboux transformation: $a=2$, $b=1$, $c=2$, $d=\sqrt{7}$. }\label{fig:2Darboux}
\end{figure}

\begin{proposition}[\cite{Izm}]\label{prop:3Darboux}
The Darboux transformation of a four-bar link $(a, b,c,d)$ is $3$-periodic if and only if:
$$
b^2d^2(a^2 -b^2+ c^2 - d^2)^2
=
(a^2 c^2-b^2d^2)^2.
$$
\end{proposition}	

We provide two proofs, different from \cite{Izm}.

\begin{proof}
\emph{First way.}	
The cofactors of $M_L$, given by equation \eqref{eq:ML}, are:
\begin{align*}
\Delta_{11}&=	8 b d \left((a+b+d)^2-c^2\right),
&
\Delta_{12}&=0,
&
\Delta_{13}&=8 b d \left(c^2-(a+b-d)^2\right),
\\
\Delta_{21}&=0,
&
\Delta_{22}&=8bd(b^2-a^2-  c^2 + d^2),
&
\Delta_{23}&=0, \\
\Delta_{31}&=8 b d \left(c^2-(a-b+d)^2\right),
&
\Delta_{32}&=0, &
\Delta_{33}&=8 b d \left((a-b-d)^2-c^2\right).
\end{align*}
Thus
$$
\Delta_L=
8bd
\begin{pmatrix}
(a+b+d)^2-c^2 & 0 & 0 & -a^2+ b^2-  c^2 + d^2
\\
0 & -a^2+ b^2-  c^2 + d^2 & c^2-(a+b-d)^2 & 0
\\
0 & c^2-(a-b+d)^2 & -a^2+ b^2-  c^2 + d^2 & 0
\\
-a^2+ b^2-  c^2 + d^2 & 0 & 0 & (a-b-d)^2-c^2
\end{pmatrix},
$$
so
\begin{align*}
\frac{\det(\Delta_L)}{(8bd)^4}
&=
16 \left(a^2 b d+a^2 c^2-b^3 d-b^2 d^2+b c^2 d-b d^3\right)
\left(a^2 b d-a^2 c^2-b^3 d+b^2 d^2+b c^2 d-b d^3\right)
\\
&=
16
\left(
b^2d^2(a^2 -b^2+ c^2 - d^2)^2
-
(a^2 c^2-b^2d^2)^2
\right)
,
\end{align*}
which immediately implies the statement.

\emph{Second way.}
The Eisenstein invariants of the bi\-qua\-dra\-tic \eqref{eq:4bar22} are:
\begin{align*}
	D_L=&
	\frac{16}{3}
	\Bigg(\left(a^4-2 a^2 \left(b^2+c^2+d^2\right)+b^4-2 b^2 \left(c^2+d^2\right)+\left(c^2-d^2\right)^2\right)^2
	\\&\qquad
	+3 \left((a+b-d)^2-c^2\right) \left((a-b+d)^2-c^2\right) \left((a-b-d)^2-c^2\right) \left((a+b+d)^2-c^2\right)\Bigg),
	\\
	E_L=&\frac{64}{27} \left(a^4-2 a^2 \left(b^2+c^2+d^2\right)+b^4-2 b^2 \left(c^2+d^2\right)+\left(c^2-d^2\right)^2\right)
	\times
	\\&\times
	\Bigg(
	9 \left((a+b-d)^2-c^2\right) \left((a-b+d)^2-c^2\right) \left((-a+b+d)^2-c^2\right) \left((a+b+d)^2-c^2\right)
	\\&\quad\quad
	-\left(a^4-2 a^2 \left(b^2+c^2+d^2\right)+b^4-2 b^2 \left(c^2+d^2\right)+\left(c^2-d^2\right)^2\right)^2
	\Bigg),
\end{align*}
while the value of the coordinate $X_0$ from \eqref{eq:XY} is:
\begin{equation}\label{eq:X0semi}
X_0=\frac{4}{3} \left(a^4-2 a^2 \left(b^2+c^2+d^2\right)+b^4-2 b^2 \left(c^2-5 d^2\right)+\left(c^2-d^2\right)^2\right).
\end{equation}
Then, the condition for $3$-periodicity is equivalent to $C_2=0$, where
$$
\sqrt{4 x^3 - D_L x + E_L}=C_0+C_1(x-X_0)+C_2(x-X_0)^2+C_3(x-X_0)^3+\dots.
$$
The direct calculation gives:
$$
C_2=
\frac{b^2 d^2 \left(a^2-b^2+c^2-d^2\right)^2-(a^2 c^2-b^2 d^2)^2}{2 b^3 d^3 \left(a^2-b^2+c^2-d^2\right)^3},
$$
which completes the second proof of Proposition \ref{prop:3Darboux}.
\end{proof}

\begin{example}
A $3$-periodic Darboux transformation is shown in Figure \ref{fig:3Darboux}.
\end{example}
	\begin{figure}[h]
		\begin{center}
			\begin{tikzpicture}[scale=0.3]
				\draw[thick] (0,0) -- (4.5,0)--(5.5, 1.73205)--(3.83904, 5.3709)--cycle;
				
				\draw[thick,shift={(6,0)}] (0,0) -- (4.5,0)--(5.5, 1.73205)--(6.22385, -2.20191)--cycle;
				
				\draw[thick,shift={(13,0)}] (0,0) -- (4.5,0)--(2.57858, -0.555101)--(6.22385, -2.20191)--cycle;
				
				\draw[thick,shift={(19,0)}] (0,0) -- (4.5,0)--(2.57858, -0.555101)--(6.57858, -0.554124)--cycle;

				\draw[thick,shift={(25,0)}] (0,0) -- (4.5,0)--(3.10989, 1.43792)--(6.57858, -0.554124)--cycle;

				\draw[thick,shift={(32,0)}] (0,0) -- (4.5,0)--(3.10989, 1.43792)--(3.83904, 5.3709)--cycle;

				\draw[thick,shift={(38,0)}] (0,0) -- (4.5,0)--(5.5, 1.73205)--(3.83904, 5.3709)--cycle;

			\end{tikzpicture}
		\end{center}\caption{A $3$-periodic Darboux transformation. }\label{fig:3Darboux}
	\end{figure}
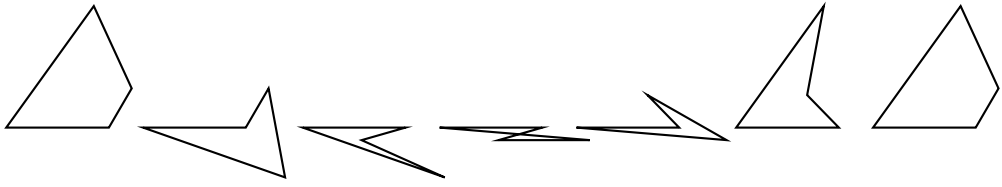
	
\begin{proposition}\label{prop:Darboux4}
The necessary and sufficient condition for $4$-periodicity of the  Darboux transformation of a four-bar link $(a, b,c,d)$ is:
$$
ac=bd
\quad\text{or}\quad
K_4=0,
$$
with
\begin{align*}
K_4=&
a^6 c^2+a^4 \left(b^2 \left(d^2-2 c^2\right)-2 c^2 d^2\right)
+b^2 d^2 \left(b^4-2 b^2 c^2+\left(c^2-d^2\right)^2\right)
\\&\
+a^2 \left(b^4 \left(c^2-2 d^2\right)-2 b^2 \left(c^4-4 c^2 d^2+d^4\right)+\left(c^3-c d^2\right)^2\right)
.
\end{align*}
\end{proposition}
\begin{proof}
The condition for $4$-periodicity is equivalent to $C_3=0$, where $C_3$ is as in the proof of Proposition \ref{prop:3Darboux}, i.e:
$$
C_3=
\frac{(a c - b d) (a c + b d)K_4 }{32 b^4 d^4 (a^2 - b^2 + c^2 - d^2)^5}.
$$
Since $ac+bd>0$ for $a$, $b$, $c$, $d$ being the lengths of the sides of the $4$-link, that completes the proof.
\end{proof}	

\begin{remark}\label{rem:4a}
The first condition from Proposition \ref{prop:Darboux4} coincides with the first of the conditions from \cite{Izm}*{Proposition 5.6}.
However, our condition $K_4=0$ is not equivalent to the second condition from there, and moreover, our simulations could not confirm the validity of that second condition of \cite{Izm}*{Proposition 5.6}.
\end{remark}

\begin{example}
The condition $ac=bd$ which gives $4$-periodic Darboux transformation is illustrated in Figure \ref{fig:4Darboux}.
\end{example}
	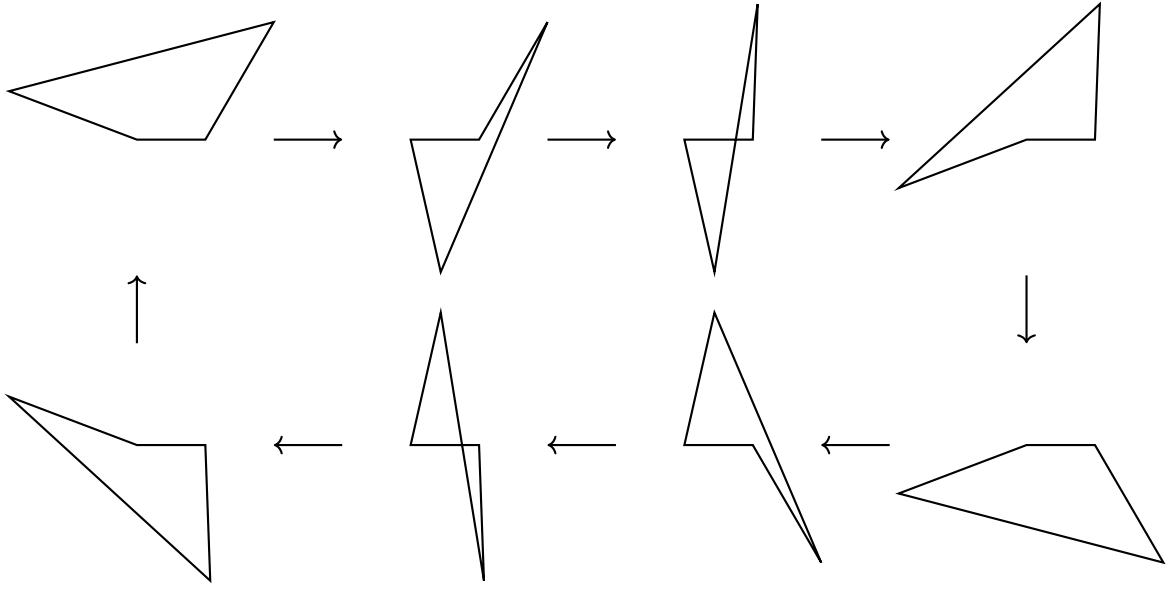
\begin{figure}[h]
	\begin{center}
		\begin{tikzpicture}[scale=0.9]
			\draw[thick] (0,0) -- (1,0)--(2., 1.73205)--(-1.86825, 0.713893)--cycle;
			
			\draw[thick,->] (2,0)--(3,0);
			
			\draw[thick,shift={(4,0)}] (0,0) -- (1,0)--(2., 1.73205)--(0.439678, -1.95107)--cycle;
			
			\draw[thick,->] (6,0)--(7,0);
			
			\draw[thick,shift={(8,0)}] (0,0) -- (1,0)--(1.07143, 1.998722)--(0.439678, -1.95107)--cycle;
			
			\draw[thick,->] (10,0)--(11,0);

			\draw[thick,shift={(13,0)}] (0,0) -- (1,0)--(1.07143, 1.998722)--(-1.86825, -0.713893)--cycle;
			
			\draw[thick,->] (13,-2)--(13,-3);

			\draw[thick,shift={(13,-4.5)}] (0,0) -- (1,0)--(2., -1.73205)--(-1.86825, -0.713893)--cycle;
			
			\draw[thick,<-] (10,-4.5)--(11,-4.5);

			\draw[thick,shift={(8,-4.5)}] (0,0) -- (1,0)--(2., -1.73205)--(0.439678, 1.95107)--cycle;
			
			\draw[thick,<-] (6,-4.5)--(7,-4.5);

			\draw[thick,shift={(4,-4.5)}] (0,0) -- (1,0)--(1.07143, -1.99872)--(0.439678, 1.95107)--cycle;
			
			\draw[thick,<-] (2,-4.5)--(3,-4.5);

			\draw[thick,shift={(0,-4.5)}] (0,0) -- (1,0)--(1.07143, -1.99872)--(-1.86825, 0.713893)--cycle;
			
			\draw[thick,<-] (0,-2)--(0,-3);
			
		\end{tikzpicture}
	\end{center}\caption{A $4$-periodic Darboux transformation, $a=1$, $b=2$, $c=4$, $d=2$. }\label{fig:4Darboux}
\end{figure}

\begin{example}
	The condition $K_4=0$ which also gives $4$-periodic Darboux transformation is illustrated in Figure \ref{fig:4Darboux2}.
\end{example}
\begin{figure}[h]
	\begin{center}
		\begin{tikzpicture}[scale=0.8]
			\draw[thick] (0,0) -- (1,0)--(-0.732051, 1.)--(-3.72263, 0.762416)--cycle;
			
			\draw[thick,->] (2,0)--(3,0);

			\draw[thick,shift={(4,0)}] (0,0) -- (1,0)--(-0.732051, 1.)--(0.398088, 3.77899)--cycle;
			
			\draw[thick,->] (6,0)--(7,0);

			\draw[thick,shift={(8,0)}] (0,0) -- (1,0)--(2.33567, 1.48862)--(0.398088, 3.77899)--cycle;
			
			\draw[thick,->] (8,-1)--(8,-2);

			\draw[thick,shift={(8,-3)}] (0,0) -- (1,0)--(2.33567, 1.48862)--(3.59365, -1.23488)--cycle;

			\draw[thick,->] (8,-4)--(8,-5);

\draw[thick,shift={(8,-6)}] (0,0) -- (1,0)--(0.686443, -1.97527)--(3.59365, -1.23488)--cycle;

\draw[thick,<-] (6,-6)--(7,-6);

\draw[thick,shift={(4,-6)}] (0,0) -- (1,0)--(0.686443, -1.97527)--(-2.05337, -3.19733)--cycle;

			\draw[thick,<-] (2,-6)--(3,-6);

\draw[thick,shift={(0,-6)}] (0,0) -- (1,0)--(-0.958737, -0.404163)--(-2.05337, -3.19733)--cycle;

			\draw[thick,<-] (0,-4)--(0,-5);

\draw[thick,shift={(0,-3)}] (0,0) -- (1,0)--(-0.958737, -0.404163)--(-3.72263, 0.762416)--cycle;

\draw[thick,->] (0,-2)--(0,-1);

		\end{tikzpicture}
	\end{center}\caption{A $4$-periodic Darboux transformation. }\label{fig:4Darboux2}
\end{figure}

\begin{proposition}
	The Darboux transformation of a four-bar link $(a, b,c,d)$  is $5$-periodic if and only if:
\begin{gather*}
b^2 d^2 \left(a^2-b^2+c^2-d^2\right)^2
\left(
a^2 c^2 \left(a^2-b^2+c^2-d^2\right)^2
-
(a^2c^2-b^2d^2)^2
\right)^2
\\=\\
(a^2c^2-b^2d^2)^2
\left(
(a^2c^2-b^2d^2)^2
-
b^2 d^2 \left(a^2-b^2+c^2-d^2\right)^2
\right)^2.
\end{gather*}
\end{proposition}
\begin{proof}
	The condition is equivalent to: $\det\begin{pmatrix}
		C_2&C_3\\C_3&C_4
	\end{pmatrix}=0$,
	where $C_2$, $C_3$, $C_4$ are as in the proof of Proposition \ref{prop:3Darboux}.
We calculate:
$$
\det\begin{pmatrix}
	C_2&C_3\\C_3&C_4
\end{pmatrix}=
\frac{A^2-B^2}{1024 b^8 d^8 \left(a^2-b^2+c^2-d^2\right)^{10}},
$$
with
\begin{align*}
&A=
(a^2c^2-b^2d^2)
\left(
(a^2c^2-b^2d^2)^2
-
b^2 d^2 \left(a^2-b^2+c^2-d^2\right)^2
\right),
\\
&B=
b d \left(a^2-b^2+c^2-d^2\right)
\left(
a^2 c^2 \left(a^2-b^2+c^2-d^2\right)^2
-
(a^2c^2-b^2d^2)^2
\right),
\end{align*}
which gives the statement.
\end{proof}

\begin{example}
	A $5$-periodic Darboux transformation is illustrated in Figure \ref{fig:5Darboux}.
\end{example}
\begin{figure}[h]
	\begin{center}
		\begin{tikzpicture}[scale=0.9]
			\draw[thick] (0,0) -- (1,0)--(2.73205, 1)--(0.28997, 1.53501)--cycle;
			
			\draw[thick,->] (1.5,0)--(2.5,0);
		
			\draw[thick, shift={(3,0)}] (0,0) -- (1,0)--(2.73205, 1)--(1.2124, -0.985108)--cycle;

		\draw[thick,->] (5.5,0)--(6.5,0);

			\draw[thick, shift={(8,0)}] (0,0) -- (1,0)--(-0.99023, 0.197444)--(1.2124, -0.985108)--cycle;

	\draw[thick,->] (9.5,0)--(10.5,0);

			\draw[thick, shift={(12,0)}] (0,0) -- (1,0)--(-0.99023, 0.197444)--(1.49751, 0.444776)--cycle;

	\draw[thick,->] (12,-0.5)--(12,-1.5);

			\draw[thick, shift={(12,-2)}] (0,0) -- (1,0)--(0.974162, -1.99983)--(1.49751, 0.444776)--cycle;

\draw[thick,->] (12,-3.5)--(12,-4.5);

			\draw[thick, shift={(12,-5)}] (0,0) -- (1,0)--(0.974162, -1.99983)--(-1.27334, -0.904963)--cycle;

\draw[thick,->] (11,-5)--(10,-5);

\draw[thick, shift={(8.5,-5)}] (0,0) -- (1,0)--(-0.393151, 1.43497)--(-1.27334, -0.904963)--cycle;

\draw[thick,->] (7,-5)--(6,-5);

\draw[thick, shift={(4,-5)}] (0,0) -- (1,0)--(-0.393151, 1.43497)--(1.55678, -0.12957)--cycle;

\draw[thick,->] (3,-5)--(2,-5);

\draw[thick, shift={(0,-5)}] (0,0) -- (1,0)--(-0.883571, -0.67243)--(1.55678, -0.12957)--cycle;

\draw[thick,<-] (1,-3.5)--(1,-4.5);

\draw[thick, shift={(0,-2.5)}] (0,0) -- (1,0)--(-0.883571, -0.67243)--(0.28997, 1.53501)--cycle;			

\draw[thick,->] (1,-1.5)--(1,-0.5);

		\end{tikzpicture}
	\end{center}\caption{A $5$-periodic Darboux transformation. }\label{fig:5Darboux}
\end{figure}

\begin{proposition}\label{prop:Darboux6}
	The necessary and sufficient condition for $6$-periodicity of the Darboux transformation of a $4$-bar link $(a, b, c, d)$ is:
$$
a^2 c^2 \left(a^2-b^2+c^2-d^2\right)^2=
(a^2 c^2-b^2 d^2)^2
\quad\text{or}\quad
K_6=0,
$$
where
\begin{align*}
K_6=&a^{10}c^2 \left(b^2 d^2+c^4\right)
-a^8 c^2 \left(4 b^4 d^2+b^2 \left(2 c^4-3 c^2 d^2+4 d^4\right)+c^6+2 c^4 d^2\right)
\\&
+a^6 c^2 \left(6 b^6 d^2+b^4 \left(c^4-10 c^2 d^2+11 d^4\right)-2 b^2 \left(c^6-9 c^4 d^2+5 c^2 d^4-3 d^6\right)+c^4 \left(c^2-d^2\right)^2\right)
\\&
+a^4 b^2 d^2 \Big(b^4 \left(11 c^4-10 c^2 d^2+d^4\right)-4 b^6 c^2-2 b^2 \left(5 c^6-c^4 d^2+5 c^2 d^4\right)
\\&\qquad\qquad\quad
+\left(3 c^2-4 d^2\right) \left(c^3-c d^2\right)^2\Big)
\\&
+a^2 b^2 d^2 \Big(b^8 c^2+b^6 \left(3 c^2 d^2-4 c^4-2 d^4\right)+2 b^4 \left(3 c^6-5 c^4 d^2+9 c^2 d^4-d^6\right)
\\&
\qquad\qquad\quad
-b^2 \left(4 c^2-3 d^2\right) \left(c^3-c d^2\right)^2+c^2 \left(c^2-d^2\right)^4\Big)
\\&
+b^6 d^6 \left(b^4-b^2 \left(2 c^2+d^2\right)+\left(c^2-d^2\right)^2\right).
\end{align*}
	
\end{proposition}
\begin{proof}
	The condition is obtained from $\det\begin{pmatrix}
		C_3&C_4\\C_4&C_5
	\end{pmatrix}=0$,
	where $C_3$, $C_4$, $C_5$ are as in the proof of Proposition \ref{prop:3Darboux}.
Namely, we have:
$$
	\det\begin{pmatrix}
		C_3&C_4\\C_4&C_5
	\end{pmatrix}=
	\frac
	{C_2K_6 \left((a^2 c^2-b^2 d^2)^2 - a^2 c^2\left(a^2-b^2+c^2-d^2\right)^2\right)
		 }
	{131072 b^{10} d^{10} \left(a^2-b^2+c^2-d^2\right)^{11}}.
$$
If $C_2=0$ is satisfied, then the transformation is $3$-periodic, thus we get the stated conditions from the last equality.
\end{proof}

The first of the conditions from Proposition \ref{prop:Darboux6} was derived in \cite{Izm}*{Proposition 5.7.}.

\begin{example}
A $6$-periodic Darboux transformation, with the link
	 satisfying the condition $K_6=0$ from Proposition \ref{prop:Darboux6} is illustrated in Figure \ref{fig:6Darboux}.
Another $6$-periodic Darboux transformation, obtained by using the first condition from Proposition \ref{prop:Darboux6} is shown in Figure \ref{fig:6Darbouxb}.
\end{example}
\begin{figure}[h]
	\begin{center}
		\begin{tikzpicture}[scale=0.8]
			\draw[thick] (0,0) -- (1,0)--(2.73205, 1.)--(1.45028, 3.71239)--cycle;
			
			\draw[thick,->] (2,0)--(3,0);
			
			\draw[thick, shift={(3.5,0)}] (0,0) -- (1,0)--(2.73205, 1)--(3.50417, -1.89894)--cycle;
			
			\draw[thick,->,shift={(5,0)}] (2,0)--(3,0);

\draw[thick, shift={(8.5,0)}] (0,0) -- (1,0)--(0.504417, -1.93763)--(3.50417, -1.89894)--cycle;

			\draw[thick,->,shift={(9.5,0)}] (2,0)--(3,0);

\draw[thick, shift={(14,0)}] (0,0) -- (1,0)--(0.504417, -1.93763)--(-2.13342, -3.36655)--cycle;

			\draw[thick,->] (14,-3)--(14,-4);

\draw[thick, shift={(14,-4.5)}] (0,0) -- (1,0)--(-0.897143, -0.633126)--(-2.13342, -3.36655)--cycle;

			\draw[thick,->,shift={(0,-3)}] (14,-3)--(14,-4);

\draw[thick, shift={(14,-8)}] (0,0) -- (1,0)--(-0.897143, -0.633126)--(-3.8868, -0.881998)--cycle;

			\draw[thick,->,shift={(0,-6)}] (14,-3)--(14,-4);

\draw[thick, shift={(14,-11)}] (0,0) -- (1,0)--(-0.998774, -0.0700309)--(-3.8868, -0.881998)--cycle;

			\draw[thick,->,shift={(9.5,0)}] (1,-11)--(0,-11);

\draw[thick, shift={(7,-11)}] (0,0) -- (1,0)--(-0.998774, -0.0700309)--(-3.97185, 0.330973)--cycle;

\draw[thick,<-] (7,-9.5)--(7,-10.5);

\draw[thick, shift={(7,-9)}] (0,0) -- (1,0)--(-0.971854, 0.334353)--(-3.97185, 0.330973)--cycle;

\draw[thick,<-] (7,-7.5)--(7,-8.5);

\draw[thick, shift={(7,-7)}] (0,0) -- (1,0)--(-0.971854, 0.334353)--(-3.33477, 2.18277)--cycle;

\draw[thick,<-] (7,-5.5)--(7,-6.5);

\draw[thick, shift={(7,-5)}] (0,0) -- (1,0)--(-0.442763, 1.38508)--(-3.33477, 2.18277)--cycle;

\draw[thick,<-] (2,-3)--(3,-3);

\draw[thick, shift={(0,-5)}] (0,0) -- (1,0)--(-0.442763, 1.38508)--(1.45028, 3.71239)--cycle;

\draw[thick,->] (0,-2)--(0,-1);

		\end{tikzpicture}
	\end{center}\caption{A $6$-periodic Darboux transformation. }\label{fig:6Darboux}
\end{figure}

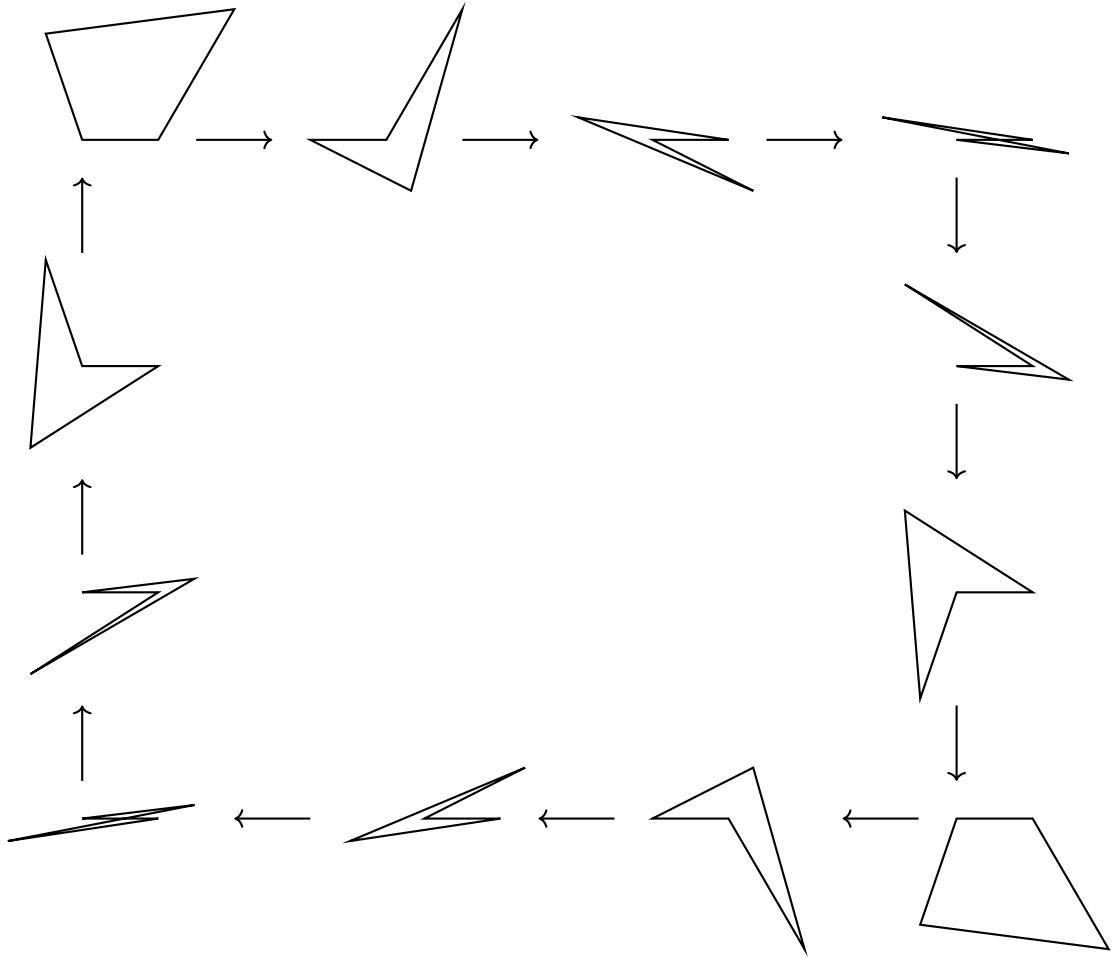
\begin{figure}[h]
	\begin{center}
		\begin{tikzpicture}[scale=1]
			\draw[thick] (0,0) -- (1,0)--(2., 1.73205)--(-0.478897, 1.4079)--cycle;
			
			\draw[thick,->] (1.5,0)--(2.5,0);
			
			\draw[thick,shift={(3,0)}] (0,0) -- (1,0)--(2., 1.73205)--(1.32505, -0.675114)--cycle;
			
				\draw[thick,->,shift={(3,0)}] (2,0)--(3,0);
			
			\draw[thick,shift={(7.5,0)}] (0,0) -- (1,0)--(-0.977613, 0.298404)--(1.32505, -0.675114)--cycle;
			
			\draw[thick,->,shift={(7,0)}] (2,0)--(3,0);
			
			\draw[thick,shift={(11.5,0)}] (0,0) -- (1,0)--(-0.977613, 0.298404)--(1.4762, -0.179932)--cycle;
			
			\draw[thick,->,shift={(11.5,-.5)}] (0,0)--(0,-1);
			
				\draw[thick,shift={(11.5,-3)}] (0,0) -- (1,0)--(-0.680803, 1.08393)--(1.4762, -0.179932)--cycle;
			
			\draw[thick,->,shift={(11.5,-3.5)}] (0,0)--(0,-1);
			
			\draw[thick,shift={(11.5,-6)}] (0,0) -- (1,0)--(-0.680803, 1.08393)--(-0.478897, -1.4079)--cycle;
			
				\draw[thick,->,shift={(11.5,-7.5)}] (0,0)--(0,-1);
			
			\draw[thick,shift={(11.5,-9)}] (0,0) -- (1,0)--(2., -1.73205)--(-0.478897, -1.4079)--cycle;
			
				\draw[thick,->,shift={(10,-9)}] (1,0)--(0,0);
				
					\draw[thick,shift={(7.5,-9)}] (0,0) -- (1,0)--(2., -1.73205)--(1.32505, 0.675114)--cycle;

				\draw[thick,->,shift={(6,-9)}] (1,0)--(0,0);
				
					\draw[thick,shift={(4.5,-9)}] (0,0) -- (1,0)--(-0.977613, -0.298404)--(1.32505, 0.675114)--cycle;
			
			\draw[thick,->,shift={(2,-9)}] (1,0)--(0,0);
			
			\draw[thick,shift={(0,-9)}] (0,0) -- (1,0)--(-0.977613, -0.298404)--(1.4762, 0.179932)--cycle;
			
				\draw[thick,<-,shift={(0,-7.5)}] (0,0)--(0,-1);
				
				\draw[thick,shift={(0,-6)}] (0,0) -- (1,0)--(-0.680803, -1.08393)--(1.4762, 0.179932)--cycle;

\draw[thick,<-,shift={(0,-4.5)}] (0,0)--(0,-1);

\draw[thick,shift={(0,-3)}] (0,0) -- (1,0)--(-0.680803, -1.08393)--(-0.478897, 1.4079)--cycle;

\draw[thick,<-,shift={(0,-.5)}] (0,0)--(0,-1);
			
		\end{tikzpicture}
	\end{center}\caption{A $6$-periodic Darboux transformation, with $a=1$, $b=2$, $c=5/2$, $d=\sqrt{115/13}/2$. }\label{fig:6Darbouxb}
\end{figure}

\subsection{Semi-periodicity for four-bar links}\label{sec:semiper}
We introduce and study here a new, natural kind of periodicity for $4$-bar links, which we are going to call {\it semi-periodicity}.
Let us recall Remark \ref{rem:opor}, where we observed that the correspondence $L$ is centrally symmetric.

\begin{definition}
We  say that the Darboux transformation is \emph{semi-periodic with the semi-period $k$} if its $k$-th iteration maps a quadrilateral $V_1V_2V_3V_4$ to the quadrilateral which is symmetric to $V_1V_2V_3V_4$ with respect to the side $V_1V_2$.

We  also say that a centrally symmetric $(2,2)$ correspondence is \emph{semi-periodic with the semi-period $k$} if the $k$-th iteration of its QRT map is the symmetry with respect to the origin.
\end{definition}

\begin{remark}\label{rem:quasi}
If a Darboux transformation is semi-periodic with the semi-period $k$, then it is periodic with period $n=2k$.
\end{remark}

Now we want to give a characterization of the semi-periodicity of $4$-bar links.
To achieve that, we will look into a more general question of semi-periodicity of the centrally symmetric $(2,2)$ correspondences.
The general form of such correspondences is:
\begin{equation}\label{eq:centrallysymmetric}
\mathcal C\ :\ Q(x, y)=a_{22}x^2y^2+a_{11}xy + a_{20}x^2+a_{02}y^2+a_{00}=0.
\end{equation}

To  a centrally-symmetric $(2,2)$ correspondence $\mathcal C$ \eqref{eq:centrallysymmetric} we assign another  $(2,2)$ correspondence $\mathcal {\hat C}$ in the following way.
Rewrite \eqref{eq:centrallysymmetric} as:
$$
a_{22}x^2y^2 + a_{20}x^2+a_{02}y^2+a_{00}=-a_{11}xy,
$$
then square both sides of the equation, and substitute $u:=x^2$ and $v:=y^2$, which gives:
\begin{equation}\label{eq:secondary}
\begin{aligned}
	\mathcal {\hat C}\ :\ \hat Q(u, v)=&\,a^2_{22}u^2v^2 +2a_{22}a_{20}u^2v + 2a_{22}a_{02}uv^2 + (2a_{22}a_{00}+2a_{02}a_{20}-a_{11}^2)uv
	\\&
	+a^2_{20}u^2 +a^2_{02}v^2
	+2a_{20}a_{00}u + 2a_{02}a_{00}v +a^2_{00}=0.
\end{aligned}	|
\end{equation}

\begin{definition}
The $(2,2)$ correspondence $\mathcal {\hat C}$ \eqref{eq:secondary} is called \emph{the secondary $(2,2)$ correspondence} of a centrally-symmetric $(2,2)$ correspondence $\mathcal C$ \eqref{eq:centrallysymmetric}.
The corresponding cubic \eqref{eq:cubic}
$$
\hat \Gamma: \mu^2=4\lambda^3 - \hat g_2\lambda - \hat g_3,
$$
where
$$
\hat g_2=D_{\mathcal {\hat C}}, \quad g_3=-E_{\mathcal {\hat C}},
$$
is called \emph{the secondary cubic} of the  centrally-symmetric $(2,2)$ correspondence $\mathcal C$ \eqref{eq:centrallysymmetric}.
\end{definition}

\begin{theorem}\label{th:ksemiperiodic}
Let $\mathcal C$ be a centrally-symmetric $(2,2)$ correspondence  given by \eqref{eq:centrallysymmetric}, with $a_{11}\neq0$.
Then $\mathcal C$ is $k$-semi-periodic if and only if it is not $k$-periodic and its secondary $(2,2)$ correspondence
$\hat{\mathcal{C}}$ \eqref{eq:secondary} is $k$-periodic.
\end{theorem}

\begin{proof}
First, notice that, if the QRT-transformation on $\mathcal{C}$ maps $(x,y)$ to $(x_1,y_1)$, then the QRT-transformation on $\hat{\mathcal{C}}$ maps $(x^2,y^2)$ to $(x_1^2,y_1^2)$.

Suppose now the $\mathcal{C}$ is $k$-semi-periodic.
Since $\mathcal{C}$ contains more than one point, it will not be $k$-periodic.
We have that $k$-th iterate of its QRT-transformation maps $(x,y)$ to $(-x,-y)$.
Then the $k$-th iterate of the QRT-transformation of $\hat{\mathcal{C}}$ maps $(x^2,y^2)$ to itself, so we have $k$-periodicity.

Now, suppose that $\mathcal {\hat C}$ \eqref{eq:secondary} is $k$-periodic, and $\mathcal{C}$ is not.
Let $(x_k,y_k)$ be the image of the point $(x,y)$ by the $k$-th iterate of the QRT-transformation of $\mathcal{C}$.
 Due to the $k$-periodicity of $\hat{\mathcal{C}}$, we will have $x^2=x_k^2$ and $y^2=y_k^2$.
If $a_{11}\neq0$, that will imply $(x,y)=(x_k,y_k)$ or $(x,y)=(-x_k,-y_k)$.
The first equality cannot hold since $\mathcal{C}$ is not $k$-periodic, thus the second one is true, implying $k$-semi-periodicity.
Let us observe that for every QRT trajectory $(u_k, v_k)$ of $\hat{\mathcal{C}}$, there exists a QRT trajectory $(x_k, y_k)$ of $\mathcal{C}$, such that $u_k=x_k^2$ and $v_k=y_k^2$.
This follows from the first observation, the fact that both correspondences are $(2,2)$ , and $a_{11}\ne 0$.
\end{proof}

\begin{lemma}\label{lem:secondlinkage}
The secondary $(2,2)$ correspondence of \eqref{eq:4bar22} is:
$$
\begin{aligned}
\hat L\ :\
&(c^2 - (a - b - d)^2)^2 +
2 (c^2 - (a + b - d)^2) (c^2 - (a - b - d)^2) u
\\
& + (c^2 - (a + b -
d)^2)^2 u^2 +
2 (c^2 - (a - b + d)^2) (c^2 - (a - b - d)^2) v
\\
&+
4 (a^4 + b^4 + (c^2 - d^2)^2 - 2 a^2 (b^2 + c^2 + d^2) -
2 b^2 (c^2 + 5 d^2)) u v
\\& +
2 (c^2 - (a + b - d)^2) (c^2 - (a + b + d)^2) u^2 v + (c^2 - (a -
b + d)^2)^2 v^2
\\&+
2 (c^2 - (a - b + d)^2) (c^2 - (a + b + d)^2) u v^2 + (c^2 - (a +
b + d)^2)^2 u^2 v^2=0.
\end{aligned}
$$
The corresponding cubic is:
\begin{equation}\label{eq:seccubic}
\hat \Gamma: \mu^2=4\lambda^3 - \hat g_2\lambda - \hat g_3,
\end{equation}
where
$$
\begin{aligned}
\hat g_2=\ &\frac{2^{16}}{3} b^4 d^4 \left(a^4-2 a^2 \left(b^2+c^2+d^2\right)+b^4-2 b^2 \left(c^2+d^2\right)+\left(c^2-d^2\right)^2\right)^2
\\
&
-2^{14}
b^4 d^4
(a-b+c+d) (a+b-c-d) (a-b-c+d) (a+b+c-d)\times
\\&\qquad\quad\times
 (a-b+c-d) (a+b-c+d) (a-b-c-d) (a+b+c+d)
,
\\
\hat{g}_3=\ &
-\frac{2^{21}}{27} b^6 d^6
\left(
a^4-2 a^2 b^2-2 a^2 c^2-2 a^2 d^2+b^4-2 b^2 c^2-2 b^2 d^2+c^4-2 c^2 d^2+d^4
\right)\times
\\&\qquad\quad\times
\left(a^4-2 a^2 b^2-2 a^2 c^2-2 a^2 d^2-24 a b c d+b^4-2 b^2 c^2-2 b^2 d^2+c^4-2 c^2 d^2+d^4\right)
\times
\\&\qquad\quad\times
 \left(a^4-2 a^2 b^2-2 a^2 c^2-2 a^2 d^2+24 a b c d+b^4-2 b^2 c^2-2 b^2 d^2+c^4-2 c^2 d^2+d^4\right).
\end{aligned}
$$
\end{lemma}

\begin{lemma}
The secondary biquadratic $\hat L$ defines a smooth elliptic curve if and only if the biquadratic $L$ does and $a_{11}\neq0$.
\end{lemma}
\begin{proof}
The discriminant of $L$ is:
$$
256\, a_{00} a_{02} a_{20} a_{22}
\left(
a_{11}^4-8 a_{11}^2 (a_{00} a_{22}+a_{02} a_{20})+16 (a_{02} a_{20}-a_{00} a_{22})^2
\right)^2,
$$
while the discriminant of $\hat{L}$ is:
$$
256\, a_{11}^{10}  (a_{00} a_{02} a_{20} a_{22})^2
\left(
a_{11}^4-8 a_{11}^2 (a_{00} a_{22}+a_{02} a_{20})+16 (a_{02} a_{20}-a_{00} a_{22})^2
\right),
$$
which immediately implies the statement.
\end{proof}

\begin{proposition}[\cite{Izm}*{Theorem 4}]	\label{prop:2semi-periodic}
The Darboux transformation of the $4$-bar link $(a, b, c, d)$ is $2$-semi-periodic if and only if
$
ac=bd.
$
\end{proposition}

We provide two ways to prove Proposition \ref{prop:2semi-periodic}, which are both different than the one used in \cite{Izm}.
The first way is in the proof below, and the second one in Example \ref{ex:second-proof}.

\begin{proof}
Let $L$ be the corresponding $(2,2)$ correspondence \eqref{eq:4bar22} and $\hat{L}$ its secondary correspondence.

According to Theorem \ref{th:ksemiperiodic},
the Darboux transformation is $2$-semi-periodic if and only if $\hat{L}$ is $2$-periodic and $L$ is not.
	
The matrix corresponding to $\hat{L}$ is:
$$
M_{\hat L}
=
\begin{pmatrix}
a_{22}^2 & 2a_{02}a_{22} & a_{02}^2\\
2a_{20}a_{22} & 2a_{22}a_{00} +2a_{20}a_{02}-a_{11}^2 & 2a_{00}a_{02}\\
a_{20}^2 & 2a_{00}a_{20} & a_{00}^2
 \end{pmatrix},
$$
with
\begin{gather*}
	a_{11}=8bd,\quad
a_{00}=(a-b-d)^2-c^2,\quad
a_{20}=(a+b-d)^2-c^2,
\\
a_{02}=(a-b+d)^2-c^2,\quad
a_{22}=(a+b+d)^2-c^2.	
\end{gather*}
According to Theorem \ref{th:cayley}, the QRT transformation of $\hat L$ is of order two if and only if $\det M_{\hat L}=0$, i.e.
$$
\det M_{\hat L}
=
-4096 b^3 d^3 (a^2-b^2+c^2-d^2) (a^2c^2-b^2d^2)=0.
$$
Since $bd\neq0$, this is equivalent to
$
(a^2-b^2+c^2-d^2) (a^2c^2-b^2d^2)=0.
$
From this expression we need to factor out the condition that $L$ is $2$-periodic, which is $a^2-b^2+c^2-d^2=0$, as derived in Proposition \ref{prop:2Darboux}.
Thus the statement is proved.
\end{proof}

Next, we want to give a geometric argument for Proposition \ref{prop:2semi-periodic}, for which we will use the following:

\begin{lemma}\label{lem:angles}
The Darboux transformation applied to a given quadrilateral $ABCD$ is $2$-semi-periodic if and only if
$$
\angle DAC'=\angle BAC\ (\mod\ \pi) \quad\text{and}\quad \angle A'CD=\angle ACB\ (\mod\ \pi),
$$
where $C'$, $A'$ are symmetric to $C$, $A$ with respect to $BD$.
\end{lemma}

\begin{proof}
The Darboux transformation is $2$-semi-periodic	if and only if quadrilaterals
$(h\circ v)(ABCD)=ABC_1D_1$ and $(v\circ h)(ABCD)=ABC_2D_2$ are symmetric to each other with respect to the line $AB$, see Figure \ref{fig:2semi}.
Note that $C'=C_1$.
	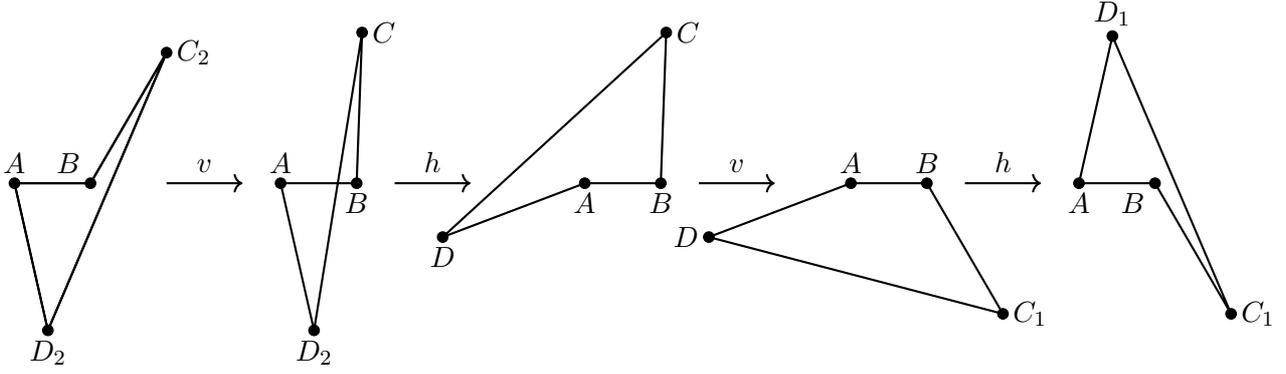
\begin{figure}[h]
	\begin{center}
		\begin{tikzpicture}[scale=1]
\draw[thick] (0,0) -- (1,0)--(2., 1.73205)--(0.439678, -1.95107)--cycle;
\begin{scope}[shift={(0,0)}]
	\draw[thick] (0,0) -- (1,0)--(2., 1.73205)--(0.439678, -1.95107)--cycle;
	\filldraw[black] (0,0) circle [radius=0.07] node [above]{$A$};
	\filldraw[black] (1,0) circle [radius=0.07] node [above left]{$B$};
	\filldraw[black] (2., 1.73205) circle [radius=0.07] node [right]{$C_2$};
	\filldraw[black] (0.439678, -1.95107) circle [radius=0.07] node [below]{$D_2$};
\end{scope}

\draw[thick,->,shift={(2,0)}] (0,0)--node[sloped, above] {$v$}(1,0);
			
\begin{scope}[shift={(3.5,0)}]
	\draw[thick] (0,0) -- (1,0)--(1.07143, 1.998722)--(0.439678, -1.95107)--cycle;
	\filldraw[black] (0,0) circle [radius=0.07] node [above]{$A$};
	\filldraw[black] (1,0) circle [radius=0.07] node [below]{$B$};
	\filldraw[black] (1.07143, 1.998722) circle [radius=0.07] node [right]{$C$};
	\filldraw[black] (0.439678, -1.95107) circle [radius=0.07] node [below]{$D_2$};
\end{scope}

\draw[thick,->,shift={(5,0)}] (0,0)--node[sloped, above]{$h$}(1,0);

\begin{scope}[shift={(7.5,0)}]
\draw[thick] (0,0) -- (1,0)--(1.07143, 1.998722)--(-1.86825, -0.713893)--cycle;
\filldraw[black] (0,0) circle [radius=0.07] node [below]{$A$};
\filldraw[black] (1,0) circle [radius=0.07] node [below]{$B$};
\filldraw[black] (1.07143, 1.998722) circle [radius=0.07] node [right]{$C$};
\filldraw[black] (-1.86825, -0.713893) circle [radius=0.07] node [below]{$D$};
\end{scope}

\draw[thick,->,shift={(9,0)}] (0,0)--node[sloped, above]{$v$}(1,0);
			
\begin{scope}[shift={(11,0)}]
	\draw[thick] (0,0) -- (1,0)--(2., -1.73205)--(-1.86825, -0.713893)--cycle;
	\filldraw[black] (0,0) circle [radius=0.07] node [above]{$A$};
	\filldraw[black] (1,0) circle [radius=0.07] node [above]{$B$};
	\filldraw[black](2., -1.73205) circle [radius=0.07] node [right]{$C_1$};
	\filldraw[black] (-1.86825, -0.713893) circle [radius=0.07] node [left]{$D$};
\end{scope}

\draw[thick,->,shift={(12.5,0)}] (0,0)--node[sloped, above]{$h$}(1,0);

\begin{scope}[shift={(14,0)}]
	\draw[thick] (0,0) -- (1,0)--(2., -1.73205)--(0.439678, 1.95107)--cycle;
	\filldraw[black] (0,0) circle [radius=0.07] node [below]{$A$};
	\filldraw[black] (1,0) circle [radius=0.07] node [below left]{$B$};
	\filldraw[black](2., -1.73205) circle [radius=0.07] node [right]{$C_1$};
	\filldraw[black] (0.439678, 1.95107) circle [radius=0.07] node [above]{$D_1$};
\end{scope}

		\end{tikzpicture}
	\end{center}\caption{A $2$-semi-periodic Darboux transformation.}\label{fig:2semi}
\end{figure}

The $2$-semi-periodicity of the Darboux transformation  implies that $s_{AB}$ maps $C_1$, $D_1$ to $C_2$, $D_2$, thus:
$$
D_1=s_{AB}(D_2)=(s_{AB}\circ s_{AC})(D)
\quad
\text{and}
\quad
D_1=s_{AC_1}(D)=(s_{AC_1}\circ s_{AD})(D).
$$
In other words, $D_1$ is obtained from $D$ as a result of the rotation with the center at $A$ by the angle $2\angle CAB$, but also as a result of the rotation with the same center by the angle $2\angle DAC_1$.
Thus, those two rotations are in fact the same map, so the two oriented angles $2\angle CAB$ and $2\angle DAC_1$ must be equal modulo $2\pi$.
By symmetry, the same holds for $2\angle ACB$ and $2\angle DCA'$.
\end{proof}

\begin{example}[Second proof for Proposition \ref{prop:2semi-periodic}.]\label{ex:second-proof}
Here, we are going to provide a planimetric proof that $ac=bd$ is equivalent to the $2$-semi-periodicity of the link $(a, b, c, d)$.

First, suppose that $ac=bd$.
For the link $(a,b,c,d)$, choose a cyclic polygonal configuration $T_1=ABCD$.
Recall that a cyclic quadrilateral that satisfies $ac=bd$ is called \emph{harmonic quadrilateral} and that it is characterized by the property that each diagonal is a \emph{symmedian} of the triangles formed by dividing the quadrilateral by the other diagonal.
Recall also that a symmedian of a triangle is the line symmetric to its median with respect to the bisector of the angle with the same vertex as the median.

Now, denote by $Q$ the midpoint of the diagonal $BD$, see Figure \ref{fig:tetivni}.
	\begin{figure}[h]
	\begin{center}
		\begin{tikzpicture}[scale=0.8]
\coordinate (A) at (0,0);
\coordinate (B) at (2,0);
\coordinate (C) at (3.30738, 4.82605);
\coordinate (D) at (-1.7669, 1.62422);
\coordinate (O) at (1, 2.86101);
\coordinate (Q) at (0.116552, 0.812108);
\coordinate (C2) at (0.997377, 6.94951);
\coordinate (C1) at (-0.611897, -4.26357);

\draw[very thick] (A) -- (B)--(C)--(D)--cycle;

\draw (B)--(D);
\draw (A)--(C);
\draw (C)--(Q);
\draw(C1)--(C2)--(D)--cycle;

\filldraw[black] (A) circle [radius=0.07] node [below left]{$A$};
\filldraw[black] (B) circle [radius=0.07] node [below]{$B$};
\filldraw[black] (C) circle [radius=0.07] node [right]{$C$};
\filldraw[black] (D) circle [radius=0.07] node [left]{$D$};
\filldraw[black] (Q) circle [radius=0.07] node [above left]{$Q$};
\filldraw[black] (C1) circle [radius=0.07] node [left]{$C_1$};
\filldraw[black] (C2) circle [radius=0.07] node [left]{$C_2$};

\draw (O) circle [radius=3.03074];

\tkzMarkAngle[size=0.5cm,mark=||](B,A,C);
\tkzMarkAngle[size=0.5cm,mark=||](Q,A,D);

\tkzMarkAngle[size=1cm,mark=|](A,C,B);
\tkzMarkAngle[size=1cm,mark=|](D,C,Q);
\tkzMarkAngle[size=0.9cm,mark=|](D,C2,Q);
\tkzMarkAngle[size=0.9cm,mark=|](Q,C1,D);
		\end{tikzpicture}
	\end{center}\caption{A harmonic quadrilateral $ABCD$ is cyclic and the products of the pairs of opposite sides are equal. }\label{fig:tetivni}
\end{figure}
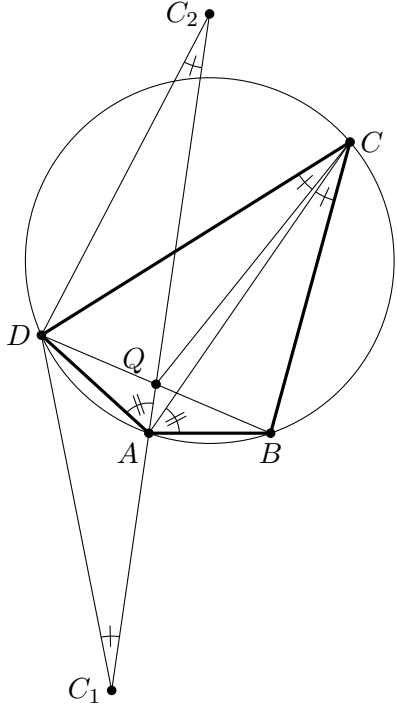
Since $AC$ is the symmedian of the triangles $ABD$ and $BCD$, we have:
$\angle DAQ\cong\angle BAC$ and $\angle DCQ \cong \angle ACB$.

In order to show that the quadrilateral $ABCD$ belongs to a $2$-semi-periodic link, according to Lemma \ref{lem:angles}, we need to prove that $Q$, $A$, and $C_1=s_{BD}(C)$ are collinear.

By the property of axial symmetry, we know that
$DC_1=DC=c$ and $\angle DC_1Q\cong\angle DCQ$.

On the ray $AQ$, we construct the point $C_2$, such that $DC_2=c$.
Observe the similarity of triangles
$
\triangle C_2DA\sim \triangle CBA,
$
which follows from
$$
\angle DAQ\cong\angle BAC \quad \text{and}\quad \frac {c}{d}=\frac{b}{a}.
$$
Thus,
\begin{align*}
\angle AC_2D\cong\angle ACB
\cong\angle QCD
\cong\angle QC_1D
\cong\angle DC_2C,
\end{align*}
which shows that
$
A\in C_1C_2
$
and
$
Q\in C_1C_2.
$
Thus, $Q$, $A$, and $C_1$ are collinear.
We also get that $Q$, $A'=s_{BD}(A)$, and $C$ are collinear.
Now, from $\angle DCQ \cong \angle ACB$, we have
$
\angle DCA' \cong \angle ACB,
$
which, according to Lemma \ref{lem:angles} implies that the Darboux transformation applied to the cyclic quadrangle $ABCD$ is $2$-semi-periodic.
By the poristic nature of $2$-semi-periodicity, the same will hold for all configurations of the same $4$-link.

\smallskip

\noindent
{\bf Converse: from $2$-semi-periodicity to $ac=bd$.}
Again, we choose a cyclic configuration $ABCD$ for the $4$-link $(a,b,c,d)$.
According to Lemma \ref{lem:angles}, we assume
$
\angle DAC_1=\pi - \angle BAC
$
and
$\angle A'CD\cong\angle ACB$.
Denote by $Q$ the intersection point of three lines $AC_1$, $A'C$ and $BD$.
On the ray $AQ$, we construct the point $C_2$, such that $DC_2\cong DC=c$, see Figure \ref{fig:tetivni-obrat}.
	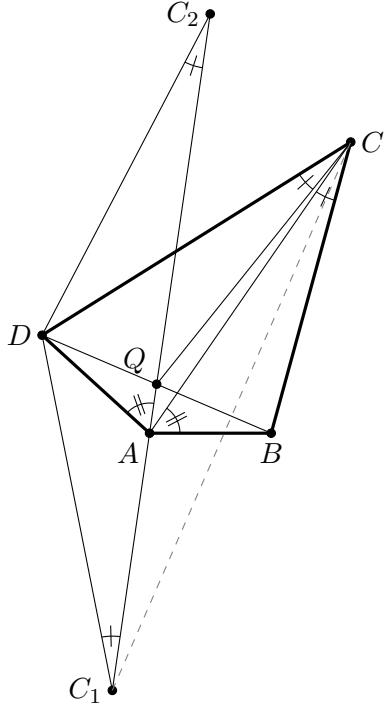
\begin{figure}[h]
	\begin{center}
		\begin{tikzpicture}[scale=0.8]
			\coordinate (A) at (0,0);
			\coordinate (B) at (2,0);
			\coordinate (C) at (3.30738, 4.82605);
			\coordinate (D) at (-1.7669, 1.62422);
%			\coordinate (O) at (1, 2.86101);
			\coordinate (Q) at (0.116552, 0.812108);
			\coordinate (C2) at (0.997377, 6.94951);
			\coordinate (C1) at (-0.611897, -4.26357);

			\draw[very thick] (A) -- (B)--(C)--(D)--cycle;
			
			\draw (B)--(D);
			\draw (A)--(C);
			\draw (C)--(Q);
			\draw(C1)--(C2)--(D)--cycle;
			
			\filldraw[black] (A) circle [radius=0.07] node [below left]{$A$};
			\filldraw[black] (B) circle [radius=0.07] node [below]{$B$};
			\filldraw[black] (C) circle [radius=0.07] node [right]{$C$};
			\filldraw[black] (D) circle [radius=0.07] node [left]{$D$};
			\filldraw[black] (Q) circle [radius=0.07] node [above left]{$Q$};
			\filldraw[black] (C1) circle [radius=0.07] node [left]{$C_1$};
			\filldraw[black] (C2) circle [radius=0.07] node [left]{$C_2$};
			
			\draw[dashed,gray](C1)--(C);

%			\draw (O) circle [radius=3.03074];
			
			\tkzMarkAngle[size=0.5cm,mark=||](B,A,C);
			\tkzMarkAngle[size=0.5cm,mark=||](Q,A,D);
			
			\tkzMarkAngle[size=1cm,mark=|](A,C,B);
			\tkzMarkAngle[size=1cm,mark=|](D,C,Q);
			\tkzMarkAngle[size=0.9cm,mark=|](D,C2,Q);
			\tkzMarkAngle[size=0.9cm,mark=|](Q,C1,D);
		\end{tikzpicture}
	\end{center}\caption{The Darboux transformation of quadrilateral $ABCD$ is $2$-semi-periodic. }\label{fig:tetivni-obrat}
\end{figure}

Observe that the similarity of triangles
$
\triangle C_2DA\sim \triangle CBA,
$
follows from
$
\angle DAQ\cong\angle BAC$
and
$\angle DC_2Q\cong\angle DC_1Q\cong\angle QCD\cong\angle ACB.
$
Thus:
$
 \frac{DC_2}{DA}=\frac{CB}{AB},
$
which gives $ac=bd$ and completes the proof.
\end{example}

\begin{proposition} A $4$-bar link $(a,b,c,d)$ is $3$-semi-periodic if and only if
\begin{equation}\label{eq:3semi}
a^2 c^2 \left(a^2-b^2+c^2-d^2\right)^2=
(a^2 c^2-b^2 d^2)^2.
\end{equation}
\end{proposition}
\begin{proof}
Denote by $L$ the $(2,2)$ correspondence joined to the link, and by $\hat L$ its secondary correspondence.
Theorem \ref{th:ksemiperiodic} says that the link is $3$-semi-periodic if and only if $\hat L$ is $3$-periodic but $L$ is not.

The cubic curve $\hat{\Gamma}$ corresponding to $\hat L$ is given in Lemma \ref{lem:secondlinkage}, while the value of the coordinate $X_0$ from \eqref{eq:XY} is:
$$
X_0=\frac{64}{3} b^2 d^2 \left(a^4-2 a^2 \left(b^2-5 c^2+d^2\right)+b^4-2 b^2 \left(c^2-5 d^2\right)+\left(c^2-d^2\right)^2\right).
$$
Then, the condition for the $3$-periodicity of $\hat L$ is equivalent to $B_2=0$, where
$$
\sqrt{4 x^3 - D_{\hat L} x + E_{\hat L}}
=
B_0+B_1(x-X_0)+B_2(x-X_0)^2+B_3(x-X_0)^3+\dots.
$$
The direct calculation gives the following:
$$
B_2=
\frac{\left(b^2 d^2 \left(a^2-b^2+c^2-d^2\right)^2-(a^2 c^2-b^2 d^2)^2\right) \left(a^2 c^2 \left(a^2-b^2+c^2-d^2\right)^2-(a^2 c^2-b^2 d^2)^2\right)}
{8 b d (b d-a c)^3 (a c+b d)^3 \left(a^2-b^2+c^2-d^2\right)^3}.
$$
Now, factoring out the condition for the $3$-periodicity of $L$, which was obtained in Proposition \ref{prop:3Darboux}, will conclude the proof.
\end{proof}

\begin{remark}\label{rem:3-periodic&semi}
We see that the condition for $3$-periodicity
$$
b^2 d^2 \left(a^2-b^2+c^2-d^2\right)^2=(a^2 c^2-b^2 d^2)^2,
$$
transforms to the condition for $3$-semi periodicity, given above in \eqref{eq:3semi}, with a cyclic transformation of the $4$-bar link $(a, b, c, d)$ to $(b, c, d, a)$.
This shows that the order of the Darboux transformation is not invariant with respect to this cyclic transformation of $4$-bar links.
For example, compare the Darboux transformation of two congruent quadrangles: in Figure \ref{fig:6Darbouxb}, it is $3$-semi-periodic, and in Figure \ref{fig:6Darbouxb-3semi} it is $3$-periodic.
\end{remark}
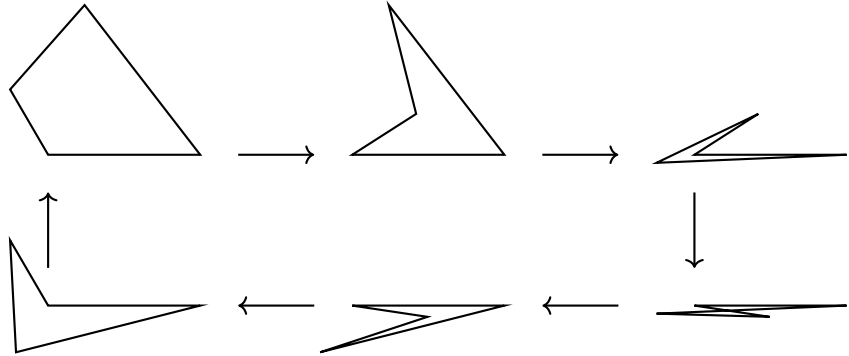
\begin{figure}[h]
	\begin{center}
		\begin{tikzpicture}[scale=1]
			\draw[thick] (0,0) -- (2,0)--(0.479833, 1.98471)--(-0.5, 0.866025)--cycle;
			
			\draw[thick,->,shift={(2.5,0)}] (0,0)--(1,0);
			
			\draw[thick,shift={(4,0)}] (0,0) -- (2,0)--(0.479833, 1.98471)--(0.840402, 0.541964)--cycle;

			\draw[thick,->,shift={(6.5,0)}] (0,0)--(1,0);
			
			\draw[thick,shift={(8.5,0)}] (0,0) -- (2,0)--(-0.497715, -0.106868)--(0.840402, 0.541964)--cycle;
			
			\draw[thick,->,shift={(8.5,-0.5)}] (0,0)--(0,-1);

\draw[thick,shift={(8.5,-2)}] (0,0) -- (2,0)--(-0.497715, -0.106868)--(0.988807, -0.149202)--cycle;

\draw[thick,<-,shift={(6.5,-2)}] (0,0)--(1,0);

\draw[thick,shift={(4,-2)}] (0,0) -- (2,0)--(-0.422141, -0.619059)--(0.988807, -0.149202)--cycle;

\draw[thick,<-,shift={(2.5,-2)}] (0,0)--(1,0);

\draw[thick,shift={(0,-2)}] (0,0) -- (2,0)--(-0.422141, -0.619059)--(-0.5, 0.866025)--cycle;

			\draw[thick,<-,shift={(0,-0.5)}] (0,0)--(0,-1);
			
		\end{tikzpicture}
	\end{center}\caption{A $3$-periodic Darboux transformation, with  $a=2$, $b=5/2$, $c=\sqrt{115/13}/2$, $d=1$.}\label{fig:6Darbouxb-3semi}
\end{figure}

\begin{remark}\label{rem:invariant}
Assume for a moment a modification of the definition of polygonal configurations of $4$-bar links in a way that we identify those obtained from each other by an isometric transformation of the Euclidean plane, oriented or nonoriented. Then we see that the secondary biquadratic and the secondary cubic provide necessary and sufficient conditions for $n$-periodicity in this new sense.
 This is exactly the problem that was solved by Izmestiev in \cite{Izm}*{Theorem 4}, where using a very clever approach based on \cite{GrifHar1978}, the following cubic was constructed for $4$-bar links:
$$
\Gamma'\ :\
y^2
=
\left(
x-(ac+bd)^2
\right)
\left(
x-(ac-bd)^2
\right)
\left(
x-\frac{(a^2-b^2+c^2-d^2)^2}4
\right).
$$
That cubic turns out to be  isomorphic to our secondary cubic $\hat \Gamma$ from \eqref{eq:seccubic}, via the isomorphism:
$$
\Gamma'\to\hat{\Gamma}\ :\
(x,y)
\mapsto
(\lambda,\mu)
=
\left(
(16bd)^2x+X_0,
(16bd)^3y
\right),
$$
where $X_0$ is calculated as in the second proof of Proposition \ref{prop:3Darboux}, see equation \eqref{eq:X0semi}.

Moreover, it was observed in \cite{Izm} that $4$-bar links can be $n$-periodic with respect to one side and $2n$-periodic with respect to the neighboring side, and explicit conditions were derived for $n=3$.
The examples we present in Figures \ref{fig:6Darbouxb} and \ref{fig:6Darbouxb-3semi} exactly fit into this framework.
\end{remark}

\subsection{Singular case: the sum of two sides equals the sum of the remaining two}

According to Remark \ref{rem:singular}, a $4$-bar link which satisfies the triangle inequality will generate a singular biquadratic curve if and only if the link has a pair of sides whose sum equals the semi-perimeter of the link.
Thus, we will have two cases:
\begin{itemize}
\item the sums of two pairs of opposite sides are equal; and
\item the sum of a pair of consecutive sides equals the sum of the remaining pair of sides.
\end{itemize}

First, consider the links $(a,b,c,d)$ satisfying $a+c=b+d$.
The quadrilaterals satisfying that relation are sometimes called \emph{the Pitot quadrilaterals}, see e.g.~\cite{DrKal}.

\begin{proposition}
The Darboux transformation applied to a Pitot quadrilateral which is not a kite is not periodic.
\end{proposition}
\begin{proof}
In that case, we have a singular $(2,2)$ correspondence:
\begin{equation}\label{eq:4bar22Sin}
		a(b+d)x^2y^2+b(a-d)x^2 +d(a-b)y^2%\\
		+2bdxy=0.
\end{equation}
We notice that the origin is an ordinary double point, unless $(a-d)(a-b)=bd$, when it is a cusp.
Thus,
Theorem \ref{th:doublepoint} implies non-periodicity for $(a-d)(a-b)=bd$, while, for $(a-d)(a-b)\neq bd$, it yields periodicity if there are natural numbers $m$, $n$ such that:
\begin{equation}\label{eq:nsingular4bar}
\frac{bd}{(a-d)(a-b)}
=
\cos^2\left(\dfrac{\pi m}{n}\right).
\end{equation}
Thus,
$$
0\le \frac{bd}{(a-d)(a-b)}\le 1.
$$
From $bd>0$, it follows that $(a-d)(a-b)>0$, and thus $bd<(a-d)(a-b)$, which is equivalent to $0<a(a-(b+d))$.
This leads to the contradiction with $a>0$ and $c=b+d-a>0$.
\end{proof}

\begin{remark}[Kites]
Kites appear as configurations of the links of the form $(a,a,c,c)$, $(a,c,c,a)$.
If $a\neq c$, the biquadratic curve \eqref{eq:4bar22Sin} is the union of a conic and a line.
If $a=c$, the configurations are rhombi and the biquadratic curve is the union of three lines.
We note that the kites can be treated as a limit case of a family of $2$-periodic cases, since kites have orthogonal diagonals.
\end{remark}

The case $a+b=c+d$ can be treated similarly.

\begin{proposition}
There are no $4$-bar links with $a+b=c+d$, that generate a periodic Darboux transformation.
\end{proposition}
\begin{proof}
The $(2,2)$ correspondence is
 $$
 a(d-b)X^2y^2+b(d-a)X^2+d(a+b)y^2+2bdXy=0,
 $$
 where $X=1/x$.
If the Darboux transformation is periodic, then there are natural numbers $m$, $n$ such that:
 $$
\frac{bd}{(d-a)(a+b)}
=
\cos^2\left(\dfrac{\pi m}{n}\right).
 $$
From $0<\dfrac{bd}{(d-a)(a+b)}<1$,
we get $d>a$ and then also $0<-ac$, which leads to contradiction.
\end{proof}

\subsection{From $4$-bar links back to random walks}\label{sec:fbrw}

In this section, we establish a new two-way relationship between the $(2,2)$ correspondences of random walks and those of the $4$-bar links. We recall

\begin{definition}\label{def:diagonal} A random walk is called \emph{diagonal} if  $p_{j0}=p_{0j}=0$ for $j\in\{-1,1\}$.
\end{definition}

Let us start with a more general statement, where we do not assume that $a$, $b$, $c$, $d$ are positive or that $0\le p_{ij}\le 1$.

\begin{proposition}\label{prop:cor}
Let $a$, $b$, $c$, $d$ be arbitrary numbers.
If we set:
\begin{align*}
p_{00}&=\frac{8bd+\lambda}{\lambda},\\
p_{j0}&=p_{0j}=0,\quad\text{for } j\in\{-1,1\},\\
p_{jk}&=\frac{(a-jb-kd)^2-c^2}{\lambda},\quad\text{for }  k,j\in\{-1,1\},
\end{align*}
where $\lambda\neq0$ is any parameter,
then the matrices \eqref{eq:P} and \eqref{eq:ML} define the same $(2,2)$ correspondence.

Conversely, suppose that $p_{jk}$, $j,k\in\{-1,0,1\}$, are arbitrary given numbers, with $p_{j0}=p_{0j}=0$ for $j\in\{-1,1\}$
and satisfying $p_{11} - p_{-1,-1}\neq\pm( p_{1,-1} - p_{-1,1})$, $p_{11} - p_{1,-1}\neq\pm( p_{-1,1} - p_{-1,-1})$.
Set:
\begin{equation}\label{eq:abcd}
\begin{aligned}
a&=-Q_2\sqrt{\frac{\lambda(p_{00}-1)Q_1}{8}}
,\\
b&=\sqrt{\frac{\lambda(p_{00}-1)Q_1}{8}}
,\\
c&=
\sqrt{
\lambda
\left(
\frac{p_{00}-1}{8Q_1}
\left(Q_1Q_2-Q_1-1\right)^2
-p_{-1,-1}
\right)
}
,\\
d&=   \sqrt{\frac{\lambda(p_{00}-1)}{8Q_1}},
\end{aligned}
\end{equation}
where
$$
Q_1
=
\dfrac{p_{11} + p_{1,-1} - p_{-1,1} - p_{-1,-1}}{p_{11} - p_{1,-1} + p_{-1,1} - p_{-1,-1}},
\qquad
Q_2
=
	\dfrac{-p_{11} + p_{1,-1} - p_{-1,1} + p_{-1,-1}}{p_{11} - p_{1,-1} - p_{-1,1} + p_{-1,-1}},
$$
and $\lambda$ is an arbitrary non-zero constant.
Then the matrices \eqref{eq:P} and \eqref{eq:ML}define the same $(2,2)$ correspondence.
\end{proposition}
\begin{proof}
The first part follows from a straightforward comparison of the corresponding entries in the matrices \eqref{eq:P} and \eqref{eq:ML}.

For the opposite direction, we first observe that
$$
\lambda\begin{pmatrix}
		p_{11} & p_{10} &  p_{1,-1}\\
		p_{01} &  p_{00}-1 &  p_{0,-1}\\
		p_{-1,1} &  p_{-1,0} &  p_{-1,-1}
	\end{pmatrix}
=
\begin{pmatrix}
	(a+b+d)^2-c^2 & 0 & (a+b-d)^2-c^2 \\
0 & 8 b d & 0 \\
(a-b+d)^2-c^2 & 0 & (a-b-d)^2-c^2
\end{pmatrix}.
$$
implies
\begin{align*}
\lambda(p_{1,-1}-p_{-1,1})&=4a(b-d),\\
\lambda(p_{-1,1}-p_{-1,-1})&=4d(a-b),\\
\lambda(p_{11}-p_{1,-1})&=4d(a+b),\\
\lambda(p_{11}-p_{-1,-1})&=4a(b+d).
\end{align*}
Thus,
$$
Q_1
=
\dfrac{p_{11} + p_{1,-1} - p_{-1,1} - p_{-1,-1}}{p_{11} - p_{1,-1} + p_{-1,1} - p_{-1,-1}}
=
\frac{b}{d},
\quad
Q_2
=
\dfrac{-p_{11} + p_{1,-1} - p_{-1,1} + p_{-1,-1}}{p_{11} - p_{1,-1} - p_{-1,1} + p_{-1,-1}}
=
-\frac{a}b,
$$
i.e.~$b=Q_1d$ and
$a=-Q_2b$.
From there and $8bd=\lambda(p_{00}-1)$, we get
$
8b^2=\lambda Q_1(p_{00}-1),
$
thus:
$$
b=\sqrt{\frac{\lambda Q_1(p_{00}-1)}{8}},
$$
and we also get
$$
d=  \sqrt{\frac{\lambda(p_{00}-1)}{8Q_1}},
\quad
a=-Q_2\sqrt{\frac{\lambda Q_1(p_{00}-1)}{8}}.
$$
Finally, substituting the expressions for $a$, $b$, $d$ into
$
c^2=(a+b+d)^2-\lambda p_{-1,-1},
$
 we calculate $c$.
\end{proof}

The following corollary is useful for obtaining examples of random walks and $4$-bar links that determine that same $(2,2)$ correspondence.
\begin{corollary}
Assume $0\le p_{00}< 1$.
Then, if $a$, $b$, $c$, $d$ are positive numbers, we must have $Q_1>0$, $Q_2<0$ and $\lambda<0$.

Moreover, if $Q_1>0$, $Q_2<0$, $\lambda<0$, $p_{-1,-1}\ge0$, $0\le p_{00}<1$, then $a$, $b$, $c$, $d$ from \eqref{eq:abcd} are positive.
\end{corollary}
\begin{example}
Consider the $4$-bar link with $a=3/2$, $b=1$, $c=\sqrt{13}/2$, $d=1$.
Then Proposition \ref{prop:cor}, with $\lambda=-10$, gives the transition probabilities of a diagonal random walk (see Definition \ref{def:diagonal}),
$p_{-1,1}=p_{1, -1}=1/4$,
$p_{-1, -1}=3/10$,
$p_{11}=0$,
$p_{00}=1/5$.
\end{example}

\subsection*{Acknowledgment}
This research  was supported
by
the Serbian Ministry of Science, Technological Development and Innovation and the Science Fund of Serbia grant IntegraRS, and the Simons Foundation grant no.~854861.
We thank Ivan Izmestiev for kindly indicating \cite{Izm} to us.

%\appendix

\end{document}